\documentclass[11pt]{amsart}

\usepackage{evearticle}

\DeclareMathOperator\inradius{inradius}
\DeclareMathOperator\outradius{outradius}

\begin{document}

\title[The conformal dimension of the Brownian sphere is two]{The conformal dimension of the Brownian sphere is two}
\author[Jason Miller and Yi Tian]{Jason Miller and Yi Tian}
\date{\today}

\begin{abstract}
The conformal dimension of a metric space $(X, d)$ is equal to the infimum of the Hausdorff dimensions among all metric spaces quasisymmetric to $(X, d)$.  It is an important quasisymmetric invariant which lies non-strictly between the topological and Hausdorff dimensions of $(X, d)$.  We consider the conformal dimension of the Brownian sphere (a.k.a.~the Brownian map), whose law can be thought of as the uniform measure on metric measure spaces homeomorphic to the standard sphere $\BS^2$ with unit area.  Since the Hausdorff dimension of the Brownian sphere is $4$, its conformal dimension lies in $[2, 4]$.  Our main result is that its conformal dimension is equal to $2$, its topological dimension.
\end{abstract}

\maketitle

\setcounter{tocdepth}{1}
\tableofcontents

\setlength{\parindent}{0pt}
\setlength{\parskip}{0.5\baselineskip plus 1pt minus 1pt}

\section{Introduction}

The main object of study in this paper is the \emph{Brownian sphere} (a.k.a.~the \emph{Brownian map}), whose law is the ``uniform measure'' on metric measure spaces homeomorphic to the two-dimensional sphere $\BS^2$ with unit area.  This interpretation arises as it was shown in independent works of Le Gall \cite{MR3112934} and Miermont \cite{MR3070569} that it is the Gromov--Hausdorff--Prokhorov scaling limit of the natural discrete uniform measures on surfaces homeomorphic to $\BS^2$ coming from the theory of random planar maps.  Recall that a \emph{planar map} is a proper embedding of a connected planar graph into the two-dimensional sphere, considered up to orientation-preserving homeomorphisms. A planar map is called a \emph{$p$-angulation} if every face is incident to exactly $p$ edges (i.e., has degree $p$).  Planar maps can be thought of as surfaces by associating with each face a copy of a regular Euclidean polygon with the same number of sides as the degree of the face.  The work \cite{MR3070569} proved the convergence of random quadrangulations to the Brownian sphere while \cite{MR3112934} proved the convergence for $p$-angulations with $p = 3$ or $p \geq 4$ even.  A number of subsequent works showed that other families of random planar maps also converge to the Brownian sphere; see \cite{MR3256874,MR3498001,MR3706731,MR4315765}. The Brownian sphere is also the scaling limit of random hyperbolic surfaces with genus zero and many cusps \cite{2508.18792}. Furthermore, \cite{MR2336042,MR2438999} proved that the Brownian sphere is almost surely homeomorphic to $\BS^2$ (non-trivial since being homeomorphic to $\BS^2$ is not preserved by Gromov--Hausdorff--Prokhorov limits; see \cite{MR2399286} for another proof) and has Hausdorff dimension $4$.

The Brownian sphere can equivalently be described as the $\sqrt{8/3}$-\emph{Liouville quantum gravity (LQG)} sphere \cite{MR3572845,MR4050102,MR4225028,MR4348679,MR4242633}.  As a consequence, these works show that if $(\SCS, D_\SCS, \SM_\SCS)$ is a Brownian sphere, there is a natural homeomorphism $\SCS \to \BS^2$ that is almost surely determined by $(\SCS, D_\SCS, \SM_\SCS)$ which is well-defined up to post-composition with conformal automorphisms of $\BS^2$. This gives the Brownian sphere a natural conformal structure, allowing it to be viewed (in some sense) as a random two-dimensional Riemannian manifold.

Let $(X, d_X)$ and $(Y, d_Y)$ be metric spaces. A homeomorphism $f \colon (X, d_X) \to (Y, d_Y)$ is \emph{quasisymmetric} if there exists an increasing homeomorphism $\eta \colon \BR_{\ge 0} \to \BR_{\ge 0}$ such that 
\begin{equation*}
    \frac{d_Y(f(x), f(z))}{d_Y(f(y), f(z))} \le \eta\!\left(\frac{d_X(x, z)}{d_X(y, z)}\right), \quad \forall \text{ distinct } x, y, z \in X. 
\end{equation*}
Two metric spaces $(X, d_X)$ and $(Y, d_Y)$ are \emph{quasisymmetrically equivalent} if there exists a quasisymmetric mapping between them.

Quasisymmetric mappings were introduced by \cite{MR595180} as a generalization of quasiconformal mappings. In turn, quasiconformal mappings --- which extend the classical theory of conformal mappings --- were first considered in \cite{MR4321180,MR4321181} and formally named by \cite{MR1555403}. For a comprehensive introduction, we refer the reader to \cite{MR188434,MR753330,MR1323982,MR1800917,MR2241787,MR2268390}. Both quasisymmetric and quasiconformal mappings have many applications across mathematics which we will not attempt to summarize here.

The \emph{conformal dimension}, introduced by Pansu \cite{MR1024425}, is an important quasisymmetric invariant. The conformal dimension of a metric space is the infimum of the Hausdorff dimensions of all metric spaces quasisymmetrically equivalent to it. By definition, it is bounded below by the topological dimension and above by the Hausdorff dimension.  A metric space is \emph{minimal} for conformal dimension if its conformal dimension equals its Hausdorff dimension.  In particular, the conformal dimension of the Brownian sphere lies in the interval $[2,4]$.  Our main result is that its conformal dimension matches its topological dimension.

\begin{theorem}
\label{thm:main}
   Almost surely, the Brownian sphere has conformal dimension $2$.
\end{theorem}

Many naturally occurring fractals arise from random processes. However, to the best of our knowledge, prior to the present work, the graph of a one-dimensional Brownian motion was the only random fractal for which the conformal dimension had been explicitly determined \cite{MR4912977}.  In contrast to \cite{MR4912977}, we find that the conformal dimension of the Brownian sphere is equal to its topological dimension $2$ which is strictly smaller than its Hausdorff dimension $4$.  That is, the Brownian sphere is \emph{not} minimal.

We remark that there is a variant of the conformal dimension commonly considered called the  Ahlfors regular conformal dimension, where one requires that the target space is Ahlfors regular.  It was shown in \cite{MR4242630} that the Assouad dimension of the Brownian sphere is infinite, hence it cannot be embedded quasisymmetrically into $\BR^n$, or any doubling space.  In particular, its canonical embedding into $\BS^2$ is not quasisymmetric and its Ahlfors regular conformal dimension is infinite.

In general, a metric space has minimal conformal dimension if it contains a ``sufficiently rich family of rectifiable curves'' (see \cite[Proposition~4.1.8]{MR2662522}).  It was shown in \cite{GeoBMap} (and later \cite{MR3706747,BM-Geodesics}) that geodesics in the Brownian sphere have the \emph{confluence property}: Two geodesics with nearby starting and target points will inevitably merge and share a common segment outside a small neighborhood of their endpoints. It was further shown in \cite{BM-Geodesics} that the \emph{geodesic frame}, the closure of the union of all geodesics minus their endpoints, of the Brownian sphere is equal to $1$, the dimension of a single geodesic.  These results are strongly suggestive that the Brownian sphere lacks a sufficiently rich family of rectifiable curves to be minimal.

For curves and surfaces which are rough and fractal in a controlled manner, one can appeal to the literature on quasisymmetric uniformization which serves to generalize the classical uniformization.  Recall that the latter states that every simply connected two-dimensional Riemannian manifold is conformally equivalent to the upper half-plane, the complex plane, or the Riemann sphere. The quasisymmetric analogue asks: When is a metric space $(X, d_X)$ homeomorphic to a ``standard'' space $(Y, d_Y)$ quasisymmetrically equivalent to it? For $n = 1$, \cite{MR595180} established that a metric space homeomorphic to $\BS^1$ is quasisymmetrically equivalent to $\BS^1$ (a \emph{quasicircle}) if and only if it is doubling and linearly locally connected (LLC). For $n = 2$, \cite{MR1930885} proved that if a metric space homeomorphic to $\BS^2$ is Ahlfors $2$-regular and LLC, then it is quasisymmetrically equivalent to $\BS^2$ (see also \cite{MR3608292,MR4073230,MR4956983,MR4608329}).  The uniformization problem for dimensions $n \ge 3$ currently remains open.  Quasicircles in particular  have conformal dimension $1$ and Ahlfors $2$-regular LLC surfaces homeomorphic to $\BS^2$ have conformal dimension $2$.  The Brownian sphere is not Ahlfors regular \cite{MR4242630} and does not satisfy the LLC property, so one cannot use \cite{MR1930885} to compute its conformal dimension.

The conformal dimension is important in the theory of Gromov hyperbolic spaces (see \Cref{section:background-Gromov}) and this was Pansu's original motivation for introducing it \cite{MR1024425}. Specifically, quasi-isometries between Gromov hyperbolic spaces induce quasisymmetric mappings on their Gromov boundaries.  Thus the conformal dimension of the Gromov boundary is a quasi-isometric invariant for a Gromov hyperbolic space.  For boundaries homeomorphic to $\BS^2$, this leads to Cannon's conjecture \cite{MR1301392}: Let $G$ be a Gromov hyperbolic group whose Gromov boundary $\partial_\infty G$ is homeomorphic to $\BS^2$. Then the following equivalent statements hold:
\begin{enumerate}
    \item\label{it:Cannon-0} There exists an isometric, properly discontinuous, and cocompact action of $G$ on the hyperbolic space $\BH^3$. 
    \item\label{it:Cannon-1} The Gromov boundary $\partial_\infty G$ is quasisymmetrically equivalent to $\BS^2$.
\end{enumerate}
Remarkably, Cannon's conjecture is equivalent to stating that the Ahlfors regular conformal dimension of $\partial_\infty G$ is attained as a minimum \cite{MR2116315}.

There are many works in the literature which study the conformal dimension of fractals, including:
\begin{itemize}
    \item The conformal dimension of any metric space with a Hausdorff dimension strictly less than $1$ is $0$ \cite{MR2239342}. 
    \item For every $d \ge 1$, there exists a metric space with Hausdorff dimension $d$ that is minimal for conformal dimension \cite{MR1833245}. 
    \item The conformal dimension of the Sierpi\'nski gasket is exactly $1$ \cite{MR2268118}.  More generally, the conformal dimension of any sufficiently nice metric space with enough local cut points is $1$ \cite{MR3213834}.
    \item The exact conformal dimension of the Sierpi\'nski carpet remains unknown; however, it lies strictly between $1$ and its Hausdorff dimension \cite{MR1676353,MR2135168,MR3276005,MR4056527}. 
    \item Every Bedford--McMullen carpet with uniform fibers is minimal for conformal dimension \cite{MR2846307,MR4912977}. 
    \item Fractal percolation is almost surely not minimal for conformal dimension, although explicit values for this dimension have not yet been determined \cite{MR4259151}. 
    \item The graph of a one-dimensional Brownian motion is almost surely minimal for conformal dimension \cite{MR4912977}. 
\end{itemize}

In view of Theorem~\ref{thm:main} and the equivalence of the $\sqrt{8/3}$-LQG sphere and the Brownian sphere \cite{MR3572845,MR4050102,MR4225028,MR4348679,MR4242633}, it is also natural to consider the conformal dimension of $\gamma$-LQG surfaces for general values of $\gamma \in (0,2)$.  In this generality, the associated metric was constructed in \cite{TightLFPP,WeakLQGMet,ConfLQG,LocMetGFF,ExUniLQG}. We conjecture that the conformal dimension of the $\gamma$-LQG sphere is almost surely $2$ for all $\gamma \in (0, 2)$.

\subsection{Outline of the proof}\label{subsection:outline}

Our argument is motivated by the work of \cite{ConformalGauge}, which shows that the Ahlfors regular conformal dimension of a compact, doubling metric space is equal to the critical exponent associated with its combinatorial modulus.  (See also the expository work \cite{CP-ModARCD} on the results of \cite{ConformalGauge}.)

The approach in \cite{ConformalGauge} relies heavily on the technique of hyperbolic fillings (see \Cref{section:hyperbolic-filling} for a more detailed review). Heuristically, for a compact metric space $(X, D)$ with $\diam(X; D) < 1$, the hyperbolic filling is constructed as follows. Fix a sufficiently small parameter $\alpha \in (0, 1)$. Let $A_0 \subset A_1 \subset A_2 \subset \cdots$ be a sequence of finite subsets of $X$ such that each $A_n$ is a maximal $\alpha^n$-separated subset. The hyperbolic filling is then defined as a graph whose vertices are the metric balls $B_{\alpha^n}(x; D)$ for $n \ge 0$ and $x \in A_n$. Two vertices $B_{\alpha^m}(x; D)$ and $B_{\alpha^n}(y; D)$ are connected by an edge if either $m = n$ and the dilated balls $B_{4\alpha^m}(x; D)$ and $B_{4\alpha^n}(y; D)$ intersect, or $\lvert m - n \rvert = 1$ and the balls $B_{\alpha^m}(x; D)$ and $B_{\alpha^n}(y; D)$ intersect. The resulting metric graph forms a Gromov hyperbolic space whose Gromov boundary is quasisymmetrically equivalent to $(X, D)$. 

A \emph{weight function} is an assignment $\sigma$ from the vertices of the hyperbolic filling to the non-negative real numbers. Such a weight function is considered \emph{admissible} if it satisfies the following condition: Suppose $B_{\alpha^{n - 1}}(y; D)$ is a vertex, and there exists a path of vertices $B_{\alpha^n}(x_0; D) \sim B_{\alpha^n}(x_1; D) \sim \cdots \sim B_{\alpha^n}(x_N; D)$ such that $B_{\alpha^{n - 1}}(y; D) \cap B_{4\alpha^n}(x_0; D) \neq \emptyset$ and $(\BC \setminus B_{2\alpha^{n - 1}}(y; D)) \cap B_{4\alpha^n}(x_N; D) \neq \emptyset$. Then we must have $\sum_{j = 0}^N \sigma(B_{\alpha^n}(x_j; D)) \ge 1$. Following \cite{ConformalGauge,CP-ModARCD}, for each admissible weight function, we can construct a metric $\widetilde D$ on $X$ that is quasisymmetrically equivalent to $D$. Heuristically, the $\widetilde D$-diameter of $B_{\alpha^n}(x; D)$ is roughly bounded above by a constant multiple of $\prod_{j = 1}^n \sigma(B_{\alpha^j}(x_j; D))$, where $x_n = x$ and $B_{\alpha^{j - 1}}(x_{j - 1}; D)$ represents the ``parent'' of $B_{\alpha^j}(x_j; D)$. Consequently, for $p > 2$, to show that the conformal dimension of $(X, D)$ is at most $p$, it suffices to construct an admissible weight function $\sigma$ such that 
\begin{equation}\label{eq:outline-0}
    \sum_{x_n \in A_n} \prod_{j = 1}^n \sigma(B_{\alpha^j}(x_j; D))^p \to 0 \quad \text{as } n \to \infty. 
\end{equation}

We now focus on the case where $(X, D) = (\SCS, D_\SCS)$ is the Brownian sphere. Our construction of the desired admissible weight function relies on the a priori embedding $\SCS \to \BC \cup \{\infty\}$. One natural candidate for the weight function is given by $\sigma(B_{\alpha^n}(x; D_\SCS)) = \frac{\diam(B_{\alpha^n}(x; D_\SCS))}{\inradius(B_{\alpha^{n - 1}/2}(x; D_\SCS))}$, where $\diam(\bullet)$ denotes the Euclidean diameter and $\inradius(B_{\alpha^{n - 1}/2}(x; D_\SCS)) \defeq \sup\{R \ge 0 : B_R(x) \subset B_{\alpha^{n - 1}/2}(x; D_\SCS)\}$ (with $B_R(x)$ denoting the Euclidean ball). One can immediately verify that this choice satisfies the admissibility condition. However, with this definition, it is not straightforward to verify~\eqref{eq:outline-0}. Instead, we define
\begin{equation*}
    \sigma(B_{\alpha^n}(x; D_\SCS)) = \frac{\diam(B_{\alpha^n}(x; D_\SCS))}{\inradius(B_{\alpha^{n - 1}/2}^\bullet(x; D_\SCS))} \mathbf{1}_{E(x, n)} + \mathbf{1}_{E(x, n)^c}
\end{equation*}
for some carefully designed event $E(x, n)$. Here, $B_{\alpha^{n - 1}/2}^\bullet(x; D_\SCS)$ denotes the filled metric ball, i.e., the complement of the unbounded connected component of $\BC \setminus B_{\alpha^{n - 1}/2}(x; D_\SCS)$.

In order to verify~\eqref{eq:outline-0}, we first transfer our analysis from the Brownian sphere to a metric ball of the Brownian plane $({\SCC}, D_{\SCC}, \SM_\SCC)$, or equivalently, the $\sqrt{8/3}$-LQG cone, which serves as the natural unbounded variant of the Brownian sphere. Since the Hausdorff dimension of $({\SCC}, D_{\SCC})$ is $4$, the cardinality of $A_n$ (taken to be a maximal $\alpha^n$-separated subset of a metric ball $B_{t}(0; D_{\SCC})$ of some radius $t > 0$ in $({\SCC}, D_{\SCC})$ centered at the ``origin'' of the Brownian plane) grows as $\#A_n = \alpha^{-(4 + o(1))n}$ as $n \to \infty$. By exploiting the property of independence across scales, we can reduce the verification of~\eqref{eq:outline-0} to showing that $\BE\lbrack\sigma(B_{\alpha^n}(0; D_{\SCC}))^p\rbrack = O(\alpha^{q})$ as $\alpha \to 0$ for some $q > 4$ uniformly in $n \in \BN$. This will subsequently follow from the fact that the event $E(x, n)$ occurs with superpolynomially high probability as $\alpha \to 0$, combined with the corresponding estimates for the ratio $\frac{\diam(B_{\alpha^n}(0; D_{\SCC}))}{\inradius(B_{\alpha^{n - 1}/2}^\bullet(0; D_{\SCC}))}$.

The remainder of this paper is organized as follows. \Cref{section:background-BSph-LQG} provides the necessary background material on the Brownian sphere, the Brownian plane, and Liouville quantum gravity. \Cref{section:estimates} establishes important estimates for the Brownian plane, with a particular focus on bounding the ratio $\frac{\diam(B_{\alpha^n}(0; D_{\SCC}))}{\inradius(B_{\alpha^{n - 1}/2}^\bullet(0; D_{\SCC}))}$. \Cref{section:background-Gromov} recalls basic concepts from Gromov hyperbolic geometry. In \Cref{section:hyperbolic-filling}, we detail the framework of hyperbolic fillings and explain how admissible weight functions are used to construct quasisymmetrically equivalent metrics. \Cref{section:weight} is devoted to explicitly constructing our target weight function and rigorously verifying its admissibility. Finally, \Cref{section:proof} consolidates these elements to complete the proof of our main result, \Cref{thm:main}.

\subsection*{Acknowledgements}

J.M.~received support from ERC consolidator grant ARPF (Horizon Europe UKRI G120614). Y.T.~was supported by a Cambridge International Scholarship from Cambridge Trust. 

\section{Background on the Brownian sphere and Liouville quantum gravity}\label{section:background-BSph-LQG}

\subsection{Notations and conventions}

We let $\BC$, $\BR$, $\BZ$, and $\BN$ denote the sets of complex numbers, real numbers, integers, and positive integers, respectively. We shall write $\widehat\BC \defeq \BC \cup \{\infty\}$. For real numbers $a < b$, we define the discrete interval $[a, b]_\BZ \defeq [a, b] \cap \BZ$. 

For variable non-negative quantities $a$ and $b$, we write $a \preceq b$ (resp.~$a \succeq b$) if there exists a constant $C > 0$, independent of the relevant parameters, such that $a \le Cb$ (resp.~$a \ge Cb$). We write $a \asymp b$ if both $a \preceq b$ and $a \succeq b$ hold.

For a non-negative quantity $a$ depending on a parameter $t > 0$, and for a fixed exponent $\alpha > 0$, we write $a = O(t^\alpha)$ as $t \to 0$ if there exists a constant $C > 0$ such that $a \le Ct^\alpha$ for all sufficiently small $t$. We write $a = O(t^\infty)$ as $t \to 0$ if $a = O(t^\alpha)$ as $t \to 0$ for every $\alpha > 0$. Similarly, we write $a = O(t^{-\alpha})$ and $a = O(t^{-\infty})$ as $t \to \infty$ to denote the analogous bounds for large $t$. 

Let $(X, D)$ be a metric space. For $x \in X$ and $t > 0$, we define the open metric ball $B_t(x; D) \defeq \{y \in X : D(x, y) < t\}$. For $0 < s < t$, we define the open metric annulus $A_{s,t}(x; D) \defeq \{y \in X : s < D(x, y) < t\}$. Finally, the diameter of the space is denoted by $\diam(X; D) \defeq \sup\{D(x, y) : x, y \in X\}$. When $D$ is the Euclidean metric, we omit it from our notation.

\subsection{The Brownian sphere and the Brownian plane}

In the present subsection, we review the construction and fundamental properties of the Brownian sphere and the Brownian plane.

Let $\{X_t\}_{t \in [0, 1]}$ be a normalized Brownian excursion. Given $X$, let $\{Z_t\}_{t \in [0, 1]}$ be a centered Gaussian process with covariance function
\begin{equation*}
    \BE\lbrack Z_s Z_t\rbrack = \inf_{u \in [s, t]} X_u, \quad \forall 0 \le s \le t \le 1. 
\end{equation*}
We define
\begin{equation*}
    D^\circ(s, t) \defeq Z_s + Z_t - 2\left(\inf_{u \in [s, t]} Z_u \vee \inf_{u \in [0, s] \cup [t, 1]} Z_u\right), \quad \forall 0 \le s \le t \le 1. 
\end{equation*}
Let $D$ be the maximal pseudo-metric on $[0, 1]$ such that $D \le D^\circ$ and $D(s, t) = 0$ whenever $D^\circ(s, t) = 0$. We define the quotient space $\SCS \defeq [0, 1]/\sim$, where $s \sim t$ if and only if $D(s, t) = 0$. Let $D_\SCS$ denote the metric on $\SCS$ induced by $D$, and let $\SM_\SCS$ be the pushforward of the Lebesgue measure on $[0, 1]$ under the canonical projection $\mathop{\mathrm{pr}} \colon [0, 1] \to \SCS$. We refer to the quintuple $(\SCS, D_\SCS, \SM_\SCS; \mathop{\mathrm{pr}}(0), \mathop{\mathrm{pr}}(t_\ast))$ as the \emph{Brownian sphere} with unit area, where $t_\ast$ denotes the unique time at which $Z$ attains its infimum. 

Almost surely, $(\SCS, D_\SCS)$ is a geodesic metric space homeomorphic to $\BS^2$ (cf.~\cite{MR2438999,MR2399286}), and its Hausdorff dimension is $4$ (cf.~\cite{MR2336042}). Conditional on $(\SCS, D_\SCS, \SM_\SCS)$, the marked points $\mathop{\mathrm{pr}}(0)$ and $\mathop{\mathrm{pr}}(t_\ast)$ are independent and distributed according to $\SM_\SCS$ (cf.~\cite{GeoBMap}). Furthermore, the measure $\SM_\SCS$ is almost surely given by the Hausdorff measure associated with the gauge function $r \mapsto r^4\log(\log(1/r))$ (cf.~\cite{MR4474537}). 

The law of the unconditioned Brownian sphere (i.e., with a random area) is defined as the pushforward of the measure $\BP \otimes c a^{-3/2} \, \rd a$, for some constant $c > 0$, under the scaling $(\SCS, D_\SCS, \SM_\SCS; x, y) \mapsto (\SCS, a^{1/4}D_\SCS, a\SM_\SCS; x, y)$, where $\BP$ denotes the law of the Brownian sphere with unit area. Equivalently, the unconditioned Brownian sphere can be constructed via the same procedure described above by replacing the normalized Brownian excursion with It\^o's excursion measure.

The Brownian plane is an unbounded variant of the Brownian sphere and is constructed in a similar manner (cf.~\cite{TBP}). More precisely, let $\{X_t\}_{t \in \BR}$ be a process such that $\{X_t\}_{t \ge 0}$ and $\{X_{-t}\}_{t \ge 0}$ are independent three-dimensional Bessel processes starting from zero. Given $X$, let $\{Z_t\}_{t \in \BR}$ be a centered Gaussian process with covariance function
\begin{equation*}
    \BE\lbrack Z_s Z_t\rbrack = 
    \begin{cases}
        \inf_{u \in [s, t]} X_u & \text{if } s \le t \le 0 \text{ or } 0 \le s \le t; \\
        \inf_{u \in (-\infty, s] \cup [t, \infty)} X_u & \text{if } s \le 0 \le t. 
    \end{cases} 
\end{equation*}
We define
\begin{equation*}
    D^\circ(s, t) \defeq Z_s + Z_t - 2\inf_{u \in [s, t]} Z_u, \quad \forall s \le t. 
\end{equation*}
Let $D$ be the maximal pseudo-metric on $\BR$ such that $D \le D^\circ$ and $D(s, t) = 0$ whenever $D^\circ(s, t) = 0$. We define the quotient space $\SCC \defeq \BR/\sim$, where $s \sim t$ if and only if $D(s, t) = 0$. Let $D_\SCC$ denote the metric on $\SCC$ induced by $D$, and let $\SM_\SCC$ be the pushforward of the Lebesgue measure on $\BR$ under the canonical projection $\mathop{\mathrm{pr}} \colon \BR \to \SCC$. We refer to the quadruple $(\SCC, D_\SCC, \SM_\SCC; \mathop{\mathrm{pr}}(0))$ as the \emph{Brownian plane}.

The Brownian plane satisfies a natural scaling property: For each deterministic scalar $\lambda > 0$, the rescaled space $(\SCC, \lambda D_\SCC, \lambda^4\SM_\SCC; \mathop{\mathrm{pr}}(0))$ has exactly the same law as the original space $(\SCC, D_\SCC, \SM_\SCC; \mathop{\mathrm{pr}}(0))$.

\subsection{The breadth-first exploration}\label{subsection:BFS}

Recall that a \emph{continuous-state branching process (CSBP)} with \emph{branching mechanism} $\psi$ is an $\BR_{\ge 0}$-valued Markov process $Y$ such that 
\begin{equation*}
    \BE_y\lbrack \re^{-\lambda Y_t}\rbrack = \re^{-yu_t(\lambda)}, \quad \forall y > 0, \ \forall \lambda > 0, \ \forall t > 0, 
\end{equation*}
where 
\begin{equation*}
    \partial_tu_t(\lambda) = -\psi(u_t(\lambda)) \quad \text{and} \quad u_0(\lambda) = \lambda,
\end{equation*}
and $\psi$ takes the L\'evy--Khintchine form
\begin{equation*}
    \psi(u) = au + bu^2 + \int_0^\infty (\re^{-ux} - 1 + ux) \, \Pi(\rd x). 
\end{equation*}
When $\psi(u) = c u^\alpha$ for some $\alpha \in (1, 2]$ and $c > 0$, we refer to $Y$ as an \emph{$\alpha$-stable CSBP}. In the present paper, we consider the case where $\psi(u) = \sqrt{8/3} u^{3/2}$. In this setting,
\begin{equation*}
    u_t(\lambda) = (\lambda^{-1/2} + \sqrt{2/3} t)^{-2}, \quad \forall \lambda > 0.
\end{equation*}
Let $\zeta \defeq \inf\{t > 0 : Y_t = 0\}$ denote the lifetime of $Y$. Then
\begin{equation*}
    \BP_y\lbrack\zeta \le t\rbrack = \BP_y\lbrack Y_t = 0\rbrack = \lim_{\lambda \to \infty} \BE_y\lbrack \re^{-\lambda Y_t}\rbrack = \exp\!\left(-\frac{3y}{2t^2}\right), \quad \forall t > 0, 
\end{equation*}
yielding the density 
\begin{equation}\label{eq:lifetime-density}
    \BP_y\lbrack\zeta \in \rd t\rbrack = 3yt^{-3} \exp\!\left(-\frac{3y}{2t^2}\right) \, \rd t, \quad \forall t > 0. 
\end{equation}
The $3/2$-stable CSBP and its associated excursion measure can be obtained from the spectrally positive $3/2$-stable L\'evy process and its excursion measure, respectively, via the Lamperti transformation. 

Let $(\SCS, D_\SCS, \SM_\SCS; x, y)$ be an (unconditioned) Brownian sphere. Unlike Euclidean space, the complement of the metric ball $B_t(x; D_\SCS)$ for $t > 0$ has countably many connected components (unless $B_t(x; D_\SCS) = \SCS$). We define $B_t^y(x; D_\SCS)$ as the complement of the connected component of $\SCS \setminus B_t(x; D_\SCS)$ that contains $y$. Informally, $B_t^y(x; D_\SCS)$ is obtained by filling in the holes of $B_t(x; D_\SCS)$ that do not contain $y$. 

By \cite{Hull-TBP,MR4225028}, there exists a c\`adl\`ag process $\{Y_t\}_{t \in [0, D_\SCS(x, y)]}$ with only positive jumps, almost surely determined by $(\SCS, D_\SCS, \SM_\SCS; x, y)$, such that for each deterministic $t \ge 0$, almost surely on the event that $t < D_\SCS(x, y)$, 
\begin{equation*}
    Y_{D_\SCS(x, y) - t} = \lim_{\varepsilon \to 0} \varepsilon^{-2} \SM_\SCS(B_{t + \varepsilon}(x; D_\SCS) \setminus B_t^y(x; D_\SCS)). 
\end{equation*}
Moreover, the law of $\{Y_t\}_{t \in [0, D_\SCS(x, y)]}$ is given by the $3/2$-stable CSBP excursion measure. 

Let $(\SCC, D_\SCC, \SM_\SCC; x)$ be a Brownian plane. In a similar vein, let $B_t^\bullet(x; D_\SCC)$ denote the complement of the unbounded connected component of $\SCC \setminus B_t(x; D_\SCC)$. There exists a c\`adl\`ag process $\{Y_{-t}\}_{-t \le 0}$ with only positive jumps, almost surely determined by $(\SCC, D_\SCC, \SM_\SCC; x)$, such that for each deterministic $t \ge 0$, almost surely, 
\begin{equation*}
    Y_{-t} = \lim_{\varepsilon \to 0} \varepsilon^{-2} \SM_\SCC(B_{t + \varepsilon}(x; D_\SCC) \setminus B_{t}^\bullet(x; D_\SCC)). 
\end{equation*}
Furthermore, the law of $\{Y_{-t}\}_{t \ge 0}$ is characterized by the following properties:
\begin{equation}\label{eq:hull-process}
    \parbox{.85\linewidth}{
    \begin{itemize}
        \item $Y_{-t} \to \infty$ as $t \to \infty$, and
        \item for each deterministic $x > 0$, if we write $\tau_x \defeq \sup\{t \ge 0 : Y_{-t} = x\}$, then $\{Y_{s - \tau_x}\}_{s \in [0, \tau_x]}$ has the law of a $3/2$-stable CSBP starting from $x$. 
    \end{itemize}}
\end{equation}

We define $\SN_\SCS(\partial B_t^y(x; D_\SCS)) \defeq Y_{D_\SCS(x, y) - t}$ and, respectively, $\SN_\SCC(\partial B_t^\bullet(x; D_\SCC)) \defeq Y_{-t}$, and refer to them as the \emph{boundary lengths} of the filled metric balls. 

As shown in \cite{MR4225028}, for each deterministic $t > 0$, on the event that $t < D_\SCS(x, y)$, the boundary length $\SN_\SCS(\partial B_t^y(x; D_\SCS))$ is almost surely determined by 
\begin{equation}\label{eq:filled-metric-ball-sphere}
    \left(B_t^y(x; D_\SCS), D_\SCS(\bullet, \bullet; B_t^y(x; D_\SCS)), \SM_\SCS|_{B_t^y(x; D_\SCS)}; x\right), 
\end{equation}
where $D_\SCS(\bullet, \bullet; B_t^y(x; D_\SCS))$ denotes the internal metric (i.e., the infimum of the $D_\SCS$-lengths of all paths entirely contained in $B_t^y(x; D_\SCS)$); it is also almost surely determined by
\begin{equation}\label{eq:Brownian-disk}
    \left(\SCS \setminus B_t^y(x; D_\SCS), D_\SCS(\bullet, \bullet; \SCS \setminus B_t^y(x; D_\SCS)), \SM_\SCS|_{\SCS \setminus B_t^y(x; D_\SCS)}; y\right); 
\end{equation}
furthermore, \eqref{eq:filled-metric-ball-sphere} and~\eqref{eq:Brownian-disk} are conditionally independent given $\SN_\SCS(\partial B_t^y(x; D_\SCS))$. Moreover, the conditional law of~\eqref{eq:Brownian-disk} given $\SN_\SCS(\partial B_t^y(x; D_\SCS))$ does not depend on the choice of $t$. In fact, this conditional law is referred to as the \emph{pointed Brownian disk} of boundary length $\SN_\SCS(\partial B_t^y(x; D_\SCS))$.

Analogously, for each deterministic $t > 0$, the boundary length $\SN_\SCC(\partial B_t^\bullet(x; D_\SCC))$ is almost surely determined by 
\begin{equation}\label{eq:filled-metric-ball-cone}
    \left(B_t^\bullet(x; D_\SCC), D_\SCC(\bullet, \bullet; B_t^\bullet(x; D_\SCC)), \SM_\SCC|_{B_t^\bullet(x; D_\SCC)}; x\right), 
\end{equation}
and also by
\begin{equation}\label{eq:Brownian-horn}
    \left(\SCC \setminus B_t^\bullet(x; D_\SCC), D_\SCC(\bullet, \bullet; \SCC \setminus B_t^\bullet(x; D_\SCC)), \SM_\SCC|_{\SCC \setminus B_t^\bullet(x; D_\SCC)}\right); 
\end{equation}
furthermore, \eqref{eq:filled-metric-ball-cone} and~\eqref{eq:Brownian-horn} are conditionally independent given $\SN_\SCC(\partial B_t^\bullet(x; D_\SCC))$. Moreover, the conditional law of~\eqref{eq:Brownian-horn} given $\SN_\SCC(\partial B_t^\bullet(x; D_\SCC))$ does not depend on the choice of $t$. Last but not least, the conditional laws of the filled metric balls~\eqref{eq:filled-metric-ball-sphere} and~\eqref{eq:filled-metric-ball-cone} given their boundary lengths are identical. 

For $0 \le s < t$, we define the \emph{metric bands} $A_{s,t}^y(x; D_\SCS) \defeq B_t^y(x; D_\SCS) \setminus B_s^y(x; D_\SCS)$ and $A_{s,t}^\bullet(x; D_\SCC) \defeq B_t^\bullet(x; D_\SCC) \setminus B_s^\bullet(x; D_\SCC)$. For $\ell_1, \ell_2 \ge 0$, the conditional law of 
\begin{equation}\label{eq:metric-band-sphere}
    \left(A_{s,t}^y(x; D_\SCS), D_\SCS(\bullet, \bullet; A_{s,t}^y(x; D_\SCS)), \SM_\SCS|_{A_{s,t}^y(x; D_\SCS)}; \partial B_s^y(x; D_\SCS), \partial B_t^y(x; D_\SCS)\right)
\end{equation}
given $\SN_\SCS(\partial B_s^y(x; D_\SCS)) = \ell_1$, $\SN_\SCS(\partial B_t^y(x; D_\SCS)) = \ell_2$, and the event that $t \le D_\SCS(x, y)$, is identical to the conditional law of 
\begin{equation}\label{eq:metric-band-cone}
    \left(A_{s,t}^\bullet(x; D_\SCC), D_\SCC(\bullet, \bullet; A_{s,t}^\bullet(x; D_\SCC)), \SM_\SCC|_{A_{s,t}^\bullet(x; D_\SCC)}; \partial B_s^\bullet(x; D_\SCC), \partial B_t^\bullet(x; D_\SCC)\right)
\end{equation}
given $\SN_\SCC(\partial B_s^\bullet(x; D_\SCC)) = \ell_1$ and $\SN_\SCC(\partial B_t^\bullet(x; D_\SCC)) = \ell_2$. Moreover, this common conditional law depends only on $\ell_1$, $\ell_2$, and the width $t - s$, and we refer to it as the law of the \emph{metric band with inner boundary length $\ell_1$, outer boundary length $\ell_2$, and width $t - s$}. 

Similarly, we can consider the conditional laws of~\eqref{eq:metric-band-sphere} and~\eqref{eq:metric-band-cone} given only that the inner boundary length is $\ell_1$ (and, for the sphere, the event that $s < D_\SCS(x, y)$), without conditioning on the outer boundary length. While these two conditional laws still depend only on $\ell_1$ and $t - s$, they are no longer identical. We refer to them as the \emph{metric bands of sphere (resp.~cone) type with inner boundary length $\ell_1$ and width $t - s$.}

These metric bands satisfy a natural scaling property: If $(\SCA, D_\SCA, \SM_\SCA; I, O)$ is a metric band with inner boundary length $\ell_1$, outer boundary length $\ell_2$, and width $t - s$, then for any deterministic $\lambda > 0$, the rescaled tuple $(\SCA, \lambda D_\SCA, \lambda^4\SM_\SCA; I, O)$ is a metric band with inner boundary length $\lambda^2\ell_1$, outer boundary length $\lambda^2\ell_2$, and width $\lambda(t - s)$. An analogous statement holds for metric bands of sphere (resp.~cone) type.

\subsection{Liouville quantum gravity and space-filling SLE}

Let $\gamma \in (0, 2)$ and $Q \defeq 2/\gamma + \gamma/2$. A \emph{$\gamma$-LQG surface} is an equivalence class of triples $(U, g, \Phi)$, where $U$ is a Riemann surface, $g$ is a conformal metric on $U$ (i.e., a Riemannian metric of the form $\varrho(z)^2 \, \rd z \, \rd\overline z$ in each local chart), and $\Phi$ is a Schwartz distribution (an instance of the Gaussian free field) on $(U, g)$. Two such triples $(U, g, \Phi)$ and $(U^\prime, g^\prime, \Phi^\prime)$ are equivalent if there exists a conformal mapping $\phi \colon U^\prime \to U$ and a smooth function $\varrho \colon U^\prime \to \BR_{> 0}$ such that $\phi^\ast(g) = \varrho^2 g^\prime$ and $\Phi^\prime = \Phi \circ \phi + Q\log(\varrho)$. The particular choice of a representative $(U, g, \Phi)$ from the equivalence class is referred to as an \emph{embedding} of the $\gamma$-LQG surface. 

We will also consider \emph{marked} $\gamma$-LQG surfaces: If $A_1, \cdots, A_n$ (resp.~$A_1^\prime, \cdots, A_n^\prime$) are subsets of the completion of $U$ (resp.~$U^\prime$), then the tuples $(U, g, \Phi; A_1, \cdots, A_n)$ and $(U^\prime, g^\prime, \Phi^\prime; A_1^\prime, \cdots, A_n^\prime)$ are considered equivalent if the conformal mapping $\phi$ additionally satisfies $\phi(A_j^\prime) = A_j$ for all $j \in [1, n]_\BZ$. In the present paper, we focus on the case where $U$ is an open subset of $\BC$ and $g$ is the Euclidean metric; consequently, we omit $g$ from the notation.

As shown in \cite{LQGKPZ}, a $\gamma$-LQG surface admits a natural \emph{$\gamma$-LQG measure}, denoted by $\SM_\Phi$. This measure is formally given by ``$\re^{\gamma\Phi} \, \rd x \, \rd y$'' and does not depend on the choice of embedding. 

Furthermore, the surface admits a natural \emph{$\gamma$-LQG metric}, denoted by $D_\Phi$, which is formally given by ``$\re^{\gamma\Phi}(\rd x^2 + \rd y^2)$'' and similarly does not depend on the embedding. This metric was initially constructed for the specific case of $\gamma = \sqrt{8/3}$ in \cite{MR3572845,MR4050102,MR4225028,MR4348679,MR4242633}, and later extended to all $\gamma \in (0, 2)$ by \cite{TightLFPP,WeakLQGMet,ConfLQG,LocMetGFF,ExUniLQG}.

Recall that a \emph{whole-plane Gaussian free field (GFF)} $\Phi$ is defined as a random Schwartz distribution modulo additive constants such that 
\begin{equation*}
    \left\{\langle\Phi, f\rangle : f \in C_c^\infty(\BC), \ \int_\BC f(z) \, \rd z = 0\right\}
\end{equation*}
forms a centered Gaussian process with the covariance function
\begin{equation*}
    \BE\lbrack\langle\Phi, f\rangle\langle\Phi, g\rangle\rbrack = -\int_{\BC \times \BC} \log(|x - y|) f(x) g(y) \, \rd x \, \rd y. 
\end{equation*}
Let $H^1(\BC)$ denote the completion of the space 
\begin{equation*}
    \left\{f \in C^\infty(\BC) / \BR : \int_\BC \|\nabla f(z)\|^2 \, \rd z < \infty\right\} 
\end{equation*}
with respect to the inner product $\langle f, g\rangle_{H^1(\BC)} \defeq \frac1{2\pi}\int_\BC \nabla f(z) \cdot \nabla g(z) \, \rd z$. The field $\Phi$ can then equivalently be represented by the series expansion $\sum_{j = 1}^\infty X_j \psi_j$, where $(\psi_j)_{j \ge 1}$ forms an orthonormal basis for $H^1(\BC)$ and $(X_j)_{j \ge 1}$ is a sequence of independent standard Gaussian random variables.

We now proceed to define the $\gamma$-LQG sphere and the $\gamma$-LQG cone (cf.~\cite{dms2021mating}). Observe that $H^1(\BC)$ admits the orthogonal decomposition $H^1(\BC) = H^{\mathrm{rad}}(\BC) \oplus H^{\mathrm{circ}}(\BC)$, where 
\begin{align*}
    H^{\mathrm{rad}}(\BC) &\defeq \left\{f \in H^1(\BC) : f \text{ is constant on } \{z \in \BC : |z| = t\} \text{ for each } t > 0\right\}; \\
    H^{\mathrm{circ}}(\BC) &\defeq \left\{f \in H^1(\BC) : f \text{ has mean zero on } \{z \in \BC : |z| = t\} \text{ for each } t > 0\right\}. 
\end{align*}
Let us define the process $\{A_t\}_{t \in \BR}$ by $A_t \defeq B_t + \gamma t$ for $t \ge 0$ and $A_t \defeq \widetilde B_{-t} + \gamma t$ for $t < 0$, where $\{B_t\}_{t \ge 0}$ is a standard Brownian motion starting from $0$, and $\{\widetilde B_t\}_{t \ge 0}$ is an independent Brownian motion starting from $0$, conditioned such that $\widetilde B_t + (Q - \gamma)t > 0$ for all $t > 0$. Next, let $\Phi$ be a whole-plane GFF independent of $\{A_t\}_{t \in \BR}$. We define $\Phi^{\mathrm{cone}}$ as the random Schwartz distribution whose projection onto $H^{\mathrm{rad}}(\BC)$ is given by $A_{-\log(|\bullet|)}$ (where the process defined by $\Phi_r^{\mathrm{cone}}(0) \defeq A_{-\log(r)}$ for $r > 0$ is referred to as the \emph{circle average} process of $\Phi^{\mathrm{cone}}$), and whose projection onto $H^{\mathrm{circ}}(\BC)$ coincides with that of $\Phi$. The pointed $\gamma$-LQG surface represented by $(\BC, \Phi^{\mathrm{cone}}; 0, \infty)$ is then referred to as a \emph{$\gamma$-LQG cone}, with this specific parametrization being its \emph{circle average embedding}. 

Let $e$ be a Bessel excursion of dimension $4 - 8/\gamma^2$. Let $\{A_t^e\}_{t \in \BR}$ be the process $(2/\gamma)\log(e)$, reparameterized to have quadratic variation $\rd t$. If we replace $\{A_t\}_{t \in \BR}$ with $\{A_t^e\}_{t \in \BR}$ in the definition of the $\gamma$-LQG cone, we obtain the \emph{$\gamma$-LQG sphere}. Since the Bessel excursion measure is $\sigma$-finite, the law of a $\gamma$-LQG sphere is also a $\sigma$-finite measure.

It has been established in \cite{MR4050102,MR4225028,MR4348679,MR4242633} that the Brownian sphere $(\SCS, D_\SCS, \SM_\SCS; x, y)$ and the $\sqrt{8/3}$-LQG sphere $(\widehat\BC, D_{\Phi^{\mathrm{sph}}}, \SM_{\Phi^{\mathrm{sph}}}; 0, \infty)$ are equivalent in law as pointed metric measure spaces. Furthermore, the $\sqrt{8/3}$-LQG surface structure --- and in particular, its conformal structure as a Riemann surface --- is almost surely determined by this underlying metric measure space structure. Analogously, the Brownian plane $(\SCC, D_\SCC, \SM_\SCC; x)$ and the $\sqrt{8/3}$-LQG cone $(\BC, D_{\Phi^{\mathrm{cone}}}, \SM_{\Phi^{\mathrm{cone}}}; 0)$ also share the same law as pointed metric measure spaces.

Let $\kappa \defeq \gamma^2 \in (0, 4)$ and $\kappa^\prime \defeq 16/\kappa > 4$. The whole-plane space-filling SLE$_{\kappa^\prime}$ curve from $\infty$ to $\infty$, first constructed in \cite{IG4}, is a versatile object that, among its many applications, provides a natural framework for formulating the translation invariance property of the $\gamma$-LQG cone. This curve, denoted by $\eta^\prime \colon \BR \to \BC$, is almost surely non-self-crossing and space-filling, with $\eta^\prime(t) \to \infty$ as $t \to \pm\infty$. Modulo reparameterization, the law of $\eta^\prime$ is invariant under M\"obius transformations that fix $\infty$. Consequently, the whole-plane space-filling SLE$_{\kappa^\prime}$ curve is naturally well-defined on a $\gamma$-LQG sphere or cone. We may fix the parameterization such that $\eta^\prime(0) = 0$. As shown in \cite{dms2021mating}, if $(\BC, \Phi; 0, \infty)$ is a $\gamma$-LQG cone independent of $\eta^\prime$, and $\eta^\prime$ is parameterized by the $\gamma$-LQG measure $\SM_\Phi$, then for any deterministic time $t \in \BR$, the pointed $\gamma$-LQG surface $(\BC, \Phi; \eta^\prime(t), \infty)$ also has the law of a $\gamma$-LQG cone.

Throughout the remainder of the present paper, we shall write
\begin{gather*}
    \gamma \defeq \sqrt{8/3}; \quad Q \defeq 2/\gamma + \gamma/2 = 5/\sqrt6; \quad d_\gamma \defeq 4; \quad \xi \defeq \gamma/d_\gamma = 1/\sqrt6; \\
    \kappa \defeq \gamma^2 = 8/3; \quad \kappa^\prime \defeq 16/\kappa = 6. 
\end{gather*}

\section{Estimates for the Brownian plane}\label{section:estimates}

In the present section, we establish several estimates for the Brownian plane. In \Cref{subsection:CSBP}, we collect estimates for $3/2$-stable CSBPs. These results show that the boundary lengths of filled metric balls in the Brownian plane cannot significantly exceed their expected values across multiple scales. In \Cref{subsection:volume}, we present estimates for the volumes of metric balls. These volume bounds will allow us to construct dense nets for the Brownian plane with high probability via independent random sampling. Finally, in \Cref{subsection:conformal-moduli}, we prove estimates concerning the conformal moduli of metric bands in the Brownian plane. These estimates control the ratio between the Euclidean diameter and the inradius mentioned in \Cref{subsection:outline}, which is an important ingredient for the proof of our main theorem.

\subsection{Boundary lengths of filled metric balls}\label{subsection:CSBP}

\begin{lemma}\label{lem:CSBP}
    Let $y, T > 0$. Let $Y$ be a $3/2$-stable CSBP starting from $y$ with lifetime $T$. Then
    \begin{equation*}
        \BP_y\lbrack Y_t > A(T - t)^2\rbrack \le 3^{3/2} \exp\!\left(-A + 3(\sqrt3 - 1)y(T - t)/T^3\right), \quad \forall t \in (0, T), \ \forall A > 0. 
    \end{equation*}
\end{lemma}

\begin{proof}
    Let $Y$ be a $3/2$-stable CSBP starting from $y$ (without conditioning on its lifetime $\zeta$). Then it follows from the Markov property and~\eqref{eq:lifetime-density} that
    \begin{align*}
        \BE_y\lbrack \re^{\lambda Y_t} \mid \zeta = T\rbrack &= \frac{\BE_y\lbrack\re^{\lambda Y_t} \mathbf 1_{\{\zeta \in \rd T\}}\rbrack}{\BP_y\lbrack\zeta \in \rd T\rbrack} = \frac{\BE_y\lbrack\re^{\lambda Y_t}\BP_{Y_t}\lbrack\zeta \in \rd (T - t)\rbrack\rbrack}{\BE_y\lbrack\BP_{Y_t}\lbrack\zeta \in \rd (T - t)\rbrack\rbrack} = \frac{\BE_y\lbrack Y_t\re^{-(\theta - \lambda)Y_t}\rbrack}{\BE_y\lbrack Y_t\re^{-\theta Y_t}\rbrack} \\
        &= \frac{u_t^\prime(\theta - \lambda)}{u_t^\prime(\theta)} \exp(-y(u_t(\theta - \lambda) - u_t(\theta))), \quad \forall \lambda \in (0, \theta), 
    \end{align*}
    where $\theta \defeq (3/2)(T - t)^{-2}$. Take $\lambda \defeq (T - t)^{-2}$. Then
    \begin{equation*}
        \frac{u_t^\prime(\theta - \lambda)}{u_t^\prime(\theta)} = \left(\frac\theta{\theta - \lambda}\right)^{3/2} \left(\frac{\theta^{-1/2} + \sqrt{2/3}t}{(\theta - \lambda)^{-1/2} + \sqrt{2/3}t}\right)^3 = 3^{3/2}\left(\frac{\theta^{-1/2} + \sqrt{2/3}t}{(\theta - \lambda)^{-1/2} + \sqrt{2/3}t}\right)^3 \le 3^{3/2}, 
    \end{equation*}
    and
    \begin{align*}
        \exp(-y(u_t(\theta - \lambda) - u_t(\theta)))
    &= \exp\!\left(\frac{3y}{2T^2}\left(1 - \left(1 + \frac{(\sqrt3 - 1)(T - t)}T\right)^{-2}\right)\right)\\
    &\le \exp\!\left(3(\sqrt3 - 1)y(T - t)/T^3\right),
    \end{align*}
    where the last inequality follows from the fact that $1 - (1 + x)^{-2} \le 2x$ for all $x \ge 0$. Thus, we conclude that
    \begin{equation*}
        \BE_y\lbrack \re^{\lambda Y_t} \mid \zeta = T\rbrack \le 3^{3/2} \exp\!\left(3(\sqrt3 - 1)y(T - t)/T^3\right), 
    \end{equation*}
    which implies that
    \begin{equation*}
        \BE_y\lbrack Y_t \ge A(T - t)^2 \mid \zeta = T\rbrack \le \re^{-A}\BE\lbrack \re^{\lambda Y_t} \mid \zeta = T\rbrack \le 3^{3/2} \exp\!\left(-A + 3(\sqrt3 - 1)y(T - t)/T^3\right). 
    \end{equation*}
    This completes the proof of \Cref{lem:CSBP}. 
\end{proof}

\begin{lemma}\label{lem:CSBP-independence}
    Let $(\BC, \Phi; 0, \infty)$ be a $\gamma$-LQG cone. For each $t \ge 0$, write $Y_{-t} \defeq \SN_\Phi(\partial B_t^\bullet(0; D_\Phi))$. 
    \begin{enumerate}
        \item\label{it:CSBP-independence-0} For each $\lambda \in (0, 1)$, $\alpha > 0$, and $b \in (0, 1)$, there exists $A = A(\lambda, \alpha, b) \in (0, 1)$ such that the following is true: Let $\{t_j\}_{j \in \BN}$ be a decreasing sequence of positive real numbers such that $t_{j + 1}/t_j \le \lambda$ for all $j \in \BN$. Then 
        \begin{equation*}
            \BP\lbrack\#\{j \in [1, n]_\BZ : Y_{-t_j} \le At_j^2\} \ge bn\rbrack \ge 1 - \re^{-\alpha n}, \quad \forall n \in \BN. 
        \end{equation*}
        \item\label{it:CSBP-independence-1} For each $\zeta > 0$ and $b \in (0, 1)$, there exists $\lambda_\ast = \lambda_\ast(\zeta, b) \in (0, 1)$
        \begin{equation*}
            \BP\lbrack\#\{j \in [1, n]_\BZ : Y_{-\lambda^j} \le \lambda^{2j - 2\zeta)}\} \ge bn\rbrack \ge 1 - \exp(-\lambda^{-\zeta}n), \quad \forall \lambda \in (0, \lambda_\ast], \ \forall n \in \BN. 
        \end{equation*}
    \end{enumerate}
\end{lemma}

\begin{proof}
    First, we consider assertion~\eqref{it:CSBP-independence-0}. Fix $\beta > 0$ to be chosen later. For each $n \in \BN$, write $j_n$ for the $n$-th smallest $j \in \BN$ for which $Y_{-t_j} \le At_j^2$. Consider $i_0 \defeq 1$ and $i_n \defeq i_{n - 1} + \left\lceil\frac{\log(Y_{-t_{i_{n - 1}}}/t_{i_{n - 1}}^2)}{\log(1/\lambda)}\right\rceil$ for all $n \in \BN$. Recall from \cite[Proposition~1.2]{Hull-TBP} that $Y_{-t}$ is a Gamma random variable with shape parameter $3/2$ and mean $t^2$. This implies that $\BE\lbrack\re^{\beta(i_1 - 1)}\mathbf 1_{\{Y_{-t_1} > A t_1^2\}}\rbrack \le C_1\re^{-C_2A}$ for some universal constants $C_1, C_2 > 0$. Moreover, 
    \begin{align*}
        \BE\!\left\lbrack\re^{\beta(i_{n + 1} - i_n)}\mathbf 1_{\{Y_{-t_{i_n}} > A t_{i_n}^2\}} \ \middle\vert \ Y_{-t_{i_{n - 1}}}\right\rbrack &\le \re^{\beta}\BE\!\left\lbrack (Y_{-t_{i_n}}/t_{i_n}^2)^{\frac\beta{\log(1/\lambda)}} \mathbf 1_{\{Y_{-t_{i_n}} > A t_{i_n}^2\}} \ \middle\vert \ Y_{-t_{i_{n - 1}}}\right\rbrack \\
        &= \frac{\re^{\beta}\beta}{\log(1/\lambda)} \int_A^\infty x^{\frac\beta{\log(1/\lambda)} - 1} \BP\!\left\lbrack Y_{-t_{i_n}}/t_{i_n}^2 \ge x \ \middle\vert \ Y_{-t_{i_{n - 1}}}\right\rbrack \, \rd x \\
        &\le \frac{\re^{\beta}\beta}{\log(1/\lambda)} \int_A^\infty x^{\frac\beta{\log(1/\lambda)} - 1} C_3 \re^{-C_4x} \, \rd x \quad \text{(by \Cref{lem:CSBP})} \\
        &\le \frac{\re^{\beta}\beta}{\log(1/\lambda)} C_3 A^{\frac\beta{\log(1/\lambda)} - 1} \re^{-C_5A}, 
    \end{align*}
    where $C_3, C_4, C_5 > 0$ are universal constants. By choosing $A$ to be sufficiently large, we may arrange that $C_1\re^{-C_2A} \le 1 - 1/\re$ and $\frac{\re^{\beta}\beta}{\log(1/\lambda)} C_3 A^{\frac\beta{\log(1/\lambda)} - 1} \re^{-C_5A} \le 1 - 1/\re$. Write $N \defeq \inf\{n \ge 0 : Y_{-t_{i_n}} \le At_{i_n}^2\}$. Then $j_1 \le i_N$. Thus, $\BE\lbrack\re^{\beta(j_1 - 1)}\rbrack \le \BE\lbrack\re^{\beta(i_N - 1)}\rbrack \le \sum_{n = 0}^\infty (1 - 1/\re)^n = \re$. In a similar vein, $\BE\lbrack\re^{\beta(j_{n + 1} - j_n - 1)} \mid j_n\rbrack \le \re$ for all $n \in \BN$. This implies that $\BE\lbrack\re^{\beta(j_n - n)}\rbrack \le \re^n$ for all $n \in \BN$. Thus, we conclude that
    \begin{equation*}
        \BP\lbrack\#\{j \in [1, n]_\BZ : Y_{-t_j} \le At_j^2\} < bn\rbrack \le \BP\lbrack j_{bn} > n\rbrack \le \re^{bn - \beta(1 - b)n}, \quad \forall n \in \BN.  
    \end{equation*}
    This completes the proof of assertion~\eqref{it:CSBP-independence-0}. 

    By setting $\alpha = \lambda^{-\zeta}$ (hence $\beta = (\lambda^{-\zeta} + b)/(1 - b)$) and $A = \lambda^{-2\zeta}$, assertion~\eqref{it:CSBP-independence-1} follows immediately from a similar argument to the argument applied in the proof of assertion~\eqref{it:CSBP-independence-0}. 
\end{proof}

Furthermore, we record a lemma showing that the boundary lengths of filled metric balls in the Brownian plane do not deviate significantly below their expected values.

\begin{lemma}\label{lem:hull-process}
    Let $(\BC, \Phi; 0, \infty)$ be a $\gamma$-LQG cone. For each $t \ge 0$, write $Y_{-t} \defeq \SN_\Phi(\partial B_t^\bullet(0; D_\Phi))$. Then there are universal constants $\alpha, C > 0$ such that
    \begin{equation*}
        \BP\!\left\lbrack\sup\nolimits_{r \in [s, t]} Y_{-r} < \varepsilon (t - s)^2 \ \middle\vert \ Y_{-s} = \ell\right\rbrack \le C\exp(-\alpha\varepsilon^{-1/2}), \quad \forall 0 \le s < t, \ \forall \varepsilon \in (0, 1), \ \forall \ell \ge 0. 
    \end{equation*}
\end{lemma}

\begin{proof}
    We follow the argument applied in the proof of \cite[Lemma~3.1]{MR4348679}. Fix $0 \le s < t$, $\varepsilon \in (0, 1)$, and $\ell \ge 0$. Write $\sigma = \inf\{r \ge 0 : Y_{-r} = \ell\}$ and $\tau \defeq \sup\{r \ge 0 : Y_{-r} = \varepsilon(t - s)^2\}$. By the Markov property, 
    \begin{equation*}
        \BP\!\left\lbrack\sup\nolimits_{r \in [s, t]} Y_{-r} < \varepsilon (t - s)^2 \ \middle\vert \ Y_{-s} = \ell\right\rbrack = \BP\!\left\lbrack\sup\nolimits_{r \in [\sigma, \sigma + (t - s)]} Y_{-r} < \varepsilon (t - s)^2\right\rbrack. 
    \end{equation*}
    Again, by the Markov property (cf.~\eqref{eq:hull-process}), $\{Y_{r - \tau}\}_{r \in [0, \tau]}$ has the law of a $3/2$-stable CSBP starting from $\varepsilon(t - s)^2$. Write $E$ for the event that there exists a subinterval of $[0, \tau]$ with length at least $t - s$ during which $Y_{r - \tau}$ is less than $\varepsilon(t - s)^2$. It is clear that the event that $\sup_{r \in [\sigma, \sigma + (t - s)]} Y_{-r} < \varepsilon (t - s)^2$ is contained in the event $E$. Thus, it suffices to show that there are universal constants $\alpha, C > 0$ such that $\BP\lbrack E\rbrack \le C\exp(-\alpha\varepsilon^{-1/2})$. This follows immediately from the proof of \cite[Lemma~3.1]{MR4348679}. This completes the proof of \Cref{lem:hull-process}. 
\end{proof}

\subsection{Volumes of metric balls}\label{subsection:volume}

\begin{lemma}\label{lem:Brownian-cone-volume}
    Let $(\BC, \Phi; 0, \infty)$ be a $\gamma$-LQG cone. Then for each $\zeta > 0$, there almost surely exists $\varepsilon_\ast \in (0, 1)$ such that 
    \begin{equation*}
        \varepsilon^{d_\gamma + \zeta} \le \SM_\Phi(B_\varepsilon(z; D_\Phi)) \le \varepsilon^{d_\gamma - \zeta}, \quad \forall z \in B_1(0; D_\Phi), \ \forall \varepsilon \in (0, \varepsilon_\ast].
    \end{equation*}
\end{lemma}

\begin{proof}
    By \cite[Lemma~2.2]{BM-Geodesics}, for almost every instance $(\widehat\BC, \Phi^{\mathrm{sph}}; 0, \infty)$ of a $\gamma$-LQG sphere, there exists $\varepsilon_\ast \in (0, 1)$ such that
    \begin{equation*}
        \varepsilon^{d_\gamma + \zeta} \le \SM_{\Phi^{\mathrm{sph}}}(B_\varepsilon(z; D_{\Phi^{\mathrm{sph}}})) \le \varepsilon^{d_\gamma - \zeta}, \quad \forall z \in \BC, \ \forall \varepsilon \in (0, \varepsilon_\ast].
    \end{equation*}
    Thus, \Cref{lem:Brownian-cone-volume} follows immediately from the fact that, on the event that $D_{\Phi^{\mathrm{sph}}}(0, \infty) > 2$, the laws of the $\gamma$-LQG surfaces parameterized by $B_2^\bullet(0; D_\Phi)$ and $B_2^\bullet(0; D_{\Phi^{\mathrm{sph}}})$ are mutually absolutely continuous. (Indeed, the laws of the boundary lengths of $B_2^\bullet(0; D_\Phi)$ and $B_2^\bullet(0; D_{\Phi^{\mathrm{sph}}})$ are both mutually absolutely continuous with respect to the Lebesgue measure on $\BR_{> 0}$, and the $\gamma$-LQG surfaces parameterized by $B_2^\bullet(0; D_\Phi)$ and $B_2^\bullet(0; D_{\Phi^{\mathrm{sph}}})$ have the same conditional law given their boundary lengths.)
\end{proof}

\begin{lemma}\label{lem:net}
    Let $(\BC, \Phi; 0, \infty)$ be a $\gamma$-LQG cone. Let $\eta^\prime \colon (-\infty, \infty) \to \BC$ be an independent whole-plane space-filling SLE$_{\kappa^\prime}$ curve from $\infty$ to $\infty$ parameterized by $\SM_\Phi$. Let $T > 0$. Given $\Phi$ and $\eta^\prime$, let $\{x_n\}_{n \in \BN}$ be conditionally independent samples from $\SM_\Phi|_{\eta^\prime([-T, T])}$ (renormalized to be a probability measure). Then for each $\zeta > 0$, almost surely on the event that $B_1(0; D_\Phi) \subset \eta^\prime([-T, T])$, there exists $\varepsilon_\ast \in (0, 1)$ such that 
    \begin{equation}\label{eq:net}
        B_1(0; D_\Phi) \subset \bigcup \left\{B_\varepsilon(x_n; D_\Phi) : n \in [1, \varepsilon^{-d_\gamma - \zeta}]_\BZ, \ x_n \in B_1(0; D_\Phi)\right\}, \quad \forall \varepsilon \in (0, \varepsilon_\ast].
    \end{equation}
\end{lemma}

\begin{proof}
    Write $E_1(T)$ for the event that $B_1(0; D_\Phi) \subset \eta^\prime([-T, T])$. For each $\varepsilon_0 \in (0, 1)$, write $E_2(\varepsilon_0)$ for the event that
    \begin{equation*}
        \varepsilon^{d_\gamma + \zeta/2} \le \SM_\Phi(B_\varepsilon(z; D_\Phi)) \le \varepsilon^{d_\gamma - \zeta/2}, \quad \forall z \in B_1(0; D_\Phi), \ \forall \varepsilon \in (0, \varepsilon_0].
    \end{equation*}
    By \Cref{lem:Brownian-cone-volume}, it suffices to show that, almost surely on the event $E_1(T) \cap E_2(\varepsilon_0)$, there exists $\varepsilon_\ast \in (0, 1)$ such that~\eqref{eq:net} holds. For each $j \in \BN$, write $F(j)$ for the event that
    \begin{equation*}
        B_1(0; D_\Phi) \subset \bigcup \left\{B_{2^{-j - 1}}(x_n; D_\Phi) : n \in [1, 2^{j(d_\gamma + \zeta)}]_\BZ, \ x_n \in B_1(0; D_\Phi)\right\}. 
    \end{equation*}
    Note that if $F(j)$ occurs for all $j \ge \lfloor\log_2(1/\varepsilon_\ast)\rfloor$, then~\eqref{eq:net} holds. Note that if $F(j)$ does not occur, then there exists $x \in B_1(0; D_\Phi)$ such that $x_n \notin B_{2^{-j - 1}}(x; D_\Phi)$ for all $n \in [1, 2^{j(d_\gamma + \zeta)}]_\BZ$ with $x_n \in B_1(0; D_\Phi)$. In this case, if $D_\Phi(0, x) \ge 2^{-j - 2}$, let $x^\prime$ denote the point on the $D_\Phi$-geodesic from $0$ to $x$ such that $D_\Phi(x^\prime, x) = 2^{-j - 2}$; otherwise, let $x^\prime \defeq 0$. In either case, we have $B_{2^{-j - 2}}(x^\prime; D_\Phi) \subset B_1(0; D_\Phi)$ and $x_n \notin B_{2^{-j - 2}}(x^\prime; D_\Phi)$ for all $n \in [1, 2^{j(d_\gamma + \zeta)}]_\BZ$. Write $G(j)$ for the event that $x_{\lfloor2^{j(d_\gamma + \zeta)}\rfloor + 1} \in B_1(0; D_\Phi)$ and that $D_\Phi(x_n, x_{\lfloor2^{j(d_\gamma + \zeta)}\rfloor + 1}) \ge 2^{-j - 3}$ for all $n \in [1, 2^{j(d_\gamma + \zeta)}]_\BZ$. Then
    \begin{align*}
        \BP\lbrack G(j) \mid E_1(T) \cap E_2(\varepsilon_0) \cap F(j)^c\rbrack &\ge \BP\!\left\lbrack x_{\lfloor2^{j(d_\gamma + \zeta)}\rfloor + 1} \in B_{2^{-j - 3}}(x^\prime; D_\Phi) \ \middle\vert \ E_1(T) \cap E_2(\varepsilon_0) \cap F(j)^c\right\rbrack \\
        &\ge 2^{-(j + 3)(d_\gamma + \zeta/2)}/(2T). 
    \end{align*}
    On the other hand, 
    \begin{equation*}
        \BP\lbrack G(j) \mid E_1(T) \cap E_2(\varepsilon_0)\rbrack \le (1 - 2^{-(j + 3)(d_\gamma + \zeta/2)}/(2T))^{\lfloor2^{j(d_\gamma + \zeta)}\rfloor} \le \exp(-c2^{j\zeta/2})
    \end{equation*}
    for some $c = c(\zeta, T) > 0$. Thus, we conclude that
    \begin{equation*}
        \BP\lbrack F(j)^c \mid E_1(T) \cap E_2(\varepsilon_0)\rbrack \le \frac{\BP\lbrack G(j) \mid E_1(T) \cap E_2(\varepsilon_0)\rbrack}{\BP\lbrack G(j) \mid E_1(T) \cap E_2(\varepsilon_0) \cap F(j)^c\rbrack} \le \exp(-c2^{j\zeta/2}) \cdot 2^{(j + 3)(d_\gamma + \zeta/2)} \cdot (2T). 
    \end{equation*}
    Thus, we conclude from the Borel--Cantelli lemma that, almost surely on the event $E_1(T) \cap E_2(\varepsilon_0)$, there exists $j_0 \in \BN$ such that $F(j)$ occurs for all $j \ge j_0$. This completes the proof of \Cref{lem:net}. 
\end{proof}

\subsection{Conformal moduli of metric bands}\label{subsection:conformal-moduli}

\begin{proposition}\label{lem:conformal-moduli-marginal}
    Let $(\BC, \Phi; 0, \infty)$ be a $\gamma$-LQG cone. Then for each $p > 2$, there exists $q = q(p) > d_\gamma$ such that
    \begin{equation*}
        \BE\lbrack\re^{-2\pi m_\varepsilon p}\rbrack = O(\varepsilon^q) \quad \text{as } \varepsilon \to 0,
    \end{equation*}
    where $m_\varepsilon$ denotes the conformal modulus of $A_{\varepsilon,1}^\bullet(0; D_\Phi)$. 
\end{proposition}

\begin{proposition}\label{lem:conformal-moduli-band}
    Let $\SCA_T$ be a metric band of cone type with inner boundary length one and width $T > 0$. Write $m_T$ for the conformal modulus of $\SCA_T$. Then for each $p > 2$, there exists $q = q(p) > d_\gamma$ such that
    \begin{equation*}
        \BE\lbrack\re^{-2\pi m_T p}\rbrack = O(T^{-q}) \quad \text{as } T \to \infty. 
    \end{equation*}
\end{proposition}

\Cref{lem:conformal-moduli-band} is a fixed-inner-boundary-length analogue of \Cref{lem:conformal-moduli-marginal} and follows easily from the latter. The proof of \Cref{lem:conformal-moduli-marginal} requires several preliminary lemmas. For an overview of the proof strategy, see the discussion immediately preceding it.

The following lemma states that any doubly connected domain surrounding the origin with a large conformal modulus contains a centered Euclidean annulus whose conformal modulus differs from that of the domain by at most an additive constant.

\begin{lemma}\label{lem:Teichmueller}
    Let $A \subset \BC$ be a doubly connected domain such that the origin is contained in the bounded connected component of $\BC \setminus A$. Write $m$ for the conformal modulus of $A$; $I$ (resp.~$O$) for the inner (resp.~outer) boundaries of $A$; $R_1 \defeq \inf\{R > 0 : I \subset B_R(0)\}$; $R_2 \defeq \sup\{R > 0 : O \cap B_R(0) = \emptyset\}$. Then
    \begin{equation*}
        \re^{2\pi m}/16 - 1 \le R_2/R_1 \le \re^{2\pi m}. 
    \end{equation*}
\end{lemma}

\begin{proof}
    It is clear that 
    \begin{equation*}
        m \ge (\text{conformal modulus of } A_{R_1,R_2}(0)) = \frac1{2\pi}\log(R_2/R_1), 
    \end{equation*}
    which implies that $R_2/R_1 \le \re^{2\pi m}$. On the other hand, by \cite[Theorem~4-7 and (4-21)]{CoInAhl}, 
    \begin{equation*}
        m \le (\text{conformal modulus of } \BC \setminus ([0, R_1] \cup [R_2, \infty))) \le \frac1{2\pi}\log(16(R_2/R_1 + 1)), 
    \end{equation*}
    which implies that $\re^{2\pi m}/16 - 1 \le R_2/R_1$. This completes the proof of \Cref{lem:Teichmueller}. 
\end{proof}

The following lemma shows that the Euclidean size of the metric ball is tightly controlled by the circle average process.

\begin{lemma}\label{lem:metric-ball-subset-Euclidean-ball}
    Let $(\BC, \Phi; 0, \infty)$ be the circle average embedding of a $\gamma$-LQG cone. For each $t > 0$, write $R_t \defeq \sup\{r > 0 : \Phi_r(0) + Q\log(r) = (1/\xi)\log(t)\}$. Then
    \begin{equation*}
        \BE\!\left\lbrack\left(\inf\{t > 0 : B_1(0; D_\Phi) \subset B_{R_t}(0)\}\right)^s\right\rbrack < \infty, \quad \forall s > 0. 
    \end{equation*}
\end{lemma}

\begin{proof}
    First, we \emph{claim} that $\inf\{t > 0 : B_1(0; D_\Phi) \subset B_{R_t}(0)\}$ has the same law as 
    \begin{equation*}
        \inf\{t > 0 : B_{1/t}(0; D_\Phi) \subset B_1(0)\} = D_\Phi(0, \partial B_1(0))^{-1}. 
    \end{equation*}
    To this end, set $\Phi^t(\bullet) \defeq \Phi(R_t\bullet) + Q\log(R_t) - (1/\xi)\log(t)$. By \cite[Proposition~4.13, (i)]{dms2021mating}, $\Phi^t$ has the same law as $\Phi$. On the other hand, note that $B_{1/t}(0; D_{\Phi^t}) = R_t^{-1} B_1(0; D_\Phi)$ almost surely. This completes the proof of the \emph{claim}. 

    Recall that $\Phi|_{B_1(0)}$ agrees in law with the corresponding restriction of a whole-plane GFF normalized so that its circle average over $\partial B_1(0)$ is zero minus $\gamma\log(\lvert\bullet\rvert)$. Thus, it follows immediately from \cite[Proposition~3.12]{WeakLQGMet} that $\BE\lbrack D_\Phi(0, \partial B_1(0))^{-s}\rbrack < \infty$ for all $s > 0$. This completes the proof of \Cref{lem:metric-ball-subset-Euclidean-ball}. 
\end{proof}

The following lemma shows that if a doubly connected domain surrounding the origin has its conformal modulus bounded from below, then its ``eccentricity'' is bounded from above by a constant depending solely on this modulus.

\begin{lemma}\label{lem:Beurling}
    For each $m > 0$, there exists $b = b(m) \in (0, 1)$ such that the following is true: Let $A \subset \BC$ be a doubly connected domain such that the origin is contained in the bounded connected component of $\BC \setminus A$. Suppose that the conformal modulus of $A$ is at least $m$. Then there exists $x > 0$ such that 
    \begin{itemize}
        \item the planar Brownian motion starting from $x$ hits the inner boundary of $A$ before the outer boundary with probability $1/2$, 
        \item $B_{bx}(x) \subset A$, 
        \item $B_{bx}(0)$ does not intersect the outer boundary of $A$, and
        \item the inner boundary of $A$ is contained in $B_{b^{-1}x}(0)$. 
    \end{itemize}
\end{lemma}

\begin{proof}
    Note that there exists $x > 0$ such that the planar Brownian motion starting from $x$ hits the inner boundary of $A$ before the outer boundary with probability $1/2$. It is well-known that there is a universal constant $c_1 > 0$ such that the planar Brownian motion starting from $x$ disconnects the inner and outer boundaries of $A$ before exiting $A$ with probability at least $c_1 \re^{-\pi/m}$. Write $b_1 \defeq \sup\{b > 0 : B_{bx}(x) \subset A\}$. We observe that $\partial A$ crosses between the inner and outer boundaries of $A_{b_1x,x}(x)$ (i.e., there is a connected component of $\partial A \cap A_{b_1x,x}(x)$ that intersects both $\partial B_{b_1x}(x)$ and $\partial B_x(x)$). Thus, it follows from the Beurling estimate that there is a universal constant $c_2 > 0$ such that the planar Brownian motion starting from $x$ exits $B_x(x)$ before exiting $A$ with probability at most $c_2 b_1^{1/2}$. Thus, we conclude that $c_2 b_1^{1/2} \ge c_1 \re^{-\pi/m}$. 
    
    Write $b_2 \defeq \sup\{b > 0 : B_{bx}(0) \text{ does not intersect the outer boundary of } A\}$. Then there exists $y \in B_{b_2x}(0)$ such that the planar Brownian motion starting from $y$ hits the inner boundary of $A$ before the outer boundary with probability $1/2$. This implies that the planar Brownian motion starting from $y$ disconnects the inner and outer boundaries of $A$ before exiting $A$ with probability at least $c_1 \re^{-\pi/m}$. On the other hand, we observe that the outer boundary of $A$ crosses between the inner and outer boundaries of $A_{2b_2x,(1 - b_2)x}(y)$. Thus, it follows from the Beurling estimate that there is a universal constant $c_3 > 0$ such that the planar Brownian motion starting from $y$ exits $B_{(1 - b_2)x}(y)$ before hitting the outer boundary of $A$ with probability at most $c_3 b_2^{1/2}$. This implies that $c_3 b_2^{1/2} \ge c_1 \re^{-\pi/m}$. 
    
    Write $b_3 \defeq \sup\{b > 0 : \text{the inner boundary of } A \text{ is contained in } B_{b^{-1}x}(0)\}$. We observe that the inner boundary of $A$ crosses between the inner and outer boundaries of $A_{x,(b_3^{-1} - 1)x}(x)$. Thus, it follows from the Beurling estimate that there is a universal constant $c_4 > 0$ such that the planar Brownian motion starting from $x$ exits $B_{(b_3^{-1} - 1)x}(x)$ before hitting the inner boundary of $A$ with probability at most $c_4 b_3^{1/2}$. This implies that $c_4 b_3^{1/2} \ge c_1 \re^{-\pi/m}$. This completes the proof of \Cref{lem:Beurling}. 
\end{proof}

\begin{lemma}\label{lem:good-band}
    Fix a smooth bump function $\psi$ supported on $B_1(0)$ with $\int \psi(z) \, \rd z = 1$. Then for each $\alpha > 0$, there exists $m = m(\alpha) > 0$, $M = M(\alpha) > 0$, and $C = C(\alpha) > 0$ such that the following is true: Let $(\BC, \Phi; 0, \infty)$ be a $\gamma$-LQG cone. For each $t > 0$, write $E_t = E_t(m, M)$ for the event that the following are true:
    \begin{enumerate}
        \item\label{it:good-band-0} The conformal modulus of $A_{t/\re,t}^\bullet(0; D_\Phi)$ is at least $m$. 
        \item\label{it:good-band-1} For each embedding $A_{t/\re,t}^\bullet(0; D_\Phi) \subset \BC$ such that the origin is contained in the bounded connected component of $\BC \setminus A_{t/\re,t}^\bullet(0; D_\Phi)$, the following is true: Let $b = b(m) \in (0, 1)$ be as in \Cref{lem:Beurling}. Then there exists $c = c(m) \in (0, b)$ such that for each $x > 0$ such that 
        \begin{itemize}
            \item the planar Brownian motion starting from $x$ hits the inner boundary of $A_{t/\re,t}^\bullet(0; D_\Phi)$ before the outer boundary with probability $1/2$, and
            \item $B_{bx}(x) \subset A_{t/\re,t}^\bullet(0; D_\Phi)$, 
        \end{itemize}
        we have 
        \begin{equation*}
            \langle\Phi, \psi_{x,c}\rangle \in [(1/\xi)\log(t) - Q\log(x) - M, (1/\xi)\log(t) - Q\log(x) + M], 
        \end{equation*}
        where $\psi_{x,c}(\bullet) \defeq (cx)^{-2}\psi((cx)^{-1}(\bullet - x))$. 
    \end{enumerate}
    Let $\{t_j\}_{j \in \BN}$ be a decreasing sequence of positive real numbers such that $t_{j + 1} \le t_j/\re^2$ for all $j \in \BN$. Then 
    \begin{equation*}
        \BP\lbrack\text{there exists } j \in [1, n]_\BZ \text{ such that } E_{t_j/\re} \text{ occurs}\rbrack \ge 1 - C\re^{-\alpha n}, \quad \forall n \in \BN. 
    \end{equation*}
\end{lemma}

\begin{proof}[Proof of \Cref{lem:good-band}]
    It is clear that the event $E_t$ is almost surely determined by the $\gamma$-LQG surface parameterized by $A_{t/\re,t}^\bullet(0; D_\Phi)$. For each $t \ge 0$, write $Y_{-t} \defeq \SN_\Phi(\partial B_t^\bullet(0; D_\Phi))$. For each $n \in \BN$, write $j_n$ for the $n$-th smallest $j \in \BN$ for which $Y_{-t_j} \le At_j^2$. By \Cref{lem:CSBP-independence}, we may choose $A$ to be sufficiently large so that $\BP\lbrack j_{\lfloor n/2\rfloor} \le n\rbrack = 1 - O(\re^{-\alpha n})$ as $n \to \infty$. 
    
    We \emph{claim} that for each $p \in (0, 1)$, we may choose $m$ to be sufficiently small and $M$ to be sufficiently large so that 
    \begin{equation*}
        \BP\lbrack E_{t_j/\re} \mid Y_{-t_j} = at_j^2\rbrack \ge p, \quad \forall j \in \BN, \ \forall a \in [0, A]. 
    \end{equation*}
    By the scaling property, $\BP\lbrack E_{t/\re} \mid Y_{-t} = at^2\rbrack$ does not depend on the choice of $t$. Thus, it suffices to consider the case $t = 1$. The conditional law of the process $\{Y_{s - 1}\}_{s \in [0, 1]}$ given $Y_{-1}$ is that of a $3/2$-stable CSBP starting from $Y_{-1}$ and hitting zero at time one. This implies that the assignment 
    \begin{equation*}
        a \mapsto (\text{the conditional law of } Y_{-1/\re} \text{ given } Y_{-1} = a)
    \end{equation*}
    for $a \in \BR_{\ge 0}$ is continuous with respect to the total variation distance. Moreover, $A_{1/\re^2,1/\re}^\bullet(0; D_\Phi)$ is conditionally independent of $Y_{-1}$ given $Y_{-1/\re}$. This implies that the assignment 
    \begin{equation*}
        a \mapsto (\text{the conditional law of } A_{1/\re^2,1/\re}^\bullet(0; D_\Phi) \text{ given } Y_{-1} = a)
    \end{equation*}
    for $a \in \BR_{\ge 0}$ is continuous with respect to the total variation distance. In a similar vein, for each $a \ge 0$, the conditional law of $A_{1/\re^2,1/\re}^\bullet(0; D_\Phi)$ given $Y_{-1} = a$ and the marginal law of $A_{1/\re^2,1/\re}^\bullet(0; D_\Phi)$ are mutually absolutely continuous (since the conditional law of $Y_{-1/\re}$ given $Y_{-1} = a$ and the marginal law of $Y_{-1/\re}$ are both mutually absolutely continuous with respect to the Lebesgue measure on $\BR_{\ge 0}$). On the other hand, it is clear that $\BP\lbrack E_{1/\re}\rbrack \to 1$ as $m \to 0$ and $M \to \infty$. Indeed, it is clear that $\BP\lbrack\eqref{it:good-band-0}\rbrack \to 1$ as $m \to 0$, and it follows immediately from \Cref{lem:canonical-LF} that $\BP\lbrack\eqref{it:good-band-1}\rbrack \to 1$ as $M \to \infty$. This implies that for each $a \ge 0$, we have $\BP\lbrack E_{1/\re} \mid Y_{-1} = a\rbrack \to 1$ as $m \to 0$ and $M \to \infty$. By the above discussion, the assignment $\BR_{\ge 0} \to [0, 1] \colon a \mapsto \BP\lbrack E_{1/\re} \mid Y_{-1} = a\rbrack$ is continuous. Thus, we conclude that 
    \begin{equation*}
        \inf_{a \in [0, A]} \BP\lbrack E_{1/\re} \mid Y_{-1} = a\rbrack \to 1 \quad \text{as } m \to 0 \text{ and } M \to \infty. 
    \end{equation*}
    This completes the proof of the \emph{claim}. 

    The $\gamma$-LQG surfaces parameterized by $B_{t_j}^\bullet(0; D_\Phi)$ and $\BC \setminus B_{t_j}^\bullet(0; D_\Phi)$ are conditionally independent given $Y_{-t_j}$. This implies that $\#\{k \in [1, n]_\BZ : E_{t_{j_k}} \text{ occurs}\}$ stochastically dominates a binomial random variable with $n$ trials and success probability $p$. Thus, by choosing $p$ to be sufficiently close to one, 
    \begin{align*}
        &\BP\lbrack\text{there does not exist } j \in [1, n]_\BZ \text{ such that } E_{t_j/\re} \text{ occurs}\rbrack \\
        &\le \BP\lbrack j_{\lfloor n/2\rfloor} > n\rbrack + \BP\lbrack\text{there does not exist } k \in [1, n/2]_\BZ \text{ such that } E_{t_{j_k}} \text{ occurs}\rbrack \\
        &\le O(\re^{-\alpha n}) + (1 - p)^{\lfloor n/2\rfloor} = O(\re^{-\alpha n}) \quad \text{as } n \to \infty. 
    \end{align*}
    This completes the proof of \Cref{lem:good-band}.
\end{proof}

\begin{remark}
    It follows immediately from a similar argument to the argument applied in the proof of \Cref{lem:good-band} that the following is true: For each $\lambda \in (0, 1)$, $\alpha > 0$, and $b \in (0, 1)$, there exists $p = p(\lambda, \alpha, b) \in (0, 1)$ and $C = C(\lambda, \alpha, b) > 0$ such that the following is true: Let $(\BC, \Phi; 0, \infty)$ be a $\gamma$-LQG cone. Let $\{t_j\}_{j \in \BN}$ be a decreasing sequence of positive real numbers such that $t_{j + 1}/t_j \le \lambda$ for all $j \in \BN$. Let $\{E_j\}_{j \in \BN}$ be a sequence of events such that each $E_j$ is almost surely determined by the $\gamma$-LQG surface parameterized by $A_{t_{j + 1},t_j}^\bullet(0; D_\Phi)$. Suppose that $\BP\lbrack E_j\rbrack \ge p$ for all $j \in \BN$. Then 
    \begin{equation*}
        \BP\lbrack\#\{j \in [1, n]_\BZ : E_j \text{ occurs}\} \ge bn\rbrack \ge 1 - C\re^{-\alpha n}, \quad \forall n \in \BN. 
    \end{equation*}
\end{remark}

\begin{lemma}\label{lem:trivial}
    For each $n \in \BN$, there exists $a_n > 0$ such that the following is true: Let $\phi \colon B_1(0) \to \BC$ be a univalent function such that $\lvert\phi^\prime(0)\rvert \le 1$. Then $\lvert\phi^{(n)}(0)\rvert \le a_n$.
\end{lemma}

\begin{proof}
    Recall from the standard gradient estimate for harmonic functions that there is a universal constant $C > 0$ such that for each harmonic function $u \colon B_1(0) \to \BR$, 
    \begin{equation*}
        \lvert\nabla u(z)\rvert \le C (1 - \lvert z\rvert)^{-1} \sup_{x \in B_1(0)} \lvert u(x)\rvert, \quad \forall z \in B_1(0). 
    \end{equation*}
    By Koebe's distortion theorem, there exists $a_1 > 0$ such that $\lvert\phi^\prime(z)\rvert \le a_1$ for all $z \in B_{1/2}(0)$. By applying the above gradient estimate repeatedly, we obtain $a_n > 0$ such that $\lvert\phi^{(n)}(z)\rvert \le a_n$ for all $z \in B_{2^{-n}}(0)$, hence that $\lvert\phi^{(n)}(0)\rvert \le a_n$. This completes the proof of \Cref{lem:trivial}.
\end{proof}

\begin{lemma}\label{lem:canonical-LF}
    In the notation of \Cref{lem:good-band}, let $A \subset \BC$ be a doubly connected domain of conformal modulus at least $m$. Consider the collection of functions of the form $\lvert\phi^\prime\rvert^2 (\psi_{x,c} \circ \phi)$, where $\phi \colon A \to \phi(A)$ is a conformal mapping such that
    \begin{itemize}
        \item the origin is contained in the bounded connected component of $\BC \setminus \phi(A)$,
        \item the planar Brownian motion starting from $x$ hits the inner boundary of $\phi(A)$ before the outer boundary with probability $1/2$, and
        \item $B_{bx}(x) \subset \phi(A)$. 
    \end{itemize}
    Then there exists $c = c(m) \in (0, b)$ such that this collection is relatively compact in the space of test functions $C_c^\infty(A)$. 
\end{lemma}

\begin{proof}
    Recall that a subset $S \subset C_c^\infty(A)$ is relatively compact if
    \begin{itemize}
        \item there is a compact subset $K \subset A$ such that $\mathop{\mathrm{supp}}(\psi) \subset K$ for all $\psi \in S$, and
        \item for each multi-index $\alpha$, there exists $C_\alpha > 0$ such that $\lVert\partial^\alpha\psi\rVert_\infty \le C_\alpha$ for all $\psi \in S$. 
    \end{itemize}
    We may assume without loss of generality that $A = A_{\re^{-\pi m},\re^{\pi m}}(0)$. Moreover, by the scaling property, we may assume without loss of generality that $x = 1$. Thus, $\phi^{-1}(1) \in \partial B_1(0)$. By Koebe's distortion theorem, $\sup_{z \in B_{b/2}(1)} \lvert(\phi^{-1})^\prime(z)\rvert$ is bounded above by a constant depending only on $m$ and $b$. This implies that there exists $c = c(m) \in (0, b/2)$ such that $\phi^{-1}(B_c(1)) \subset A_{\re^{-\pi m/2},\re^{\pi m/2}}(0)$. Note that $\mathop{\mathrm{supp}}(\lvert\phi^\prime\rvert^2 (\psi_{1,c} \circ \phi)) \subset \phi^{-1}(B_c(1))$. Thus, we obtain that $\mathop{\mathrm{supp}}(\lvert\phi^\prime\rvert^2 (\psi_{1,c} \circ \phi)) \subset A_{\re^{-\pi m/2},\re^{\pi m/2}}(0)$. Fix a multi-index $\alpha$. Note that there exists $C_\alpha = C_\alpha(\psi, c)$ and a polynomial $P_\alpha \in \BR[X]$ such that 
    \begin{equation*}
        \left\lVert\partial^\alpha\!\left(\lvert\phi^\prime\rvert^2(\psi_{1,c} \circ \phi)\right)\right\rVert_\infty \le C_\alpha P_\alpha\!\left(\sup_{k \in [1, \lvert\alpha\rvert + 1]_\BZ} \lVert\phi^{(k)}|_{\phi^{-1}(B_c(1))}\rVert_\infty\right).
    \end{equation*}
    Thus, it suffices to show that for each $k \in \BN$, there exists $a_k = a_k(m) > 0$ such that $\lVert\phi^{(k)}|_{\phi^{-1}(B_c(1))}\rVert_\infty \le a_k$. Fix $y \in \phi^{-1}(B_c(1))$. It follows from the above discussion that $y \in A_{\re^{-\pi m/2},\re^{\pi m/2}}(0)$. In particular, there exists $a_0 = a_0(m) > 0$ such that $B_{a_0}(y) \subset A$. Since $0 \notin \phi(B_{a_0}(y))$ and $\lvert\phi(y)\rvert < 1 + c$, it follows from Koebe's quarter theorem that $\lvert\phi^\prime(y)\rvert \le a_1$ for some $a_1 = a_1(m) > 0$. Combining this with \Cref{lem:trivial}, we conclude that there exists $a_k = a_k(a_0, a_1) > 0$ such that $\lvert\phi^{(k)}(y)\rvert \le a_k$. This completes the proof of \Cref{lem:canonical-LF}.  
\end{proof}

The idea of the proof of \Cref{lem:conformal-moduli-marginal} is as follows. We first fix a sufficiently small $\zeta > 0$. Recall from \Cref{lem:metric-ball-subset-Euclidean-ball} that we can bound the Euclidean size of the metric ball. Suppose $T > 0$ is a deterministic time such that $\Phi_{\re^{-T}}(0) - QT \ge -((1 - \zeta)/\xi)\log(1/\varepsilon)$. Conditional on this, the process $\{\Phi_{\re^{-t}}(0)\}_{t \ge T}$ is a Brownian motion with drift $\gamma$ starting from $\Phi_{\re^{-T}}(0)$. If we define $R \defeq \sup\{r \in (0, \re^{-T}) : \Phi_r(0) + Q\log(r) = -((1 - \zeta)/\xi)\log(1/\varepsilon)\}$, then $B_\varepsilon(0; D_\Phi) \subset B_R(0)$ holds with superpolynomially high probability as $\varepsilon \to 0$. Combining \Cref{lem:metric-ball-subset-Euclidean-ball} with the Laplace transform of the first hitting time for a drifted Brownian motion, we obtain
\begin{equation*}
    \BE\lbrack\outradius(B_\varepsilon^\bullet(0; D_\Phi))^p \wedge 1\rbrack = O(\varepsilon^q) \quad \text{as } \varepsilon \to 0
\end{equation*}
for some constant $q = q(p) > 4$.

To estimate the conformal modulus, we also need to bound the inradius of $B_1^\bullet(0; D_\Phi)$. We do this by applying \Cref{lem:good-band} to the metric band $A_{\varepsilon^\zeta,1}^\bullet(0; D_\Phi)$. This implies that, with high probability, the band has a conformal modulus bounded below, and the average of $\Phi$ over a small Euclidean ball inside the band is bounded. Next, \Cref{lem:Beurling} allows us to control the inradius of $B_1^\bullet(0; D_\Phi)$. It guarantees there is a Euclidean circle centered at the origin that the metric band $A_{\varepsilon^\zeta,1}^\bullet(0; D_\Phi)$ stays close to. The difference between the circle average of $\Phi$ and its average over this small Euclidean ball has a Gaussian tail and is therefore well controlled. Finally, we apply the deterministic radius $\re^{-T}$ discussed above to all possible spatial scales of this Euclidean circle.

\begin{proof}[Proof of \Cref{lem:conformal-moduli-marginal}]
    We may assume without loss of generality that $(\BC, \Phi; 0, \infty)$ is the circle average embedding. By \Cref{lem:Teichmueller}, it suffices to show that
    \begin{equation*}
        \BE\!\left\lbrack\left(\frac{\outradius(B_\varepsilon^\bullet(0; D_\Phi))}{\inradius(B_1^\bullet(0; D_\Phi))}\right)^p \wedge 1\right\rbrack = O(\varepsilon^q) \quad \text{as } \varepsilon \to 0,
    \end{equation*}
    where $\outradius(B_\varepsilon^\bullet(0; D_\Phi)) \defeq \inf\{R > 0 : B_\varepsilon^\bullet(0; D_\Phi) \subset B_R(0)\}$ and $\inradius(B_1^\bullet(0; D_\Phi)) \defeq \sup\{R > 0 : B_R(0) \subset B_1^\bullet(0; D_\Phi)\}$. Fix $p > 2$ and $\alpha > 4$. Fix a sufficiently small $\zeta > 0$ to be chosen later. By \Cref{lem:good-band}, there exists $m = m(\zeta, \alpha) > 0$ and $M = M(\zeta, \alpha) > 0$ such that the following is true: Let the events $\{E^j \defeq E_{\re^{-j}} = E_{\re^{-j}}(m, M)\}_{j \in \BN}$ be as in \Cref{lem:good-band}. Then it holds with probability $1 - O(\varepsilon^\alpha)$ as $\varepsilon \to 0$ that there exists $j \in [\zeta\log(1/\varepsilon), 2\zeta\log(1/\varepsilon)]_\BZ$ such that $E^j$ occurs. Write $F_1(\varepsilon)$ for this event, i.e., $F_1(\varepsilon) \defeq \bigcup_{j \in [\zeta\log(1/\varepsilon), 2\zeta\log(1/\varepsilon)]_\BZ} E^j$. By definition, on the event $F_1(\varepsilon)$, there exists $j \in [\zeta\log(1/\varepsilon), 2\zeta\log(1/\varepsilon)]_\BZ$ and $x > 0$ such that 
    \begin{itemize}
        \item $B_{bx}(x) \subset A_{\re^{- j - 1},\re^{-j}}^\bullet(0; D_\Phi)$, 
        \item $\inradius(B_1^\bullet(0; D_\Phi)) \ge \inradius(B_{\re^{-j}}^\bullet(0; D_\Phi)) \ge bx$, and
        \item $\langle\Phi, \psi_{x,c}\rangle \in [-j/\xi - Q\log(x) - M, -j/\xi - Q\log(x) + M]$. 
    \end{itemize}
    Set $\lambda \defeq 1 - b$. By possibly decreasing $b$, we may assume without loss of generality that $x = \lambda^k$ for some $k \in \BZ$. 

    Write $F_2(\varepsilon)$ for the event that $B_{\varepsilon^\zeta}(0; D_\Phi) \subset B_1(0)$. Since $\BE\lbrack D_\Phi(0, \partial B_1(0))^{-s}\rbrack < \infty$ for all $s > 0$ (cf.~\cite[Proposition~3.12]{WeakLQGMet}), it follows that $\BP\lbrack F_2(\varepsilon)\rbrack = 1 - O(\varepsilon^\infty)$ as $\varepsilon \to 0$. On the other hand, it follows from \cite[Proposition~3.18]{WeakLQGMet} that we may choose a sufficiently large $A = A(\zeta, \alpha, b) > 0$ such that it holds with probability $1 - O(\varepsilon^\alpha)$ as $\varepsilon \to 0$ that $B_{\varepsilon^{A\log(1/\lambda)}}(0) \subset B_{\varepsilon^{2\zeta}}(0; D_\Phi)$. Write $F_3(\varepsilon)$ for this event. In particular, on the event $F_1(\varepsilon) \cap F_2(\varepsilon) \cap F_3(\varepsilon)$, if $x = \lambda^k$ is as in the preceding paragraph, then $k \in [1, A\log(1/\varepsilon)]_\BZ$. 

    Since $\Phi|_{B_1(0)}$ agrees in law with the corresponding restriction of a whole-plane GFF normalized so that its circle average over $\partial B_1(0)$ is zero minus $\gamma\log(\lvert\bullet\rvert)$, this implies that for $x \in (0, 1)$, the random variable $\Phi_x(0) - \langle\Phi, \psi_{x,c}\rangle$ is Gaussian with mean and variance not depending on $x$. Thus, it follows from the standard Gaussian tail estimate and a union bound that it holds with superpolynomially high probability as $\varepsilon \to 0$ that $\lvert\Phi_{\lambda^k}(0) - \langle\Phi, \psi_{\lambda^k,b}\rangle\rvert \le (\zeta/\xi)\log(1/\varepsilon)$ for all $k \in [1, A\log(1/\varepsilon)]_\BZ$. Write $F_4(\varepsilon)$ for this event. In particular, we conclude that, on the event $F_1(\varepsilon) \cap F_2(\varepsilon) \cap F_3(\varepsilon) \cap F_4(\varepsilon)$, there exists $k \in [1, A\log(1/\varepsilon)]_\BZ$ such that 
    \begin{itemize}
        \item $\inradius(B_1^\bullet(0; D_\Phi)) \ge b\lambda^k$, and
        \item $\Phi_{\lambda^k}(0) \in [Qk\log(1/\lambda) - M - (3\zeta/\xi)\log(1/\varepsilon), Qk\log(1/\lambda) + M]$. 
    \end{itemize}
    For each $k \in \BN$, write $\Phi^k(\bullet) \defeq \Phi(\lambda^k\bullet) - Qk\log(1/\lambda)$; $F^k$ for the event that $\Phi_{\lambda^k}(0) \in [Qk\log(1/\lambda) - M - (3\zeta/\xi)\log(1/\varepsilon), Qk\log(1/\lambda) + M]$ (or, equivalently, that $\Phi_1^k(0) \in [- M - (3\zeta/\xi)\log(1/\varepsilon), M]$); $R_{\varepsilon^{1 - \zeta}}^k \defeq \sup\{r \in (0, 1) : \Phi_r(0) + Q\log(r) + ((1 - \zeta)/\xi)\log(1/\varepsilon) = 0\}$. Note that, given $\Phi_1^k(0)$, $\Phi^k|_{B_1(0)}$ agrees in law with the corresponding restriction of a whole-plane GFF normalized so that its circle average over $\partial B_1(0)$ is equal to $\Phi_1^k(0)$ minus $\gamma\log(\lvert\bullet\rvert)$. Thus, we conclude from \Cref{lem:metric-ball-subset-Euclidean-ball} that, almost surely on the event $F^k$, 
    \begin{equation*}
        \BP\!\left\lbrack B_\varepsilon(0; D_{\Phi^k}) \subset B_{R_{\varepsilon^{1 - \zeta}}^k}(0) \ \middle\vert \ \Phi_1^k(0)\right\rbrack = 1 - O(\varepsilon^\infty), 
    \end{equation*}
    at a rate which is uniform in $k$. Moreover, we note that $B_\varepsilon(0; D_\Phi) = \lambda^k B_\varepsilon(0; D_{\Phi^k})$. Write $F_5(\varepsilon)$ for the event that $B_\varepsilon(0; D_{\Phi^k}) \subset B_{R_{\varepsilon^{1 - \zeta}}^k}(0)$ for all $k \in [1, A\log(1/\varepsilon)]_\BZ$ and $F(\varepsilon) \defeq F_1(\varepsilon) \cap F_2(\varepsilon) \cap F_3(\varepsilon) \cap F_4(\varepsilon) \cap F_5(\varepsilon)$. It follows from the above discussion that
    \begin{align*}
        \BE\!\left\lbrack\left(\frac{\outradius(B_\varepsilon^\bullet(0; D_\Phi))}{\inradius(B_1^\bullet(0; D_\Phi))}\right)^p \wedge 1\right\rbrack &\le \BE\!\left\lbrack\left(\frac{\outradius(B_\varepsilon^\bullet(0; D_\Phi))}{\inradius(B_1^\bullet(0; D_\Phi))}\right)^p \mathbf 1_{F(\varepsilon)}\right\rbrack + \BP[F(\varepsilon)^c] \\
        &\le b^{-p} \sum_{k \in [1, A\log(1/\varepsilon)]_\BZ} \BE\!\left\lbrack (R_{\varepsilon^{1 - \zeta}}^k)^p \mathbf 1_{F^k}\right\rbrack + O(\varepsilon^\alpha) \quad \text{as } \varepsilon \to 0. 
    \end{align*}
    Finally, since, given $\Phi_1^k(0)$, $\log(1/R_{\varepsilon^{1 - \zeta}}^k)$ has the law of the first time at which a Brownian motion with drift $-(Q - \gamma)$ starting from $\Phi_1^k(0)$ hits $-((1 - \zeta)/\xi)\log(1/\varepsilon)$, it follows that, almost surely on the event $F^k$, 
    \begin{align*}
        \BE\!\left\lbrack (R_{\varepsilon^{1 - \zeta}}^k)^p \ \middle\vert \ \Phi_1^k(0)\right\rbrack &= \exp\!\left(-\left(\sqrt{(Q - \gamma)^2 + 2p} - (Q - \gamma)\right)\left(((1 - \zeta)/\xi)\log(1/\varepsilon) + \Phi_1^k(0)\right)\right) \\
        &\le \exp\!\left(-\left(\sqrt{(Q - \gamma)^2 + 2p} - (Q - \gamma)\right)\left(((1 - 4\zeta)/\xi)\log(1/\varepsilon) - M\right)\right) \\
        &\le \varepsilon^{(1 - 4\zeta)(\sqrt{(Q - \gamma)^2 + 2p} - (Q - \gamma))/\xi + o(1)} \quad \text{as } \varepsilon \to 0, 
    \end{align*}
    at a rate which is uniform in $k$. By choosing $\zeta > 0$ to be sufficiently small, we have 
    \begin{equation*}
        \left.(1 - 4\zeta)\left(\sqrt{(Q - \gamma)^2 + 2p} - (Q - \gamma)\right)\middle/\xi\right. > \left.\left(\sqrt{(Q - \gamma)^2 + 4} - (Q - \gamma)\right)\middle/\xi\right. = d_\gamma. 
    \end{equation*}
    This completes the proof of \Cref{lem:conformal-moduli-marginal}. 
\end{proof}

\begin{proof}[Proof of \Cref{lem:conformal-moduli-band}]
    Fix $p > 2$. Let $(\BC, \Phi; 0, \infty)$ be a $\gamma$-LQG cone. By \Cref{lem:conformal-moduli-marginal} and the scaling property, there exists $q > 4$ such that $\BE\lbrack\re^{-2\pi m_{1,T} p}\rbrack = O(T^{-q})$ as $T \to \infty$, where $m_{1,T}$ denotes the conformal modulus of $A_{1,T}^\bullet(0; D_\Phi)$. Write $\tau \defeq \inf\{t \ge 0 : \SN_\Phi(\partial B_t^\bullet(0; D_\Phi)) = 1\} \wedge 1$. Note that $\tau$ is a stopping time for the filtration generated by 
    \begin{equation*}
        \{\text{the $\gamma$-LQG surface parameterized by } B_t^\bullet(0; D_\Phi)\}_{t \ge 0}, 
    \end{equation*}
    and, given $\tau$ and the event that $\tau < 1$, the $\gamma$-LQG surface parameterized by $A_{\tau,T}^\bullet(0; D_\Phi)$ is conditionally a metric band of cone type with inner boundary length one and width $T - \tau$. This implies that
    \begin{equation*}
        \BE\lbrack\re^{-2\pi m_{1,T} p}\rbrack \ge \BE\lbrack\re^{-2\pi m_{\tau,T} p}\rbrack \ge \BE\lbrack\re^{-2\pi m_{\tau,T} p} \mathbf 1_{\{\tau < 1\}}\rbrack \ge \BP\lbrack\tau < 1\rbrack \BE\lbrack\re^{-2\pi m_T p}\rbrack. 
    \end{equation*}
    Since $\BP\lbrack\tau < 1\rbrack > 0$, it follows that 
    \begin{equation*}
        \BE\lbrack\re^{-2\pi m_T p}\rbrack \le \BP\lbrack\tau < 1\rbrack^{-1} \BE\lbrack\re^{-2\pi m_{1,T} p}\rbrack = O(T^{-q}) \quad \text{as } T \to \infty.
    \end{equation*}
    This completes the proof of \Cref{lem:conformal-moduli-band}. 
\end{proof}

\section{Background on Gromov hyperbolic geometry}\label{section:background-Gromov}

In the present section, we review standard material on Gromov hyperbolic geometry (cf., e.g., \cite{MR919829,MR1075994,MR1086648,MR1744486,MR2275649,MR2327160,MR2662522}). 

Let $(X, d)$ be a metric space. The \emph{Gromov product} is given by
\begin{equation*}
    (x, y)_o \defeq \frac12(d(x, o) + d(y, o) - d(x, y)), \quad \forall o, x, y \in X. 
\end{equation*}
The metric space $(X, d)$ is called \emph{Gromov ($\delta$-)hyperbolic} if there exists $\delta \ge 0$ such that
\begin{equation*}
    (x, y)_o \ge (x, z)_o \wedge (y, z)_o - \delta, \quad \forall o, x, y, z \in X. 
\end{equation*}
One verifies immediately that if there exists $o \in X$ such that $(x, y)_o \ge (x, z)_o \wedge (y, z)_o - \delta$ for all $x, y, z \in X$, then $(X, d)$ is Gromov $2\delta$-hyperbolic. If $(X, d)$ is geodesic, then it is Gromov hyperbolic if and only if there exists $\delta^\prime \ge 0$ such that $P_{x,y} \subset B_{\delta^\prime}(P_{x,z} \cup P_{y,z}; d)$ for all $x, y, z \in X$, where $P_{x,y}$ denotes any $d$-geodesic connecting $x$ and $y$. 

Let $(X, d_X)$ and $(Y, d_Y)$ be metric spaces. A mapping $F \colon (X, d_X) \to (Y, d_Y)$ is called a \emph{quasi-isometry} if there exists $C \ge 1$ such that 
\begin{itemize}
    \item $C^{-1} d_X(x, y) - C \le d_Y(F(x), F(y)) \le C d_X(x, y) + C$ for all $x, y \in X$, and
    \item $\inf_{x \in X} d_Y(F(x), y) \le C$ for all $y \in Y$. 
\end{itemize}
The metric spaces $(X, d_X)$ and $(Y, d_Y)$ are called \emph{quasi-isometric} if there exists a quasi-isometry $F \colon (X, d_X) \to (Y, d_Y)$. If $(X, d_X)$ and $(Y, d_Y)$ are geodesic and quasi-isometric, then $(X, d_X)$ is Gromov hyperbolic if and only if $(Y, d_Y)$ is Gromov hyperbolic.

Let $(X, d)$ be a Gromov $\delta$-hyperbolic space. The \emph{Gromov boundary} of $(X, d)$ is given by
\begin{equation*}
    \partial_\infty(X, d) \defeq \{\{x_j\}_{j \in \BN} \subset X : (x_j, x_k)_o \to \infty \text{ as } j, k \to \infty\}/\sim, 
\end{equation*}
where $\{x_j\}_{j \in \BN} \sim \{y_j\}_{j \in \BN}$ if $(x_j, y_j)_o \to \infty$ as $j \to \infty$. The Gromov boundary $\partial_\infty(X, d)$ does not depend on the choice of $o$. If $(X, d)$ is proper and geodesic, then $\partial_\infty(X, d)$ may be naturally identified with
\begin{equation*}
    \{P \colon [0, \infty) \to X \text{ is a } d\text{-geodesic ray}\}/\sim, 
\end{equation*}
where $P \sim Q$ if the $d$-Hausdorff distance between $P$ and $Q$ is finite. The Gromov product extends naturally to the Gromov boundary
\begin{equation*}
    (a, b)_o \defeq \inf_{\{x_j\}_{j \in \BN}, \{y_j\}_{j \in \BN}} \liminf_{j \to \infty} (x_j, y_j)_o \in [0, \infty], \quad \forall a, b \in \partial_\infty(X, d), 
\end{equation*}
where the infimum is over all $\{x_j\}_{j \in \BN}, \{y_j\}_{j \in \BN} \subset X$ with $\{x_j\}_{j \in \BN} \sim a$ and $\{y_j\}_{j \in \BN} \sim b$. (Here, we note that $(a, b)_o = \infty$ if and only if $a = b$.) Moreover, for each $\{x_j\}_{j \in \BN}, \{y_j\}_{j \in \BN} \subset X$ with $\{x_j\}_{j \in \BN} \sim a$ and $\{y_j\}_{j \in \BN} \sim b$, 
\begin{equation}\label{eq:product-boundary}
    (a, b)_o \le \liminf_{j \to \infty} (x_j, y_j)_o \le \limsup_{j \to \infty} (x_j, y_j)_o \le (a, b)_o + 2\delta. 
\end{equation}

Let $(X, d)$ be a Gromov $\delta$-hyperbolic space. A metric $D$ on the Gromov boundary $\partial_\infty(X, d)$ is called \emph{visual} with parameter $\varepsilon > 0$ if there exists $C \ge 1$ such that 
\begin{equation*}
    C^{-1} \re^{-\varepsilon(a, b)_o} \le D(a, b) \le C \re^{-\varepsilon(a, b)_o}, \quad \forall a, b \in \partial_\infty(X, d). 
\end{equation*}
There exists $\varepsilon_\ast = \varepsilon_\ast(\delta) > 0$ such that for each $\varepsilon \in (0, \varepsilon_\ast]$, there exists a visual metric on $\partial_\infty(X, d)$ with parameter $\varepsilon$. Let $D$ (resp.~$D^\prime$) be a visual metric on $\partial_\infty(X, d)$ with parameter $\varepsilon$ (resp.~$\varepsilon^\prime$). Then there exists $C \ge 1$ such that 
\begin{equation*}
    C^{-1} D(a, b)^{\varepsilon^\prime/\varepsilon} \le D^\prime(a, b) \le C D(a, b)^{\varepsilon^\prime/\varepsilon}, \quad \forall a, b \in \partial_\infty(X, d). 
\end{equation*}
In particular, any two visual metrics on $\partial_\infty(X, d)$ are quasisymmetrically equivalent, i.e., $\partial_\infty(X, d)$ is equipped with the conformal gauge of visual metrics. 

Let $(X, d)$ be a proper and geodesic Gromov $\delta$-hyperbolic space. Then there exists $\varepsilon_\ast = \varepsilon_\ast(\delta) > 0$ such that for each $\varepsilon \in (0, \varepsilon_\ast]$, the following is true (cf.~\cite{MR1829896}): Write 
\begin{equation*}
    \phi_\varepsilon(x) \defeq \re^{-\varepsilon d(o, x)}, \quad \forall x \in X; \quad d_\varepsilon(x, y) \defeq \inf_P \int_a^b \phi_\varepsilon(P(t)) \, \rd t, \quad \forall x, y \in X, 
\end{equation*}
where the infimum is over all paths $P \colon [a, b] \to X$ parameterized by $d$-length. Write $(\overline X_\varepsilon, d_\varepsilon)$ for the completion of $(X, d_\varepsilon)$ and $\partial_\varepsilon X \defeq \overline X_\varepsilon \setminus X$. Then $\partial_\varepsilon X$ may be naturally identified with $\partial_\infty(X, d)$ and $d_\varepsilon$ induces a visual metric on $\partial_\infty(X, d)$ with parameter $\varepsilon$. 

The following is well-known (cf., e.g., \cite[Th\'eor\`eme~3.1]{MR2605327}). 

\begin{theorem}
    Let $(X, d_X)$ and $(Y, d_Y)$ be proper and geodesic Gromov hyperbolic metric spaces. Then every quasi-isometry $(X, d_X) \to (Y, d_Y)$ induces a natural quasisymmetric mapping $\partial_\infty(X, d_X) \to \partial_\infty(Y, d_Y)$. 
\end{theorem}

\section{Hyperbolic fillings}\label{section:hyperbolic-filling}

The present section reviews the technique of hyperbolic fillings, a framework utilized to construct a Gromov hyperbolic space with a prescribed Gromov boundary. This construction was first introduced in \cite{CohLpBesov} (cf.~also \cite{MR2327160,MR3781334,MR4377995}). By employing this technique and adapting the arguments established in \cite{ConformalGauge,CP-ModARCD}, we show that for every admissible weight function defined on the hyperbolic filling, one can construct a metric space quasisymmetrically equivalent to the prescribed one. To ensure the present paper remains self-contained, we include complete proofs of these results. Furthermore, the arguments presented in the present section apply to any compact metric space, without requiring the assumption of the doubling property.

Let $(X, D)$ be a compact metric space with $\diam(X; D) < 1$. 

Fix a sufficiently small parameter $\alpha \in (0, 1)$. Let $A_0 \subset A_1 \subset A_2 \subset \cdots$ be a family of finite subsets of $X$ such that each $A_n$ is a maximal $\alpha^n$-separated subset, i.e., 
\begin{equation*}
    X \subset \bigcup_{x \in A_n} B_{\alpha^n}(x; D) \quad \text{and} \quad D(x, y) \ge \alpha^n, \quad \forall \text{ distinct } x, y \in A_n. 
\end{equation*}
We shall write $V_n \defeq \{(x, n) : x \in A_n\}$ and $V \defeq \coprod_{n \ge 0} V_n$. For distinct $(x, n), (y, n) \in V_n$, we shall write $(x, n) \sim (y, n)$ if $B_{4\alpha^n}(x; D)$ and $B_{4\alpha^n}(y; D)$ intersect. For $(x, n) \in V_n$ and $(y, n + 1) \in V_{n + 1}$, we shall write $(x, n) \sim (y, n + 1)$ if $B_{\alpha^n}(x; D)$ and $B_{\alpha^{n + 1}}(y; D)$ intersect. We shall refer to the edges $(x, n) \sim (y, n)$ as \emph{horizontal}; we shall refer to the edges $(x, n) \sim (y, n + 1)$ as \emph{vertical}. Since $\diam(X; D) < 1$, it follows that $\#A_0 = 1$. We shall write $o \in V_0$ for the root. By convention, we shall write $u \sim u$ for all $u \in V$.

By choosing $\alpha$ to be sufficiently small, we may assume without loss of generality that the following is true: Let $u_0, u_1, \cdots, u_{100} \in V_n$ such that $u_0 \sim u_1 \sim \cdots \sim u_{100}$. Let $v_0, v_{100} \in V_{n - 1}$ such that $u_0 \sim v_0$ and $u_{100} \sim v_{100}$. Then $v_0 \sim v_{100}$.

We shall write $(G, d_G)$ for the metric graph associated with $(V, \sim)$; 
\begin{equation*}
    (u, v)_G \defeq \frac12(d_G(o, u) + d_G(o, v) - d_G(u, v)), \quad \forall u, v \in G
\end{equation*}
for the Gromov product. 

\begin{lemma}\label{lem:hyperbolic-filling}
    \begin{enumerate}
        \item\label{it:hyperbolic-filling-0} There exists $\delta = \delta(\alpha) \ge 0$ such that $(G, d_G)$ is Gromov $\delta$-hyperbolic.
        \item\label{it:hyperbolic-filling-1} The Gromov boundary $\partial_\infty(G, d_G)$ may be naturally identified with $X$, and $D$ induces a visual metric on $\partial_\infty(G, d_G)$ with parameter $\log(1/\alpha)$. 
    \end{enumerate}
\end{lemma}

\begin{proof}
    We follow the argument applied in the proof of \cite[Theorem~3.4]{MR4377995}. We \emph{claim} that there exists $C = C(\alpha) \ge 1$ such that 
    \begin{equation}\label{eq:hyperbolic-filling-proof-2}
        C^{-1} \alpha^{(u, v)_G} \le D(x, y) + \alpha^m + \alpha^n \le C \alpha^{(u, v)_G}, \quad \forall u = (x, m) \in V, \ \forall v = (y, n) \in V. 
    \end{equation}
    We may assume without loss of generality that $m \le n$. 
    
    First, we consider the first inequality in~\eqref{eq:hyperbolic-filling-proof-2}. First, we consider the case where there exists a path $(y_m, m) \sim \cdots \sim (y_n, n) = (y, n)$ such that either $(x, m) \sim (y_m, m)$ or $(x, m) = (y_m, m)$. In this case, 
    \begin{equation*}
        (u, v)_G = \frac12(m + n - d_G(u, v)) \ge \frac12(m + n - (n - m + 1)) = m - \frac12, 
    \end{equation*}
    Thus, $D(x, y) + \alpha^m + \alpha^n \ge \alpha^m \ge \alpha^{1/2} \alpha^{(u, v)_G}$. Next, we consider the case where $k < m$ is the largest non-negative integer such that there exist paths $(x_k, k) \sim \cdots \sim (x_m, m) = (x, m)$ and $(y_k, k) \sim \cdots \sim (y_n, n) = (y, n)$ such that $(x_k, k) \sim (y_k, k)$. In this case, $(x_{k + 1}, k + 1) \nsim (y_{k + 1}, k + 1)$, which implies that
    \begin{align*}
        D(x, y) &\ge D(x_{k + 1}, y_{k + 1}) - D(x, x_{k + 1}) - D(y, y_{k + 1}) \\
        &\ge D(x_{k + 1}, y_{k + 1}) - (D(x_m, x_{m - 1}) + \cdots + D(x_{k + 2}, x_{k + 1})) - (D(y_n, y_{n - 1}) + \cdots + D(y_{k + 2}, y_{k + 1})) \\
        &> 8\alpha^{k + 1} - ((\alpha^m + \alpha^{m - 1}) + \cdots + (\alpha^{k + 2} + \alpha^{k + 1})) - ((\alpha^n + \alpha^{n - 1}) + \cdots + (\alpha^{k + 2} + \alpha^{k + 1})) \\
        &> 6\alpha^{k + 1} - \frac{4\alpha^{k + 2}}{1 - \alpha}. 
    \end{align*}
    Moreover, 
    \begin{equation*}
        (u, v)_G = \frac12(m + n - d_G(u, v)) \ge \frac12(m + n - (m + n - 2k + 1)) = k - \frac12. 
    \end{equation*}
    Thus, 
    \begin{equation*}
        D(x, y) + \alpha^m + \alpha^n > D(x, y) > 6\alpha^{k + 1} - \frac{4\alpha^{k + 2}}{1 - \alpha} \ge \left(6\alpha^{3/2} - \frac{4\alpha^{5/2}}{1 - \alpha}\right)\alpha^{(u, v)_G}. 
    \end{equation*}
    This completes the proof of the first inequality in~\eqref{eq:hyperbolic-filling-proof-2}. 

    Next, we consider the second inequality in~\eqref{eq:hyperbolic-filling-proof-2}. Let $(x, m) = (x_0, m_0) \sim (x_1, m_1) \cdots (x_N, m_N) = (y, n)$ be a $d_G$-geodesic connecting $(x, m)$ and $(y, n)$. Note that $\lvert m_{j - 1} - m_j\rvert \le 1$ and $D(x_{j - 1}, x_j) \le 4\alpha^{m_{j - 1}} + 4\alpha^{m_j}$ for all $j \in [1, N]_\BZ$. This implies that 
    \begin{equation*}
        D(x, y) \le (4\alpha^m + 4\alpha^{m - 1}) + \cdots + (4\alpha^{k + 1} + 4\alpha^k) + 8\alpha^k + (4\alpha^k + 4\alpha^{k + 1}) + \cdots + (4\alpha^{n - 1} + 4\alpha^n) < \frac{16\alpha^k}{1 - \alpha}, 
    \end{equation*}
    where $k \defeq \lfloor(m + n - N)/2\rfloor \le m$. Since $(u, v)_G = (m + n - N)/2$, we conclude that
    \begin{equation*}
        D(x, y) + \alpha^m + \alpha^n < \frac{16\alpha^k}{1 - \alpha} + 2\alpha^k \le \left(\frac{16\alpha^{-1/2}}{1 - \alpha} + 2\alpha^{-1/2}\right) \alpha^{(u, v)_G}. 
    \end{equation*}
    This completes the proof of the second inequality in~\eqref{eq:hyperbolic-filling-proof-2}. 

    First, we consider assertion~\eqref{it:hyperbolic-filling-0}. We conclude from~\eqref{eq:hyperbolic-filling-proof-2} that
    \begin{align*}
        C^{-1} \alpha^{(u, v)_G} &\le D(x, y) + \alpha^m + \alpha^n \le D(x, z) + \alpha^m + \alpha^l + D(y, z) + \alpha^n + \alpha^l \\
        &\le C\alpha^{(u, w)_G} + C\alpha^{(v, w)_G} \le 2C\alpha^{(u, w)_G \wedge (v, w)_G} 
    \end{align*}
    for all $u = (x, m) \in V$, $v = (y, n) \in V$, and $w = (z, l) \in V$. This implies that $(u, v)_G \ge (u, w)_G \wedge (v, w)_G - \delta$ for some $\delta = \delta(\alpha) \ge 0$ and all $u, v, w \in V$. By possibly increasing $\delta$, this implies that $(u, v)_G \ge (u, w)_G \wedge (v, w)_G - \delta$ for some $\delta = \delta(\alpha) \ge 0$ and all $u, v, w \in G$. This completes the proof of assertion~\eqref{it:hyperbolic-filling-0}. 

    Next, we consider assertion~\eqref{it:hyperbolic-filling-1}. One verifies immediately that
    \begin{equation*}
        \partial_\infty(G, d_G) = \{\{u_j\}_{j \in \BN} \subset V : (u_j, u_k)_G \to \infty \text{ as } j, k \to \infty\}/\sim,  
    \end{equation*}
    i.e., it suffices to consider the case where $u_j \in V$ for all $j \in \BN$. Write $u_j = (x_j, m_j)$. Then it follows from~\eqref{eq:hyperbolic-filling-proof-2} that $D(x_j, x_k) \le C\alpha^{(u_j, u_k)_G} - \alpha^{m_j} - \alpha^{n_k} \to 0$ as $j, k \to \infty$, i.e., that $\{x_j\}_{j \in \BN}$ is a Cauchy sequence in $(X, D)$. Since $(X, D)$ is complete, $x \defeq \lim_{j \to \infty} x_j$ exists. Thus, we obtain a natural mapping 
    \begin{equation}\label{eq:hyperbolic-filling-proof-0}
        \partial_\infty(G, d_G) \to X \colon \{u_j\}_{j \in \BN} \mapsto x. 
    \end{equation}
    Let $\{v_j = (y_j, n_j)\}_{j \in \BN} \subset V$ be another sequence in $\partial_\infty(G, d_G)$ with $y \defeq \lim_{j \to \infty} y_j$. Suppose that $\{u_j\}_{j \in \BN} \nsim \{v_j\}_{j \in \BN}$. Then it follows from~\eqref{eq:hyperbolic-filling-proof-2} that
    \begin{equation}\label{eq:hyperbolic-filling-proof-1}
        D(x, y) = \lim_{j \to \infty} D(x_j, y_j) \ge \limsup_{j \to \infty} \left(C^{-1} \alpha^{(u_j, v_j)_G} - \alpha^{m_j} - \alpha^{n_j}\right) > 0. 
    \end{equation}
    This implies that the mapping in~\eqref{eq:hyperbolic-filling-proof-0} is injective. Let $z \in X$. Since $X \subset \bigcup_{x \in A_n} B_{\alpha^n}(x; D)$ for all $n \in \BN$, there exists a sequence $\{w_j = (z_j, n)\}_{j \in \BN}$ such that $D(z_j, z) \le \alpha^n$ for all $j \in \BN$. By~\eqref{eq:hyperbolic-filling-proof-2}, $C^{-1} \alpha^{(w_j, w_k)_G} \le D(z_j, z_k) + \alpha^j + \alpha^k \to 0$ as $j, k \to \infty$. This implies that $\{w_j\}_{j \in \BN} \in \partial_\infty(G, d_G)$. Thus, the mapping in~\eqref{eq:hyperbolic-filling-proof-0} is surjective, hence bijective. It remains to show that $D$ induces a visual metric on $\partial_\infty(G, d_G)$. Indeed, it follows from~\eqref{eq:product-boundary} and~\eqref{eq:hyperbolic-filling-proof-1} that
    \begin{equation*}
        D(x, y) \ge \limsup_{j \to \infty} \left(C^{-1} \alpha^{(u_j, v_j)_G} - \alpha^{m_j} - \alpha^{n_j}\right) \ge C^{-1} \alpha^{(x, y)_G + 2\delta}. 
    \end{equation*}
    It follows from~\eqref{eq:product-boundary} and~\eqref{eq:hyperbolic-filling-proof-2} that 
    \begin{equation*}
        D(x, y) = \lim_{j \to \infty} D(x_j, y_j) \le \liminf_{j \to \infty} \left(C\alpha^{(u_j, v_j)_G} - \alpha^{m_j} - \alpha^{n_j}\right) \le C \alpha^{(x, y)_G}. 
    \end{equation*}
    This completes the proof of assertion~\eqref{it:hyperbolic-filling-1}. 
\end{proof}

Fix an assignment $\sigma \colon V \setminus \{o\} \to \BR_{\ge 0}$. Suppose that the following condition is satisfied: 
\begin{equation}\label{eq:admissibility}
    \tag{$\star$}
    \parbox{.85\linewidth}{For each $(y, n - 1) \in V_{n - 1}$ and $(x_0, n), (x_1, n), \cdots, (x_N, n) \in V_n$ such that $(x_0, n) \sim (x_1, n) \sim \cdots \sim (x_N, n)$, $B_{\alpha^{n - 1}}(y; D) \cap B_{4\alpha^n}(x_0; D) \neq \emptyset$, and $(X \setminus B_{2\alpha^{n - 1}}(y; D)) \cap B_{4\alpha^n}(x_N; D) \neq \emptyset$,
    \begin{equation*}\hspace{-.15\linewidth}
        \sum_{j = 0}^N \sigma(x_j, n) \ge 1. 
    \end{equation*}} 
\end{equation}
Fix a sufficiently small parameter $\eta > 0$. For each $n \in \BN$ and $u \in V_n$, we shall write
\begin{equation}\label{eq:definition-nu-mu}
    \nu(u) \defeq 2 \sup\{\sigma(w) : v, w \in V_n, \ u \sim v \sim w\}; \quad \mu(u) \defeq \eta \vee \nu(u) \wedge (1 - \eta). 
\end{equation}

Choose a subgraph $(V, \sim_Z)$ of $(V, \sim)$ with the same vertices satisfying the following conditions:
\begin{itemize}
    \item Every horizontal edge $u \sim v$ is also an edge $u \sim_Z v$. 
    \item For each $(y, n + 1) \in V_{n + 1}$, there is exactly one $(x, n) \in V_n$ such that $(x, n) \sim_Z (y, n + 1)$ and 
    \begin{equation}\label{eq:parent}
        D(x, y) \le \sup_{(x^\prime, n) \in V_n \setminus \{(x, n)\}} D(x^\prime, y). 
    \end{equation}
    In particular, the vertices together with the vertical edges of $(V, \sim_Z)$ form a tree. (Note that if $(x, n) \sim_Z (y, n + 1)$, then $D_\Phi(x, y) \le \alpha^n$.)
\end{itemize}
For each $u \in V_n$, write $o = g(u)_0 \sim_Z g(u)_1 \sim_Z \cdots \sim_Z g(u)_n = u$ for the unique vertical path from the root to $u$. We shall write $(Z, d_Z)$ for the metric graph associated with $(V, \sim_Z)$. Note that if $u \in V_n$ and $v \in V_{n + 1}$ such that $u \sim v$, then $u \sim g(v)_n$. This implies that $d_Z(u, v) \le 2d_G(u, v) + 2$ for all $u, v \in Z$. Thus, $(Z, d_Z)$ and $(G, d_G)$ are quasi-isometric. In particular, $(Z, d_Z)$ is Gromov hyperbolic and the Gromov boundary $\partial_\infty(Z, d_Z)$ is quasisymmetrically equivalent to $(X, D)$. 

The following is reproduced from \cite[Lemma~2.13]{ConformalGauge} and \cite[Lemma~4.4]{CP-ModARCD}.

\begin{lemma}\label{lem:induction}
    Let $n \ge 0$. Let $\pi \colon V_n \to \BR_{> 0}$ and $\pi^\prime \colon V_{n + 1} \to \BR_{> 0}$. Suppose that the following conditions are satisfied:
    \begin{enumerate}[label=(\alph*),ref=\alph*]
        \item\label{it:induction-condition-0} $\eta \le \pi(u)/\pi(u^\prime) \le \eta^{-1}$ for all $u, u^\prime \in V_n$ with $u \sim u^\prime$. 
        \item\label{it:induction-condition-1} For each $v \in V_{n + 1}$, there exists $u \in V_n$ such that $u \sim v$ and $1 \le \pi(u)/\pi^\prime(v) \le \eta^{-1}$. 
    \end{enumerate}
    Then there exists $\pi \colon V_{n + 1} \to \BR_{> 0}$ satisfying the following conditions:
    \begin{enumerate}
        \item\label{it:induction-0} $\eta \le \pi(v)/\pi(v^\prime) \le \eta^{-1}$ for all $v, v^\prime \in V_{n + 1}$ with $v \sim v^\prime$. 
        \item\label{it:induction-1} For each $v \in V_{n + 1}$, either $\pi(v) = \pi^\prime(v)$ or $\pi(v) = \eta\sup\{\pi^\prime(v^\prime) : v^\prime \in V_{n + 1}, \ v \sim v^\prime\} > \pi^\prime(v)$.  
        \item\label{it:induction-2} For each $v \in V_{n + 1}$, there exists $u \in V_n$ such that $u \sim v$ and $1 \le \pi(u)/\pi(v) \le \eta^{-1}$. 
    \end{enumerate}
\end{lemma}

\begin{proof}
    For each $v \in V_{n + 1}$, let $h(v) \defeq u$ be as in condition~\eqref{it:induction-condition-1}. For each $v, v^\prime \in V_{n + 1}$, we shall write $v \succ v^\prime$ if $v \sim v^\prime$ and $\pi^\prime(v) > \eta^{-1} \pi^\prime(v^\prime)$. We \emph{claim} that there do not exist $v, v^\prime, v^\pprime \in V_{n + 1}$ such that $v \succ v^\prime \succ v^\pprime$. Indeed, since $v \sim v^\prime \sim v^\pprime$, it follows that $h(v) \sim h(v^\pprime)$. However, $\pi(h(v)) \ge \pi^\prime(v) > \eta^{-2}\pi^\prime(v^\pprime) \ge \eta^{-1}\pi(h(v^\pprime))$, in contradiction to condition~\eqref{it:induction-condition-0}. This completes the proof of the \emph{claim}. 

    For each $v \in V_{n + 1}$, if there does not exist $v^\prime \in V_{n + 1}$ such that $v^\prime \succ v$, then set $\pi(v) \defeq \pi^\prime(v)$; otherwise, set $\pi(v) \defeq \eta\sup\{\pi^\prime(v^\prime) : v^\prime \in V_{n + 1}, \ v^\prime \succ v\} > \pi^\prime(v)$. It is clear that condition~\eqref{it:induction-1} is satisfied. Condition~\eqref{it:induction-2} follows immediately from condition~\eqref{it:induction-1}, together with the fact that if $v \sim v^\prime$, then $h(v) \sim h(v^\prime)$. Thus, it suffices to verify condition~\eqref{it:induction-0}. Let $v, v^\prime \in V_{n + 1}$ with $v \sim v^\prime$. 

    First, we consider the case where $\pi(v) = \pi^\prime(v)$ and $\pi(v^\prime) = \pi^\prime(v^\prime)$. In this case, by definition, $v \nsucc v^\prime$ and $v^\prime \nsucc v$, which implies that $\eta \le \pi(v)/\pi(v^\prime) = \pi^\prime(v)/\pi^\prime(v^\prime) \le \eta^{-1}$. Next, we consider the case where $\pi(v) = \pi^\prime(v)$ and $\pi(v^\prime) > \pi^\prime(v^\prime)$. In this case, by definition, there exists $v^\pprime \succ v^\prime$ such that $\pi(v^\prime) = \eta\sup\{\pi^\prime(v^{\prime\prime\prime}) : v^{\prime\prime\prime} \in V_{n + 1}, \ v^\prime \sim v^{\prime\prime\prime}\} = \eta\pi^\prime(v^\pprime)$. In particular, $\pi(v^\prime) = \eta\pi^\prime(v^\pprime) \ge \eta\pi^\prime(v) = \eta\pi(v)$. On the other hand, since $v \sim v^\prime \sim v^\pprime$, it follows that $h(v) \sim h(v^\pprime)$, which implies that $\pi(v^\prime) = \eta\pi^\prime(v^\pprime) \le \eta\pi(h(v^\pprime)) \le \pi(h(v)) \le \eta^{-1}\pi^\prime(v) = \eta^{-1}\pi(v)$. Finally, we consider the case where $\pi(v) > \pi^\prime(v)$ and $\pi(v^\prime) > \pi^\prime(v^\prime)$. By symmetry, it suffices to show that $\pi(v) \ge \eta\pi(v^\prime)$. By definition, there exists $v^\pprime \succ v^\prime$ such that $\pi(v^\prime) = \eta\pi^\prime(v^\pprime)$. Since $v \sim v^\prime \sim v^\pprime$, it follows that $h(v) \sim h(v^\pprime)$. Thus, $\pi(v) > \pi^\prime(v) \ge \eta\pi(h(v)) \ge \eta^2\pi(h(v^\pprime)) \ge \eta^2\pi^\prime(v^\pprime) = \eta\pi(v^\prime)$. This completes the proof of condition~\eqref{it:induction-0}. 
\end{proof}

Set $\pi(o) \defeq 1$. Then, inductively, for each $n \in \BN$, set $\pi^\prime(u) \defeq \pi(g(u)_{n - 1})\mu(u)$ for all $u \in V_n$ and let $\pi \colon V_n \to \BR_{> 0}$ be as in \Cref{lem:induction}. (Here, we note that conditions~\eqref{it:induction-condition-0} and~\eqref{it:induction-condition-1} of \Cref{lem:induction} are satisfied.) We shall write
\begin{equation}\label{eq:rho}
    \varrho(u) \defeq \pi(u)/\pi(g(u)_{n - 1}), \quad \forall u \in V_n. 
\end{equation} 

\Cref{lem:H1-H2,lem:H3} below correspond to Axioms~(H1), (H2), and (H3) of \cite{ConformalGauge}. The proof of \Cref{lem:H3} is adapted from \cite[Section~4.7]{CP-ModARCD}.

\begin{lemma}\label{lem:H1-H2}
    \begin{enumerate}
        \item\label{it:H1} Let $n \in \BN$ and $u \in V_n$. Then 
        \begin{itemize}
            \item $\eta \le \varrho(u) \le 1 - \eta$, and 
            \item $\mu(u) \le \varrho(u) \le \sup\{\mu(v) : v \in V_n, \ u \sim v\}$.
        \end{itemize}
        \item\label{it:H2} $\eta^2 \le \pi(u)/\pi(v) \le \eta^{-2}$ for all $u, v \in V$ with $u \sim v$. 
    \end{enumerate}
\end{lemma}

\begin{proof}
    First, we consider assertion~\eqref{it:H1}. By \Cref{lem:induction}, \eqref{it:induction-1}, either $\varrho(u) = \mu(u)$ or there exists a horizontal edge $u \sim v$ such that
    \begin{equation*}
        \varrho(u) = \frac{\eta\pi(g(v)_{n - 1})\mu(v)}{\pi(g(u)_{n - 1})} > \mu(u). 
    \end{equation*}
    Since $\eta \le \mu(u) \le 1 - \eta$, it suffices to consider the second case. Since $g(u)_{n - 1} \sim g(v)_{n - 1}$, we conclude that $\pi(g(v)_{n - 1})/\pi(g(u)_{n - 1}) \le \eta^{-1}$, which implies that $\eta \le \mu(u) < \varrho(u) \le \mu(v) \le 1 - \eta$. This completes the proof of assertion~\eqref{it:H1}. 

    Next, we consider assertion~\eqref{it:H2}. If $u \sim v$ is a horizontal edge, then the assertion follows immediately from \Cref{lem:induction}, \eqref{it:induction-0}. Suppose that $u \in V_n$ and $v \in V_{n + 1}$. By \Cref{lem:induction}, \eqref{it:induction-2}, there exists $u^\prime \in V_n$ such that $u^\prime \sim v$ and $1 \le \pi(u^\prime)/\pi(v) \le \eta^{-1}$. We observe that $u \sim u^\prime$. Thus $\eta \le \pi(u)/\pi(u^\prime) \le \eta^{-1}$, hence $\eta^2 \le \pi(u)/\pi(v) \le \eta^{-2}$. This completes the proof of assertion~\eqref{it:H2}. 
\end{proof}

We shall write
\begin{equation*}
    n(u, v) \defeq \sup\{n \ge 0 : g(u)_n \sim g(v)_n\}; \quad \pi(u, v) \defeq \pi(g(u)_{n(u, v)}) \vee \pi(g(v)_{n(u, v)})
\end{equation*} 
for all $u, v \in V$.

\begin{lemma}\label{lem:H3}
    Let $u_0 \sim_Z u_1 \sim_Z \cdots \sim_Z u_N$ be a path in $(V, \sim_Z)$. Then $\sum_{j = 0}^N \pi(u_j) \succeq \pi(u_0, u_N)$, where the implicit constant depends only on $\eta$.
\end{lemma}

\begin{proof}
    All edges considered in this proof belong to $(V, \sim_Z)$. Whenever the unscripted notation $\sim$ is used instead of $\sim_Z$, it is to emphasize that the corresponding edge is horizontal.
    To lighten notation, for $v = (y, n - 1) \in V_{n - 1}$, write $\Gamma(v)$ for the collection of paths $(x_0, n) \sim (x_1, n) \sim \cdots \sim (x_N, n)$ such that $B_{\alpha^{n - 1}}(y; D) \cap B_{4\alpha^n}(x_0; D) \neq \emptyset$ and $(X \setminus B_{2\alpha^{n - 1}}(y; D)) \cap B_{4\alpha^n}(x_N; D) \neq \emptyset$. We \emph{claim} that for each $v \in V_{n - 1}$ and $(u_0 \sim u_1 \sim \cdots \sim u_N) \in \Gamma(v)$,
    \begin{equation}\label{eq:H3-proof-0}
        \sum_{j = 1}^N \inf\!\left\{\varrho(w) : w \in V_n, \ w \sim u_{j - 1} \text{ or } w \sim u_j\right\} \ge 1. 
    \end{equation}
    Indeed, by \Cref{lem:H1-H2}, \eqref{it:H1}, 
    \begin{equation*}
        \sum_{j = 1}^N \inf\!\left\{\varrho(w) : w \in V_n, \ w \sim u_{j - 1} \text{ or } w \sim u_j\right\} \ge \sum_{j = 1}^N \inf\!\left\{\mu(w) : w \in V_n, \ w \sim u_{j - 1} \text{ or } w \sim u_j\right\}. 
    \end{equation*}
    If $\inf\{\mu(w) : w \in V_n, \ w \sim u_{j - 1} \text{ or } w \sim u_j\} = 1 - \eta$ for some $j \in [1, N]_\BZ$, since $N \ge 2$, it follows that
    \begin{equation*}
        \sum_{j = 1}^N \inf\!\left\{\mu(w) : w \in V_n, \ w \sim u_{j - 1} \text{ or } w \sim u_j\right\} \ge (1 - \eta) + \eta = 1. 
    \end{equation*}
    Thus, we may assume without loss of generality that $\inf\{\mu(w) : w \in V_n, \ w \sim u_{j - 1} \text{ or } w \sim u_j\} < 1 - \eta$ for all $j \in [1, N]_\BZ$. By the definition of $\mu$, this implies that
    \begin{equation*}
        \sum_{j = 1}^N \inf\!\left\{\mu(w) : w \in V_n, \ w \sim u_{j - 1} \text{ or } w \sim u_j\right\} \ge \sum_{j = 1}^N \inf\!\left\{\nu(w) : w \in V_n, \ w \sim u_{j - 1} \text{ or } w \sim u_j\right\}. 
    \end{equation*}
    By the definition of $\nu$, $\inf\{\nu(w) : w \in V_n, \ w \sim u_{j - 1} \text{ or } w \sim u_j\} \ge \sigma(u_{j - 1}) + \sigma(u_j)$ for each $j \in [1, N]_\BZ$. Combining this with~\eqref{eq:admissibility}, we obtain that
    \begin{equation*}
        \sum_{j = 1}^N \inf\!\left\{\nu(w) : w \in V_n, \ w \sim u_{j - 1} \text{ or } w \sim u_j\right\} \ge \sum_{j = 0}^N \sigma(u_j) \ge 1. 
    \end{equation*}
    This completes the proof of~\eqref{eq:H3-proof-0}.

    To lighten notation, for $u \in V_n$, write 
    \begin{equation*}
        \pi^\prime(u) \defeq \inf\{\pi(v) : v \in V_n, \ u \sim v\} \quad \text{and} \quad \varrho^\prime(u) \defeq \inf\{\varrho(v) : v \in V_n, \ u \sim v\}. 
    \end{equation*}
    Write 
    \begin{equation*}
        \begin{cases}
            \ell(u \sim v) \defeq \pi^\prime(u) \wedge \pi^\prime(v) & \text{if } u \sim v \text{ is horizontal}; \\
            \ell(g(u)_{n - 1} \sim u) \defeq \eta^{-4}\pi^\prime(u) & \text{if } u \in V_n \text{ for some } n \ge 1.
        \end{cases}
    \end{equation*}
    Note that for each path $u_0 \sim_Z u_1 \sim_Z \cdots \sim_Z u_N$ in $(V, \sim_Z)$, we have $\sum_{j = 0}^N \pi(u_j) \ge \sum_{j = 0}^N \pi^\prime(u_j) \ge \eta^4 \ell(u_0 \sim_Z u_1 \sim_Z \cdots \sim_Z u_N)$. Thus, it suffices to show that 
    \begin{equation}\label{eq:H3-proof-4}
        \ell(u_0 \sim_Z u_1 \sim_Z \cdots \sim_Z u_N) \succeq \pi(u_0, u_N). 
    \end{equation}
    We \emph{claim} that for each $v \in V_{n - 1}$ and $(u_0 \sim u_1 \sim \cdots \sim u_N) \in \Gamma(v)$,
    \begin{equation}\label{eq:H3-proof-1}
        \ell(u_0 \sim u_1 \sim \cdots \sim u_N) \ge \pi^\prime(v). 
    \end{equation}
    Write $v = (y, n - 1)$ and $u_j = (x_j, n)$ for $j \in [0, N]_\BZ$. We may assume without loss of generality that $B_{4\alpha^n}(x_j; D) \cap B_{2\alpha^{n - 1}}(y; D) \neq \emptyset$ for all $j \in [0, N]_\BZ$. We observe that for each $j \in [0, N]_\BZ$ and $w \in V_n$ such that $w \sim u_j$, we have $v \sim g(w)_{n - 1}$. Indeed, if we write $w = (z, n)$ and $g(w)_{n - 1} = (z^\prime, n - 1)$, then
    \begin{equation*}
        D(y, z^\prime) \le D(y, x_j) + D(x_j, z) + D(z, z^\prime) \le (2\alpha^{n - 1} + 4\alpha^n) + 8\alpha^n + \alpha^{n - 1} < 8\alpha^{n - 1}. 
    \end{equation*}
    This implies that $v \sim g(w)_{n - 1}$. Thus, we conclude that
    \begin{multline*}
        \ell(u_0 \sim u_1 \sim \cdots \sim u_N) = \sum_{j = 1}^N \inf\!\left\{\pi(w) : w \in V_n, \ w \sim u_{j - 1} \text{ or } w \sim u_j\right\} \\
        \ge \pi^\prime(v) \sum_{j = 1}^N \inf\!\left\{\varrho(w) : w \in V_n, \ w \sim u_{j - 1} \text{ or } w \sim u_j\right\} = \pi^\prime(v) \sum_{j = 1}^N \varrho^\prime(u_{j - 1}) \wedge \varrho^\prime(u_j) \ge \pi^\prime(v), 
    \end{multline*}
    where the last inequality follows from~\eqref{eq:H3-proof-0} and the definition of $\varrho^\prime$. This completes the proof of~\eqref{eq:H3-proof-1}.

    Next, we \emph{claim} that 
    \begin{equation}\label{eq:H3-proof-2}
        \parbox{.85\linewidth}{for each path $u_0 \sim_Z u_1 \sim \cdots \sim u_{M - 1} \sim_Z u_M$ with $u_0, u_M \in V_n$ and $u_1, \cdots, u_{M - 1} \in V_{n + 1}$, there exists a path $u_0 = v_0 \sim v_1 \sim \cdots \sim v_N = u_M$ such that $v_0, v_1, \cdots, v_N \in V_n$ and $\ell(v_0 \sim v_1 \sim \cdots \sim v_N) \le \ell(u_1 \sim \cdots \sim u_{M - 1} \sim_Z u_M)$.}
    \end{equation}
    (Here, we note that the edge $u_0 \sim_Z u_1$ is excluded from the upper bound; this exclusion will be necessary for the subsequent proofs of~\eqref{eq:H3-proof-7} and~\eqref{eq:H3-proof-6}.) If $u_0 \sim u_M$, then it follows from \Cref{lem:H1-H2}, \eqref{it:H2} and the definition of $\ell$ that $\ell(u_{M - 1} \sim_Z u_M) = \eta^{-4} \inf\{\pi(v) : v \in V_{n + 1}, \ u_{M - 1} \sim v\} \ge \eta^{-2} \pi(u_{M - 1}) \ge \pi(u_M) \ge \ell(u_0 \sim u_M)$. Thus, it suffices to consider the case $u_0 \nsim u_M$. In this case, $(u_1 \sim \cdots \sim u_{M - 1}) \in \Gamma(u_0)$. Indeed, if we write $u_0 = (x_0, n)$, $u_M = (x_M, n)$, and $u_j = (x_j, n + 1)$ for $j \in [1, M - 1]_\BZ$, then $u_0 \sim_Z u_1$ implies that $B_{\alpha^n}(x_0; D) \cap B_{4\alpha^{n + 1}}(x_1; D) \neq \emptyset$, and
    \begin{equation*}
        D(x_0, x_{M - 1}) \ge D(x_0, x_M) - D(x_M, x_{M - 1}) > 8\alpha^n - \alpha^n = 7\alpha^n, 
    \end{equation*}
    which implies that $(X \setminus B_{2\alpha^n}(x_0; D)) \cap B_{4\alpha^{n + 1}}(x_{M - 1}; D) \neq \emptyset$. Write $j_1 > 1$ for the smallest number such that $(u_1 \sim \cdots \sim u_{j_1}) \in \Gamma(u_0)$. Then it follows from~\eqref{eq:H3-proof-1} and the proof of~\eqref{eq:H3-proof-1} that $u_0 \sim g(u_{j_1})_n$ and $\ell(u_0 \sim g(u_{j_1})_n) \le \pi^\prime(u_0) \le \ell(u_1 \sim \cdots \sim u_{j_1})$. Inductively, for each $k \in \BN$, if $g(u_{j_k})_n \nsim u_M$, then we may write $j_{k + 1} > j_k$ for the smallest number such that $(u_{j_k} \sim \cdots \sim u_{j_{k + 1}}) \in \Gamma(g(u_{j_k})_n)$, in which case we have $g(u_{j_k})_n \sim g(u_{j_{k + 1}})_n$ and $\ell(g(u_{j_k})_n \sim g(u_{j_{k + 1}})_n) \le \pi^\prime(g(u_{j_k})_n) \le \ell(u_{j_k} \sim \cdots \sim u_{j_{k + 1}})$. Thus, we obtain $1 < j_1 < \cdots < j_N \le M - 1$ such that 
    \begin{align*}
        \ell(u_0 \sim g(u_{j_1})_n \sim \cdots \sim g(u_{j_N})_n \sim u_M) &\le \ell(u_1 \sim \cdots \sim u_{M - 1}) + \ell(g(u_{j_N})_n \sim u_M) \\
        &\le \ell(u_1 \sim \cdots \sim u_{M - 1}) + \ell(u_{M - 1} \sim_Z u_M). 
    \end{align*}
    This completes the proof of~\eqref{eq:H3-proof-2}.

    By repeatedly applying~\eqref{eq:H3-proof-2}, we obtain that
    \begin{equation}\label{eq:H3-proof-5}
        \parbox{.85\linewidth}{for each path $u_0 \sim_Z u_1 \sim_Z \cdots \sim_Z u_M$ with $u_0, u_M \in V_n$, there exists a path $u_0 = v_0 \sim_Z v_1 \sim_Z \cdots \sim_Z v_N = u_M$ such that $v_0, v_1, \cdots, v_N \in \bigcup_{k = 0}^n V_k$ and $\ell(v_0 \sim_Z v_1 \sim_Z \cdots \sim_Z v_N) \le \ell(u_0 \sim_Z u_1 \sim_Z \cdots \sim_Z u_M)$,}
    \end{equation}
    and that
    \begin{equation}\label{eq:H3-proof-7}
        \parbox{.85\linewidth}{for each path $u_0 \sim_Z u_1 \sim_Z \cdots \sim_Z u_M$ with $u_0 \in V_n$ and $u_M \in \bigcup_{k = n + 1}^\infty V_k$, there exists a path $u_0 = v_0 \sim_Z v_1 \sim_Z \cdots \sim_Z v_N = g(u_M)_n$ such that $v_0, v_1, \cdots, v_N \in \bigcup_{k = 0}^n V_k$ and $\ell(v_0 \sim_Z v_1 \sim_Z \cdots \sim_Z v_N) \le \ell(u_0 \sim_Z u_1 \sim_Z \cdots \sim_Z u_M)$.}
    \end{equation}
    Next, we \emph{claim} that 
    \begin{equation}\label{eq:H3-proof-6}
        \parbox{.85\linewidth}{for each path $u_0 \sim_Z u_1 \sim_Z \cdots \sim_Z u_M$ with $u_0, u_M \in V_n$, if $u_0 \nsim u_M$ and $g(u_0)_{n - 1} \nsim g(u_M)_{n - 1}$, then there exists a path $g(u_0)_{n - 1} = v_0 \sim_Z v_1 \sim_Z \cdots \sim_Z v_N = g(u_M)_{n - 1}$ such that $v_0, v_1, \cdots, v_N \in \bigcup_{k = 0}^{n - 1} V_k$ and $\ell(v_0 \sim_Z v_1 \sim_Z \cdots \sim_Z v_N) \le \ell(u_0 \sim_Z u_1 \sim_Z \cdots \sim_Z u_M)$.}
    \end{equation}
    By~\eqref{eq:H3-proof-5}, we may assume without loss of generality that $u_0, u_1, \cdots, u_M \in \bigcup_{k = 0}^n V_k$.  First, we consider the case where $\{u_0, u_1, \cdots, u_M\} \not\subset V_n$.  In this case, write $j$ (resp.~$k$) for the smallest (resp.~largest) number such that $u_j \in V_{n - 1}$ (resp.~$u_k \in V_{n - 1}$). By~\eqref{eq:H3-proof-2}, there exists a path from $g(u_0)_{n - 1}$ to $u_j$ in $(V_{n - 1}, \sim)$ (resp.~$g(u_M)_{n - 1}$ to $u_k$ in $(V_{n - 1}, \sim)$) whose $\ell$-length is at most $\ell(u_0 \sim \cdots \sim u_{j - 1} \sim_Z u_j)$ (resp.~$\ell(u_M \sim \cdots \sim u_{k + 1} \sim_Z u_k)$). By~\eqref{eq:H3-proof-5}, there exists a path from $u_j$ to $u_k$ in $(\bigcup_{k = 0}^{n - 1} V_k, \sim_Z)$ whose $\ell$-length is at most $\ell(u_j \sim_Z \cdots \sim_Z u_k)$. By concatenating these three paths, we complete the proof of the case where $\{u_0, u_1, \cdots, u_M\} \not\subset V_n$. It remains to consider the case where $u_0, u_1, \cdots, u_M \in V_n$. In this case, by a similar argument to the argument applied in the proof of~\eqref{eq:H3-proof-2}, there are $0 = j_0 < j_1 < \cdots < j_N \le M$ such that if we write $u_j = (x_j, n)$ and $g(u_j)_{n - 1} = (x_j^\prime, n - 1)$, then
    \begin{itemize}
        \item for each $k \in [1, N]_\BZ$, 
        \begin{itemize}
            \item $(u_{j_{k - 1}} \sim \cdots \sim u_{j_k}) \in \Gamma(g(u_{j_{k - 1}})_{n - 1})$, 
            \item $B_{2\alpha^{n - 1}}(x_{j_{k - 1}}^\prime; D) \cap B_{4\alpha^n}(x_{j_k}; D) \neq \emptyset$, 
            \item $\ell(u_{j_{k - 1}} \sim \cdots \sim u_{j_k}) \ge \pi^\prime(g(u_{j_{k - 1}})_{n - 1})$,
        \end{itemize}
        \item $(u_{j_N} \sim \cdots \sim u_M) \notin \Gamma(g(u_{j_N})_{n - 1})$.
    \end{itemize}
    This implies that
    \begin{align*}
        D(x_{j_{N - 1}}^\prime, x_M^\prime) &\le D(x_{j_{N - 1}}^\prime, x_{j_N}) + D(x_{j_N}, x_{j_N}^\prime) + D(x_{j_N}^\prime, x_M) + D(x_M, x_M^\prime) \\
        &\le (2\alpha^{n - 1} + 4\alpha^n) + \alpha^{n - 1} + 2\alpha^{n - 1} + \alpha^{n - 1} < 8\alpha^{n - 1}, 
    \end{align*}
    which implies that $g(u_{j_{N - 1}})_{n - 1} \sim g(u_M)_{n - 1}$. Thus, we conclude that
    \begin{align*}
        &\ell(g(u_0)_{n - 1} \sim g(u_{j_1})_{n - 1} \sim \cdots \sim g(u_{j_{N - 1}})_{n - 1} \sim g(u_M)_{n - 1})\\
        \le& \sum_{k = 0}^{N - 1} \pi^\prime(g(u_{j_k})_{n - 1})
        \le  \ell(u_0 \sim u_1 \sim \cdots \sim u_M). 
    \end{align*}
    This completes the proof of~\eqref{eq:H3-proof-6}.

    We are now ready to prove~\eqref{eq:H3-proof-4}. By repeatedly applying~\eqref{eq:H3-proof-7} and~\eqref{eq:H3-proof-6}, we reduce immediately to the case where $u_0, u_N \in V_n$ for some $n \ge 0$ and either $u_0 \sim u_N$ or $g(u_0)_{n - 1} \sim g(u_N)_{n - 1}$, in which case one verifies immediately that~\eqref{eq:H3-proof-4} is true. This completes the proof of \Cref{lem:H3}.
\end{proof}

We shall write $d_\varrho$ for the geodesic metric on $Z$ such that 
\begin{equation*}
    \begin{cases}
        \len(u \sim v; d_\varrho) \defeq 2\log(1/\eta) & \text{if } u \sim v \text{ is horizontal}; \\
        \len(g(u)_{n - 1} \sim u; d_\varrho) \defeq \log(1/\varrho(u)) & \text{if } u \in V_n \text{ for some } n \ge 1.
    \end{cases}
\end{equation*}
It follows immediately from \Cref{lem:H1-H2}, \eqref{it:H1} that $(Z, d_\varrho)$ is bi-Lipschitz equivalent to $(Z, d_Z)$. In particular, $(Z, d_\varrho)$ is Gromov hyperbolic and the Gromov boundary $\partial_\infty(Z, d_\varrho)$ is quasisymmetrically equivalent to $(X, D)$. We shall write
\begin{equation*}
    (u, v)_\varrho \defeq \frac12(d_\varrho(o, u) + d_\varrho(o, v) - d_\varrho(u, v)), \quad \forall u, v \in Z
\end{equation*}
for the Gromov product. 

For each $\varepsilon \in (0, 1]$, write 
\begin{equation*}
    \phi_\varepsilon(u) \defeq \re^{-\varepsilon d_\varrho(o, u)}, \quad \forall u \in Z; \quad d_\varepsilon(u, v) \defeq \inf_P \int_a^b \phi_\varepsilon(P(t)) \, \rd t, \quad \forall u, v \in Z, 
\end{equation*}
where the infimum is over all paths $P \colon [a, b] \to Z$ from $u$ to $v$ parameterized by $d_\varrho$-length. Write $(\overline Z_\varepsilon, d_\varepsilon)$ for the completion of $(Z, d_\varepsilon)$ and $\partial_\varepsilon Z \defeq \overline Z_\varepsilon \setminus X$. By the discussion of \Cref{section:background-Gromov}, there exists $\varepsilon_\ast > 0$ such that for each $\varepsilon \in (0, \varepsilon_\ast]$, we have that $\partial_\varepsilon Z$ may be naturally identified with $\partial_\infty(Z, d_\varrho) \sim X$ and $d_\varepsilon$ induces a visual metric on $X$ with parameter $\varepsilon$. It turns out that each $d_\varepsilon$ for $\varepsilon \in (0, 1]$ induces a visual metric on $X$ with parameter $\varepsilon$. 

For distinct $x, y \in X$, we shall write
\begin{gather*}
    n(x, y) \defeq \sup\{n \ge 0 : \text{there exists } (z, n) \in V_n \text{ such that } x, y \in B_{2\alpha^n}(z; D)\}; \\
    c(x, y) \defeq \{(z, n(x, y)) \in V_{n(x, y)} : x, y \in B_{2\alpha^{n(x, y)}}(z; D)\}; \\
    \pi(x, y) \defeq \sup\{\pi(u) : u \in c(x, y)\}. 
\end{gather*} 

The following appears in \cite[Proposition~2.4]{ConformalGauge}.

\begin{lemma}\label{lem:d_epsilon}
    For each $\varepsilon \in (0, 1]$, we have $d_\varepsilon(x, y) \asymp \pi(x, y)^\varepsilon$ for all $x, y \in X$. 
\end{lemma}

\begin{proof}
    First, we \emph{claim} that every vertical path in $Z$ is a $d_\varrho$-geodesic. In particular, $d_\varrho(o, u) = \log(1/\pi(u))$ for all $u \in V$. Indeed, it suffices to show that each vertical path in $Z$ from $o$ to a point of $V$ is a $d_\varrho$-geodesic. Suppose by way of contradiction that this is false. Then there exists $n \ge 1$ and $u, v \in V_n$ such that $u \sim v$ and 
    \begin{align*}
        \log(1/\pi(v)) + 2\log(1/\eta) &= \len(o \sim g(v)_1 \sim \cdots \sim g(v)_{n - 1} \sim v \sim u; d_\varrho) \\
        &< \len(o \sim g(u)_1 \sim \cdots \sim g(u)_{n - 1} \sim u; d_\varrho) = \log(1/\pi(u)), 
    \end{align*}
    in contradiction to \Cref{lem:H1-H2}, \eqref{it:H2}. This completes the proof of the \emph{claim}. 

    Next, we \emph{claim} that for each $\varepsilon \in (0, 1]$, 
    \begin{equation}\label{eq:d_epsilon-proof-0}
        \len(u \sim_Z v; d_\varepsilon) \asymp \pi(u)^\varepsilon, \quad \forall u, v \in V \text{ with } u \sim_Z v. 
    \end{equation}
    Indeed, by \Cref{lem:H1-H2}, \eqref{it:H1} and the definition of $d_\varrho$, we have $\len(u \sim_Z v; d_\varrho) \asymp 1$. Combining this with the \emph{claim} of the preceding paragraph, we obtain that $\re^{-\varepsilon d_\varrho(o, u)} \asymp \pi(u)^\varepsilon$. Thus, \eqref{eq:d_epsilon-proof-0} follows immediately from the definition of $d_\varepsilon$. 

    Fix $\varepsilon \in (0, 1]$ and $x, y \in X$. Fix $u = (x_m, m) \in V$ and $v = (y_n, n) \in V$ such that $m \wedge n \ge n(x, y)$, $x \in B_{\alpha^m}(x_m; D)$, and $y \in B_{\alpha^n}(y_n; D)$. Since $d_\varepsilon(u, v) \to d_\varepsilon(x, y)$ as $m, n \to \infty$, it suffices to show that $d_\varepsilon(u, v) \asymp \pi(x, y)^\varepsilon$. 

    First, we verify that $d_\varepsilon(u, v) \preceq \pi(x, y)^\varepsilon$. Fix $(z, n(x, y)) \in c(x, y)$. We \emph{claim} that $g(u)_{n(x, y)} \sim (z, n(x, y)) \sim g(v)_{n(x, y)}$. Indeed, if we write $g(u)_j = (x_j, j)$, then
    \begin{equation*}
        D(x_{n(x, y)}, x_m) \le \sum_{j = n(x, y)}^{m - 1} D(x_j, x_{j + 1}) \le \sum_{j = n(x, y)}^{m - 1} \alpha^j \le \frac{\alpha^{n(x, y)}}{1 - \alpha}, 
    \end{equation*}
    where the second inequality follows from~\eqref{eq:parent}. This implies that 
    \begin{equation*}
        D(x_{n(x, y)}, z) \le D(x_{n(x, y)}, x_m) + D(x_m, x) + D(x, z) \le \frac{\alpha^{n(x, y)}}{1 - \alpha} + \alpha^m + 2\alpha^{n(x, y)} < 8\alpha^{n(x, y)}, 
    \end{equation*}
    hence that $g(u)_{n(x, y)} \sim (z, n(x, y))$. In a similar vein, $(z, n(x, y)) \sim g(v)_{n(x, y)}$. Thus, we conclude that 
    \begin{align*}
        d_\varepsilon(u, v) &\le \len(g(u)_m \sim_Z \cdots \sim_Z g(u)_{n(x, y)} \sim_Z (z, n(x, y)) \sim_Z g(v)_{n(x, y)} \sim_Z \cdots \sim_Z g(v)_n; d_\varepsilon) \\
        &\preceq \sum_{j = n(x, y)}^m \pi(g(u)_j)^\varepsilon + \pi(z, n(x, y))^\varepsilon + \sum_{j = n(x, y)}^n \pi(g(v)_j)^\varepsilon \quad \text{(by~\eqref{eq:d_epsilon-proof-0})} \\
        &\le \pi(z, n(x, y))^\varepsilon \left(1 + 2\eta^{-2\varepsilon}\sum_{j = 0}^\infty (1 - \eta)^{j\varepsilon}\right) \quad \text{(by \Cref{lem:H1-H2}, \eqref{it:H1} and~\eqref{it:H2})}. 
    \end{align*}
    This completes the proof that $d_\varepsilon(u, v) \preceq \pi(x, y)^\varepsilon$. 

    Next, we verify that $d_\varepsilon(u, v) \succeq \pi(x, y)^\varepsilon$. Let $u = u_0 \sim_Z u_1 \sim_Z \cdots \sim_Z u_N = v$ be a path in $(V, \sim_Z)$ connecting $u$ and $v$. Then it follows from~\eqref{eq:d_epsilon-proof-0} and \Cref{lem:H3} that
    \begin{equation*}
        \len(u_0 \sim_Z u_1 \sim_Z \cdots \sim_Z u_N; d_\varepsilon) \succeq \sum_{j = 0}^N \pi(u_j)^\varepsilon \ge \left(\sum_{j = 0}^N \pi(u_j)\right)^\varepsilon \succeq \pi(u, v)^\varepsilon, 
    \end{equation*}
    where the second inequality follows from the fact that $\varepsilon \in (0, 1]$. Thus, it suffices to show that $\pi(u, v) \succeq \pi(x, y)$. Write $g(u)_{n(u, v)} = (x_{n(u, v)}, n(u, v))$ and $g(v)_{n(u, v)} = (y_{n(u, v)}, n(u, v))$. By definition, either $(x_{n(u, v)}, n(u, v)) \sim (y_{n(u, v)}, n(u, v))$ or $x_{n(u, v)} = y_{n(u, v)}$. Write $g(u)_j = (x_j, j)$ and $g(v)_j = (y_j, j)$. We \emph{claim} that $x, y \in B_{2\alpha^{n(u, v) - 1}}(x_{n(u, v) - 1}; D)$. Indeed, 
    \begin{align*}
        D(x_{n(u, v) - 1}, x) &\le D(x_{n(u, v) - 1}, x_{n(u, v)}) + \cdots + D(x_{m - 1}, x_m) + D(x_m, x) \\
        &\le \alpha^{n(u, v) - 1} + \cdots + \alpha^{m - 1} + \alpha^m \\
        &< 2\alpha^{n(u, v) - 1}, 
    \end{align*}
    and
    \begin{align*}
        &D(x_{n(u, v) - 1}, y) \\
        &\le D(x_{n(u, v) - 1}, x_{n(u, v)}) + D(x_{n(u, v)}, y_{n(u, v)}) + D(y_{n(u, v)}, y_{n(u, v) + 1}) + \cdots + D(y_{n - 1}, y_n) + D(y_n, y) \\
        &\le \alpha^{n(u, v) - 1} + 8\alpha^{n(u, v)} + \alpha^{n(u, v)} + \cdots + \alpha^{n - 1} + \alpha^n \\
        &< 2\alpha^{n(u, v) - 1}. 
    \end{align*}
    This completes the proof that $x, y \in B_{2\alpha^{n(u, v) - 1}}(x_{n(u, v) - 1}; D)$. This implies that $n(x, y) \ge n(u, v) - 1$. Fix $(z, n(x, y)) \in c(x, y)$. We \emph{claim} that $g(z, n(x, y))_{n(u, v) - 1} \sim (x_{n(u, v) - 1}, n(u, v) - 1)$. Indeed, if we write $g(z, n(x, y))_{n(u, v) - 1} = (z_{n(u, v) - 1}, n(u, v) - 1)$, then 
    \begin{align*}
        D(z_{n(u, v) - 1}, x_{n(u, v) - 1}) &\le D(z_{n(u, v) - 1}, z) + D(z, x) + D(x, x_{n(u, v) - 1}) \\
        &\le \alpha^{n(u, v) - 1} + \cdots + \alpha^{n(x, y)} + 2\alpha^{n(x, y)} + 2\alpha^{n(u, v) - 1} \\
        &< 8\alpha^{n(u, v) - 1}. 
    \end{align*}
    This completes the proof that $g(z, n(x, y))_{n(u, v) - 1} \sim (x_{n(u, v) - 1}, n(u, v) - 1)$. Thus, 
    \begin{equation*}
        \pi(z, n(x, y)) \le \pi(g(z, n(x, y))_{n(u, v) - 1}) \preceq \pi(x_{n(u, v) - 1}, n(u, v) - 1) \preceq \pi(x_{n(u, v)}, n(u, v)) \le \pi(u, v), 
    \end{equation*}
    where the second and third inequalities follow from \Cref{lem:H1-H2}, \eqref{it:H2}, and the last inequality follows from the definition that $\pi(u, v) = \pi(x_{n(u, v)}, n(u, v)) \vee \pi(y_{n(u, v)}, n(u, v))$. This completes the proof that $d_\varepsilon(u, v) \succeq \pi(x, y)^\varepsilon$, hence the proof of \Cref{lem:d_epsilon}. 
\end{proof}

\begin{corollary}\label{lem:corollary}
    The metric $d_1$ induces a visual metric on $X$ with parameter one such that 
    \begin{equation*}
        \diam(B_{\alpha^n}(x; D); d_1) \preceq \pi(u), \quad \forall u = (x, n) \in V. 
    \end{equation*}
    In particular, since $\pi(u) \le (1 - \eta)^n$ for $u \in V_n$ (cf.~\Cref{lem:H1-H2}, \eqref{it:H1}), if $\sum_{u \in V_n} \pi(u)^p \to 0$ as $n \to \infty$, then the Hausdorff dimension of $(X, d_1)$ (hence also the conformal dimension of $(X, D)$) is at most $p$.
\end{corollary}

\begin{proof}
    Recall that there exists $\varepsilon_\ast > 0$ such that for each $\varepsilon \in (0, \varepsilon_\ast]$, we have $\partial_\varepsilon Z$ may be naturally identified with $\partial_\infty(Z, d_\varrho) \sim X$ and $d_\varepsilon$ induces a visual metric on $X$ with parameter $\varepsilon$. Combining this with \Cref{lem:d_epsilon}, we obtain that 
    \begin{equation*}
        d_1(x, y)^{\varepsilon_\ast} \asymp \pi(x, y)^{\varepsilon_\ast} \asymp d_{\varepsilon_\ast}(x, y) \asymp \re^{-\varepsilon_\ast(x, y)_\varrho}, \quad \forall x, y \in X, 
    \end{equation*}
    hence that $d_1$ induces a visual metric on $X$ with parameter one. Let $u = (x, n) \in V$ and $y, z \in B_{\alpha^n}(x; D)$. By a similar argument to the argument applied in the proof of \Cref{lem:d_epsilon}, one verifies that $n \le n(y, z)$ and $u \sim g(v)_n$ for all $v \in c(y, z)$. Thus, we conclude from \Cref{lem:d_epsilon} and \Cref{lem:H1-H2}, \eqref{it:H2} that
    \begin{equation*}
        d_1(y, z) \preceq \pi(y, z) \le \sup\{\pi(g(v)_n) : v \in c(y, z)\} \preceq \pi(u). 
    \end{equation*}
    This completes the proof of \Cref{lem:corollary}. 
\end{proof}

\section{Constructing an admissible weight}\label{section:weight}

In the present section, we adopt the notation of \Cref{section:hyperbolic-filling} with $X$ a subset of $\BC$ and $D = D_\Phi$, where $(\BC, \Phi; \infty)$ is an embedding of a $\gamma$-LQG surface; we construct a weight function $\sigma \colon \BC \times \BN \to \BR_{\ge 0}$ that satisfies~\eqref{eq:admissibility}.

Fix sufficiently small parameters $\zeta > 0$ and $\alpha = \alpha(\zeta) > 0$ to be chosen later. 

\begin{definition}\label{def:G}
    Let $\SCA = (A, \Phi_A; I, O)$ be a $\gamma$-LQG surface parameterized by a doubly connected domain $A$, where $I$ (resp.~$O$) denotes the inner (resp.~outer) boundary of $A$. Suppose that 
    \begin{equation}\label{eq:G}
        w \defeq D_{\Phi_A}(I, O) = D_{\Phi_A}(I, x), \quad \forall x \in O. 
    \end{equation}
    Then:
    \begin{itemize}
        \item Let $0 \le s < t \le w$. Then we shall write $B_t^O(I; D_{\Phi_A})$ for the complement of the connected component of $A \setminus B_t(I; D_{\Phi_A})$ whose boundary contains $O$; we shall write $A_{s,t}^O(I; D_{\Phi_A}) \defeq B_t^O(I; D_{\Phi_A}) \setminus B_s^O(I; D_{\Phi_A})$. 
        \item We shall write $G(\SCA)$ for the event that the metric ball $B_{w/8}(x; D_{\Phi_A})$ does not disconnect $I$ and $O$ for all $x \in A_{w/4,3w/4}^O(I; D_{\Phi_A})$. (Here, we note that $B_{w/8}(x; D_{\Phi_A})$ does not intersect $I$ and $O$ for all $x \in A_{w/4,3w/4}^O(I; D_{\Phi_A})$.) 
    \end{itemize}
    Let $x \in \BC$. Then we shall write $G_{s,t}(x) \defeq G(A_{s,t}^\bullet(x; D_\Phi))$. 
\end{definition}

Heuristically, the event $G(\SCA)$ ensures that the surface $\SCA$ does not contain a narrow bottleneck that could disconnect its inner and outer boundaries.

The point of defining $G$ for an abstract $\gamma$-LQG surface (rather than only for metric bands) is that, viewed purely as a $\gamma$-LQG surface, a metric band does not determine its center or its inner and outer radii.

\begin{lemma}\label{lem:G-measurability}
    In the notation of \Cref{def:G}, the event $G(\SCA)$ is almost surely determined by $\SCA$ (as a $\gamma$-LQG surface). 

\end{lemma}

\begin{proof} 
    This follows immediately from the definitions.  
\end{proof} 

\begin{definition}\label{def:F}
    Let $x \in \BC$ and $n \in \BN$. Then we shall write $F(x, n)$ for the event that there exists $\alpha^{n - 1}/8 \le t < t + 8\alpha^{n - 1 + \zeta} \le \alpha^{n - 1}/4$ such that the event $G_{t,t + 8\alpha^{n - 1 + \zeta}}(x)$ occurs. 
\end{definition}

\begin{lemma}\label{lem:F-measurability}
    In the notation of \Cref{def:F}, the event $F(x, n)$ is almost surely determined by the $\gamma$-LQG surface parameterized by 
    \begin{equation*}
        \left(A_{\alpha^{n - 1}/8,\alpha^{n - 1}/4}^\bullet(x; D_\Phi); \partial B_{\alpha^{n - 1}/8}^\bullet(x; D_\Phi), \partial B_{\alpha^{n - 1}/4}^\bullet(x; D_\Phi)\right). 
    \end{equation*}
\end{lemma}

\begin{proof}
    This follows immediately from the definitions and \Cref{lem:G-measurability}.
\end{proof}

\begin{definition}\label{def:tilde-F}
    Let $x \in \BC$ and $n \in \BN$. Then we shall write $\widetilde F(x, n)$ for the event that there exists a doubly connected domain $A$ that is contained in $A_{\alpha^{n - 1}/16,5\alpha^{n - 1}/16}^\bullet(x; D_\Phi)$ and disconnects the inner and outer boundaries of $A_{\alpha^{n - 1}/16,5\alpha^{n - 1}/16}^\bullet(x; D_\Phi)$ such that~\eqref{eq:G} is satisfied with $w = 8\alpha^{n - 1 + \zeta}$, and the event $G$ of the $\gamma$-LQG surface parameterized by $A$ occurs.
\end{definition}

The point of introducing the event $\widetilde F(x, n)$ is to obtain a version of $F(x, n)$ that is robust under small perturbations of the center $x$.

\begin{lemma}\label{lem:F-2-tilde-F}
    Let $x \in \BC$ and $n \in \BN$. If $F(x, n)$ occurs, then $\widetilde F(x^\prime, n)$ occurs for all $x^\prime \in B_{\alpha^{n - 1}/16}(x; D_\Phi)$. 
\end{lemma}

\begin{proof}
    This follows immediately from the definitions.
\end{proof}

\begin{lemma}\label{lem:F-probability}
    Suppose that $(\BC, \Phi; 0, \infty)$ is a $\gamma$-LQG cone. Then for each $n \in \BN$, 
    \begin{equation}\label{eq:F-probability}
        \inf_{\ell \ge 0} \BP\!\left\lbrack F(0, n) \ \middle\vert \ \SN_\Phi(\partial B_{\alpha^{n - 1}/8}^\bullet(0; D_\Phi)) = \ell\right\rbrack = 1 - O(\alpha^\infty), 
    \end{equation}
    as $\alpha \to 0$, at a rate which is uniform in $n$. 
\end{lemma}

\begin{proof}
    By the scaling property, the left-hand side of~\eqref{eq:F-probability} does not depend on the choice of $n$. Thus, it suffices to consider the case $n = 1$.
    
    We observe that if $t > 0$ such that the event $G_{t,t + 8\alpha^\zeta}(0)$ does not occur, then 
    \begin{equation*}
        \sup\!\left\{D_\Phi(x, y; \BC \setminus B_t^\bullet(0; D_\Phi)) : x, y \in \partial B_{t + 8\alpha^\zeta}^\bullet(0; D_\Phi)\right\} \le 32\alpha^\zeta. 
    \end{equation*}
    Indeed, for each $x, y \in \partial B_{t + 8\alpha^\zeta}^\bullet(0; D_\Phi)$, we may choose $D_\Phi$-geodesics contained in $A_{t,t + 8\alpha^\zeta}^\bullet(0; D_\Phi)$ and connecting $x$ and $y$ to $\partial B_t^\bullet(0; D_\Phi)$. By the definition of the event $G_{t,t + 8\alpha^\zeta}(0)$, these two $D_\Phi$-geodesics must cross the same metric ball of radius $\alpha^\zeta$, and this metric ball is contained in $A_{t,t + 8\alpha^\zeta}^\bullet(0; D_\Phi)$. This implies that $D_\Phi(x, y; \BC \setminus B_t^\bullet(0; D_\Phi)) \le 18\alpha^\zeta \le 32\alpha^\zeta$.

    Let $(\SCD, D_\SCD, \SM_\SCD)$ denote the random metric measure space defined by~\eqref{eq:Brownian-horn}, conditioned to have boundary length one. (Recall from the discussion immediately following~\eqref{eq:Brownian-horn} that the law of this space is well-defined.) Write $I$ for its boundary. We may choose a sufficiently small constant $c_\ast > 0$ such that for each $c \in (0, c_\ast]$, it holds with probability at most $1/2$ that $\sup_{x, y \in \partial B_c^\bullet(I; D_{\SCD})}D_\SCD(x, y) \le 4c$. By the scaling property, this implies that for each deterministic $t > 0$, 
    \begin{equation*}
        \BP\!\left\lbrack\sup\!\left\{D_\Phi(x, y; \BC \setminus B_t^\bullet(0; D_\Phi)) : x, y \in \partial B_{t + 8\alpha^\zeta}^\bullet(0; D_\Phi)\right\} \le 32\alpha^\zeta \ \middle\vert \ \SN_\Phi(\partial B_t^\bullet(0; D_\Phi)) = \ell\right\rbrack \le 1/2, 
    \end{equation*}
    hence $\BP\lbrack G_{t,t + 8\alpha^\zeta}(0) \mid \SN_\Phi(\partial B_t^\bullet(0; D_\Phi)) = \ell\rbrack \ge 1/2$, for all $\ell \ge c_\ast^{-2} \cdot 64 \cdot \alpha^{2\zeta}$.

    On the other hand, it follows from \Cref{lem:hull-process} that there are universal constants $a, C > 0$ such that for each deterministic $t \ge 0$ and $\ell \ge 0$, given $\SN_\Phi(\partial B_t^\bullet(0; D_\Phi)) = \ell$, it holds with conditional probability at least $1 - C\exp(-b\alpha^{-\zeta/2})$ that there exists $s \in [t, t + \alpha^{\zeta/2}]$ such that $\SN_\Phi(\partial B_t^\bullet(0; D_\Phi)) \ge c_\ast^{-2} \cdot 64 \cdot \alpha^{2\zeta}$, where $b \defeq ac_\ast/8$.

    Set $\tau_1 \defeq \inf\{t \ge 1/8 : \SN_\Phi(\partial B_t^\bullet(0; D_\Phi)) \ge c_\ast^{-2} \cdot 64 \cdot \alpha^{2\zeta}\}$. Inductively, for each $n \in \BN$, set $\tau_{n + 1} \defeq \inf\{t \ge \tau_n + 8\alpha^\zeta : \SN_\Phi(\partial B_t^\bullet(0; D_\Phi)) \ge c_\ast^{-2} \cdot 64 \cdot \alpha^{2\zeta}\}$. Thus, we conclude from the above discussions that for each deterministic $\ell \ge 0$, 
    \begin{equation*}
        \BP\!\left\lbrack\tau_{\lfloor\alpha^{-\zeta/2}/16\rfloor} \le 1/4 - 8\alpha^\zeta \ \middle\vert \ \SN_\Phi(\partial B_{1/8}^\bullet(0; D_\Phi)) = \ell\right\rbrack
    \end{equation*}
    is at least the conditional probability that for each $j \in [1, \alpha^{-\zeta/2}/16]_\BZ$, there exists $s \in [1/8 + 2(j - 1)\alpha^{\zeta/2}, 1/8 + (2j - 1)\alpha^{\zeta/2}]$ such that $\SN_\Phi(\partial B_t^\bullet(0; D_\Phi)) \ge c_\ast^{-2} \cdot 64 \cdot \alpha^{2\zeta}$, which, by a union bound, is at least $1 - C\exp(-b\alpha^{-\zeta/2}) \cdot (\alpha^{-\zeta/2}/16) = 1 - O(\alpha^\infty)$ as $\alpha \to 0$. Moreover, it follows from the above discussions that 
    \begin{multline*}
        \BP\!\left\lbrack\text{there exists } j \in [1, \alpha^{-\zeta/2}/16]_\BZ \text{ such that } G_{\tau_j,\tau_j + 8\alpha^\zeta}(0) \text{ occurs} \ \middle\vert \ \SN_\Phi(\partial B_{1/8}^\bullet(0; D_\Phi)) = \ell\right\rbrack \\
        \ge 1 - 2^{-\lfloor\alpha^{-\zeta/2}/16\rfloor} = 1 - O(\alpha^\infty) \quad \text{as } \alpha \to 0. 
    \end{multline*}
    Thus, we conclude that
    \begin{align*}
        \BP\!\left\lbrack F(0, n)^c \ \middle\vert \ \SN_\Phi(\partial B_{1/8}^\bullet(0; D_\Phi)) = \ell\right\rbrack &\le \BP\!\left\lbrack\tau_{\lfloor\alpha^{-\zeta/2}/16\rfloor} > 1/4 - 8\alpha^\zeta \ \middle\vert \ \SN_\Phi(\partial B_{1/8}^\bullet(0; D_\Phi)) = \ell\right\rbrack \\
        &+ \BP\!\left\lbrack\left(\bigcup\nolimits_{j \in [1, \alpha^{-\zeta/2}/16]_\BZ} G_{\tau_j,\tau_j + 8\alpha^\zeta}(0)\right)^c \ \middle\vert \ \SN_\Phi(\partial B_{1/8}^\bullet(0; D_\Phi)) = \ell\right\rbrack \\
        &= O(\alpha^\infty) \quad \text{as } \alpha \to 0. 
    \end{align*}
    This completes the proof of \Cref{lem:F-probability}.
\end{proof}

\begin{definition}
    Let $x \in \BC$ and $n \in \BN$. Then we shall write
    \begin{equation*}
        \sigma(x, n) \defeq \left(\frac{\diam(B_{4\alpha^n}(x; D_\Phi))}{\inradius(B_{\alpha^{n - 1 + \zeta}}^\bullet(x; D_\Phi))} \wedge 1\right)\mathbf 1_{\widetilde F(x, n)} + \mathbf 1_{\widetilde F(x, n)^c}, 
    \end{equation*}
    where $\inradius(B_{\alpha^{n - 1 + \zeta}}^\bullet(x; D_\Phi)) \defeq \sup\{R > 0 : B_R(0) \subset B_{\alpha^{n - 1 + \zeta}}^\bullet(x; D_\Phi)\}$. 
\end{definition}

\begin{lemma}\label{lem:admissibility}
    Let $n \in \BN$. Let $y, x_0, x_1, \cdots, x_N \in \BC$ such that
    \begin{itemize}
        \item $B_{\alpha^{n - 1}}(y; D_\Phi) \cap B_{4\alpha^n}(x_0; D_\Phi) \neq \emptyset$,
        \item $B_{4\alpha^n}(x_{j - 1}; D_\Phi) \cap B_{4\alpha^n}(x_j; D_\Phi) \neq \emptyset$ for all $j \in [1, N]_\BZ$, and 
        \item $(\BC \setminus B_{2\alpha^{n - 1}}(y; D_\Phi)) \cap B_{4\alpha^n}(x_N; D_\Phi) \neq \emptyset$. 
    \end{itemize}
    Then $\sum_{j = 0}^N \sigma(x_j, n) \ge 1$. In particular, condition~\eqref{eq:admissibility} is satisfied. 
\end{lemma}

We begin by outlining the proof of \Cref{lem:admissibility}. The goal is to find a subinterval $[a, b]_\BZ \subset [0, N]_\BZ$ such that 
\begin{equation}\label{eq:admissibility-illustration-0}
    \bigcup_{k \in [a, b]_\BZ} B_{4\alpha^n}(x_k; D_\Phi) \not\subset B_{\alpha^{n - 1 + \zeta}}^\bullet(x_j; D_\Phi), \quad \forall j \in [a, b]_\BZ. 
\end{equation}
Indeed, in this case, we would have
\begin{multline*}
    \sum_{j \in [a, b]_\BZ} \frac{\diam(B_{4\alpha^n}(x_j; D_\Phi))}{\inradius(B_{\alpha^{n - 1 + \zeta}}^\bullet(x_j; D_\Phi))} \\
    \ge \sum_{j \in [a, b]_\BZ} \diam(B_{4\alpha^n}(x_j; D_\Phi)) \cdot \diam\left(\bigcup_{k \in [a, b]_\BZ} B_{4\alpha^n}(x_k; D_\Phi)\right)^{-1} \ge 1, 
\end{multline*}
where the last inequality follows from the fact that $\bigcup_{j \in [a, b]_\BZ} B_{4\alpha^n}(x_j; D_\Phi)$ is connected. We choose $z \in \BC \setminus B_{2\alpha^{n - 1}}(y; D_\Phi)$ such that the path of balls $B_{4\alpha^n}(x_0; D_\Phi), B_{4\alpha^n}(x_1; D_\Phi), \cdots, B_{4\alpha^n}(x_N; D_\Phi)$ crosses between the inner and outer boundaries of $A_{\alpha^{n - 1},2\alpha^{n - 1}}^z(y; D_\Phi)$. The proof is divided into four cases. We first distinguish between two primary cases, depending on whether $B_{3\alpha^{n - 1}/2}^z(y; D_\Phi)$ is bounded or unbounded.

Suppose $B_{3\alpha^{n - 1}/2}^z(y; D_\Phi)$ is bounded, so that $B_{3\alpha^{n - 1}/2}^z(y; D_\Phi) = B_{3\alpha^{n - 1}/2}^\bullet(y; D_\Phi)$. We first consider the metric band $A_{\alpha^{n - 1} + 2\alpha^{n - 1 + \zeta},\alpha^{n - 1} + 4\alpha^{n - 1 + \zeta}}^\bullet(y; D_\Phi)$, which is centered at $y$, has width $2\alpha^{n - 1 + \zeta}$, and lies near the inner boundary of $A_{\alpha^{n - 1},3\alpha^{n - 1}/2}^\bullet(y; D_\Phi)$. We choose a minimal subinterval $[i_1, j_1]_\BZ \subset [0, N]_\BZ$ such that the path $B_{4\alpha^n}(x_{i_1}; D_\Phi), B_{4\alpha^n}(x_{i_1 + 1}; D_\Phi), \cdots, B_{4\alpha^n}(x_{j_1}; D_\Phi)$ crosses between the inner and outer boundaries of this metric band. We then verify that if $y \notin B_{\alpha^{n - 1 + \zeta}}^\bullet(x_j; D_\Phi)$ for all $j \in [i_1, j_1]_\BZ$, the subinterval $[i_1, j_1]_\BZ$ satisfies~\eqref{eq:admissibility-illustration-0}.

Alternatively, consider the subcase where $B_{3\alpha^{n - 1}/2}^z(y; D_\Phi)$ is bounded but there exists $k_1 \in [i_1, j_1]_\BZ$ such that $y \in B_{\alpha^{n - 1 + \zeta}}^\bullet(x_{k_1}; D_\Phi)$. Here, we rely on the event $\widetilde F(x_{k_1}, n)$. Recall that the event $\widetilde F(x_{k_1}, n)$ ensures the existence of a ``good'' doubly connected domain $A$ contained within $A_{\alpha^{n - 1}/16,5\alpha^{n - 1}/16}^\bullet(x_{k_1}; D_\Phi)$ that disconnects its inner and outer boundaries. To apply this, we first verify that $A_{\alpha^{n - 1}/16,5\alpha^{n - 1}/16}^\bullet(x_{k_1}; D_\Phi)$ is contained in $A_{\alpha^{n - 1},3\alpha^{n - 1}/2}^\bullet(y; D_\Phi)$ and disconnects its inner and outer boundaries; consequently, $A$ also shares these properties. Thus, we may choose a minimal subinterval $[i_{k_1}, j_{k_1}]_\BZ \subset [0, N]_\BZ$ such that the path $B_{4\alpha^n}(x_{i_{k_1}}; D_\Phi), B_{4\alpha^n}(x_{i_{k_1} + 1}; D_\Phi), \cdots, B_{4\alpha^n}(x_{j_{k_1}}; D_\Phi)$ crosses between the inner and outer boundaries of $A$ (more precisely, a specific subdomain of $A$). We then use the conditions of the event $\widetilde F(x_{k_1}, n)$ to verify that $[i_{k_1}, j_{k_1}]_\BZ$ satisfies~\eqref{eq:admissibility-illustration-0}.

When $B_{3\alpha^{n - 1}/2}^z(y; D_\Phi)$ is unbounded, the metric band $A_{3\alpha^{n - 1}/2,2\alpha^{n - 1}}^z(y; D_\Phi)$ must be bounded. Intuitively, in this case, $A_{3\alpha^{n - 1}/2,2\alpha^{n - 1}}^z(y; D_\Phi)$ plays a role analogous to that of $A_{\alpha^{n - 1},3\alpha^{n - 1}/2}^\bullet(y; D_\Phi)$ in the bounded setting (where $\partial B_{2\alpha^{n - 1}}^z(y; D_\Phi)$ and $\partial B_{3\alpha^{n - 1}/2}^z(y; D_\Phi)$ act as the inner and outer boundaries, respectively). The verification follows from similar arguments.

In all four cases, the path $B_{4\alpha^n}(x_a; D_\Phi), B_{4\alpha^n}(x_{a + 1}; D_\Phi), \cdots, B_{4\alpha^n}(x_b; D_\Phi)$ crosses between the inner and outer boundaries of a metric band of width $2\alpha^{n - 1 + \zeta}$, and verifying~\eqref{eq:admissibility-illustration-0} amounts to showing that for each $j \in [a, b]_\BZ$, either $B_{4\alpha^n}(x_a; D_\Phi) \not\subset B_{\alpha^{n - 1 + \zeta}}^\bullet(x_j; D_\Phi)$ or $B_{4\alpha^n}(x_b; D_\Phi) \not\subset B_{\alpha^{n - 1 + \zeta}}^\bullet(x_j; D_\Phi)$. The width of the metric band immediately implies that either $B_{4\alpha^n}(x_a; D_\Phi) \not\subset B_{\alpha^{n - 1 + \zeta}}(x_j; D_\Phi)$ or $B_{4\alpha^n}(x_b; D_\Phi) \not\subset B_{\alpha^{n - 1 + \zeta}}(x_j; D_\Phi)$. However, some careful checking is still needed to pass from the (unfilled) metric balls to the filled metric balls.

\begin{figure}[ht!]
    \centering
    \includegraphics[width=.6\linewidth]{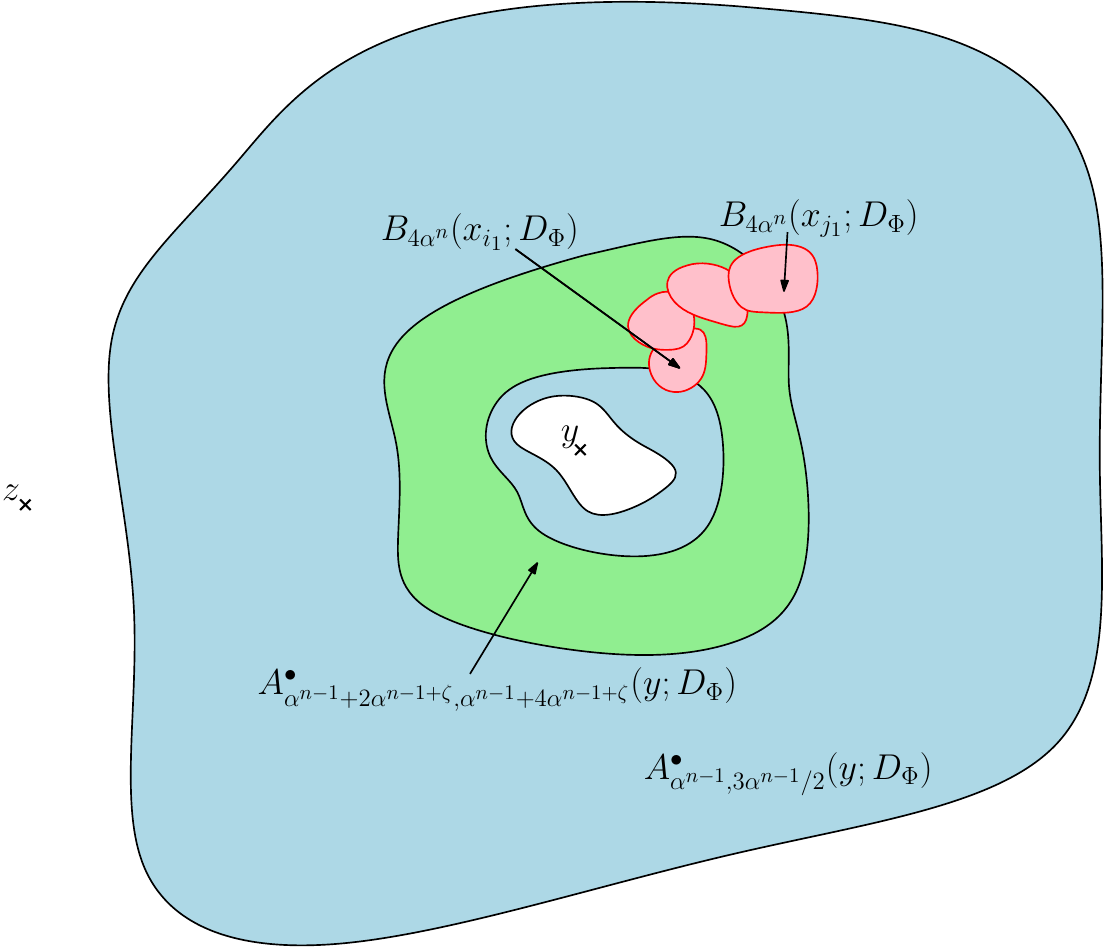}
    \caption{Illustration of the first of four cases in the proof of \Cref{lem:admissibility}. The blue region represents the metric band $A_{\alpha^{n - 1},3\alpha^{n - 1}/2}^\bullet(y; D_\Phi)$, with the green metric band $A_{\alpha^{n - 1} + 2\alpha^{n - 1 + \zeta},\alpha^{n - 1} + 4\alpha^{n - 1 + \zeta}}^\bullet(y; D_\Phi)$ overlaid upon it. The sequence of red metric balls, $B_{4\alpha^n}(x_{i_1}; D_\Phi), \cdots, B_{4\alpha^n}(x_{j_1}; D_\Phi)$, forms a minimal subpath crossing between the inner and outer boundaries of the green band. This figure depicts the scenario where $B_{3\alpha^{n - 1}/2}^z(y; D_\Phi)$ is bounded and $y \notin B_{\alpha^{n - 1 + \zeta}}^\bullet(x_j; D_\Phi)$ for all $j \in [i_1, j_1]_\BZ$. Consequently, the subpath satisfies condition~\eqref{eq:admissibility-illustration-0} with $[a, b] = [i_1, j_1]$. Note that the holes of the red metric balls are omitted for clarity.}
    \label{fig:admissibility1}
\end{figure}

\begin{figure}[ht!]
    \centering
    \includegraphics[width=.6\linewidth]{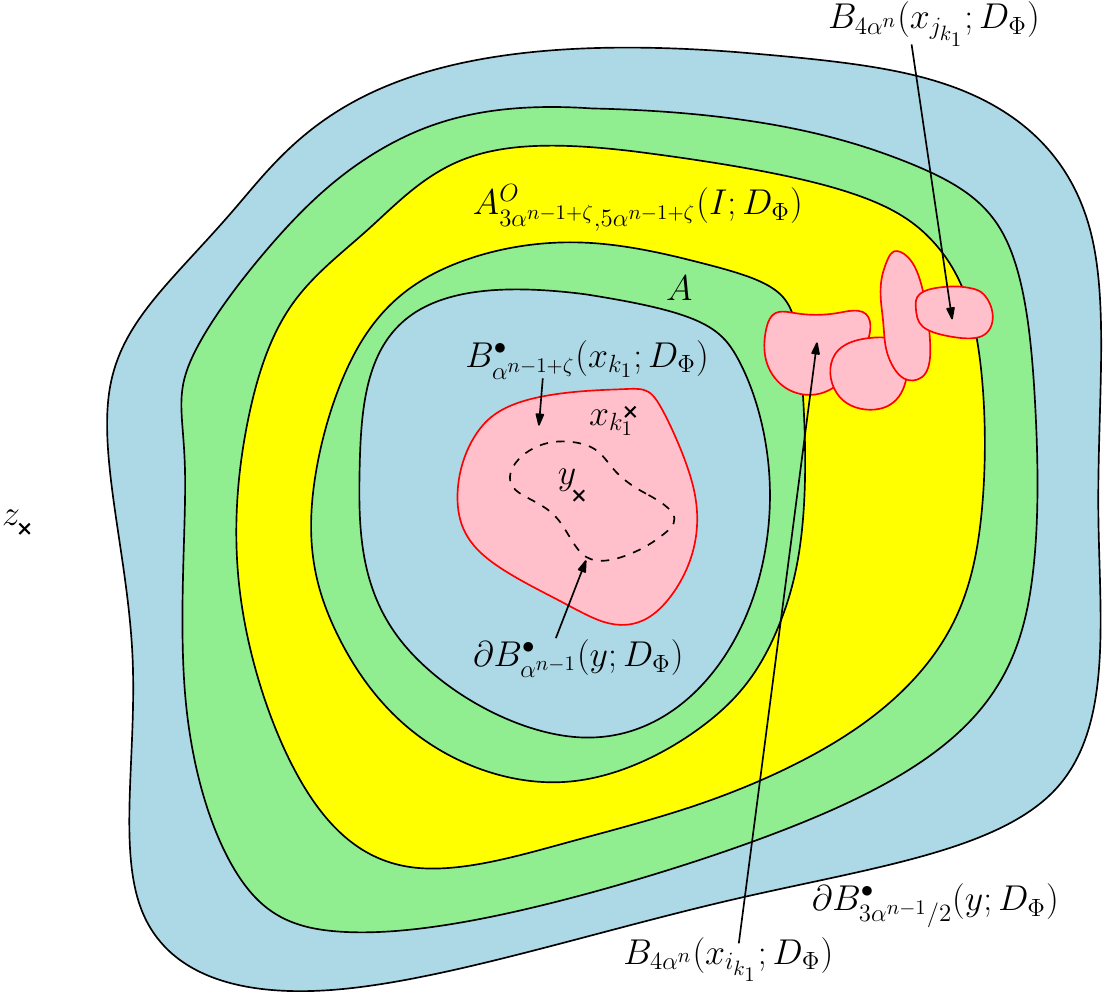}
    \caption{Illustration of the second of four cases in the proof of \Cref{lem:admissibility}. This figure depicts the scenario where $B_{3\alpha^{n - 1}/2}^z(y; D_\Phi)$ is bounded, but unlike the first case, $y$ is contained in at least one of the filled metric balls forming the minimal subpath (highlighted here as the central red ball). In this situation, we utilize the event $\widetilde F(x_{k_1}, n)$ to identify a suitable metric band $A$, represented in green. We then find a minimal subpath crossing the inner and outer boundaries of $A_{3\alpha^{n - 1 + \zeta},5\alpha^{n - 1 + \zeta}}^O(I; D_\Phi)$, shown in yellow, which satisfies condition~\eqref{eq:admissibility-illustration-0}.}
    \label{fig:admissibility2}
\end{figure}

\begin{proof}[Proof of \Cref{lem:admissibility}]
    By definition, we may assume without loss of generality that $\widetilde F(x_j, n)$ occurs and 
    \begin{equation*}
        \sigma(x_j, n) = \frac{\diam(B_{4\alpha^n}(x_j; D_\Phi))}{\inradius(B_{\alpha^{n - 1 + \zeta}}^\bullet(x_j; D_\Phi))}
    \end{equation*}
    for all $j \in [0, N]_\BZ$ (otherwise $\sigma(x_j, n) = 1$ for some $j \in [0, N]_\BZ$). Note that there exists $z \in \BC \setminus B_{2\alpha^{n - 1}}(y; D_\Phi)$ such that the path $B_{4\alpha^n}(x_0; D_\Phi), B_{4\alpha^n}(x_1; D_\Phi), \cdots, B_{4\alpha^n}(x_N; D_\Phi)$ crosses between the inner and outer boundaries of $A_{\alpha^{n - 1},2\alpha^{n - 1}}^z(y; D_\Phi)$.

    First, we consider the case where $B_{3\alpha^{n - 1}/2}^z(y; D_\Phi)$ is bounded, or, equivalently, $B_{3\alpha^{n - 1}/2}^z(y; D_\Phi) = B_{3\alpha^{n - 1}/2}^\bullet(y; D_\Phi)$. Write
    \begin{itemize}
        \item $j_1$ for the smallest $j \in [0, N]_\BZ$ such that $B_{4\alpha^n}(x_j; D_\Phi)$ intersects $\BC \setminus B_{\alpha^{n - 1} + 4\alpha^{n - 1 + \zeta}}^\bullet(y; D_\Phi)$; 
        \item $i_1$ for the largest $i \in [0, N]_\BZ$ smaller than $j_1$ such that $B_{4\alpha^n}(x_i; D_\Phi)$ intersects $B_{\alpha^{n - 1} + 2\alpha^{n - 1 + \zeta}}^\bullet(y; D_\Phi)$. 
    \end{itemize}
    First, we consider the case where $y \notin B_{\alpha^{n - 1 + \zeta}}^\bullet(x_j; D_\Phi)$ for all $j \in [i_1, j_1]_\BZ$ (cf.~\Cref{fig:admissibility1}). We \emph{claim} that 
    \begin{equation}\label{eq:admissibility-proof-1}
        \sum_{j \in [i_1, j_1]_\BZ} \sigma(x_j, n) = \sum_{j \in [i_1, j_1]_\BZ} \frac{\diam(B_{4\alpha^n}(x_j; D_\Phi))}{\inradius(B_{\alpha^{n - 1 + \zeta}}^\bullet(x_j; D_\Phi))} \ge 1.
    \end{equation}
    It suffices to show that 
    \begin{equation}\label{eq:admissibility-proof-2}
        \bigcup_{i \in [i_1, j_1]_\BZ} B_{4\alpha^n}(x_i; D_\Phi) \not\subset B_{\alpha^{n - 1 + \zeta}}^\bullet(x_j; D_\Phi), \quad \forall j \in [i_1, j_1]_\BZ. 
    \end{equation}
    Indeed, in this case, 
    \begin{multline*}
        \sum_{j \in [i_1, j_1]_\BZ} \frac{\diam(B_{4\alpha^n}(x_j; D_\Phi))}{\inradius(B_{\alpha^{n - 1 + \zeta}}^\bullet(x_j; D_\Phi))} \\
        \ge \sum_{j \in [i_1, j_1]_\BZ} \diam(B_{4\alpha^n}(x_j; D_\Phi)) \cdot \diam\left(\bigcup_{i \in [i_1, j_1]_\BZ} B_{4\alpha^n}(x_i; D_\Phi)\right)^{-1} \ge 1, 
    \end{multline*}
    where the last inequality follows from the fact that $\bigcup_{i \in [i_1, j_1]_\BZ} B_{4\alpha^n}(x_i; D_\Phi)$ is connected. Fix $j \in [i_1, j_1]_\BZ$. Since the inner and outer boundaries of $A_{\alpha^{n - 1} + 2\alpha^{n - 1 + \zeta},\alpha^{n - 1} + 4\alpha^{n - 1 + \zeta}}^\bullet(y; D_\Phi)$ have $D_\Phi$-distance $2\alpha^{n - 1 + \zeta}$, and the path
    \[ B_{4\alpha^n}(x_{i_1}; D_\Phi), B_{4\alpha^n}(x_{i_1 + 1}; D_\Phi), \cdots, B_{4\alpha^n}(x_{j_1}; D_\Phi)\]
    crosses between the inner and outer boundaries of $A_{\alpha^{n - 1} + 2\alpha^{n - 1 + \zeta},\alpha^{n - 1} + 4\alpha^{n - 1 + \zeta}}^\bullet(y; D_\Phi)$, it follows that either
    \begin{itemize}
        \item $B_{\alpha^{n - 1 + \zeta}}(x_j; D_\Phi)$ does not intersect $B_{\alpha^{n - 1} + 2\alpha^{n - 1 + \zeta}}^\bullet(y; D_\Phi)$, or
        \item $B_{\alpha^{n - 1 + \zeta}}(x_j; D_\Phi)$ does not intersect $\BC \setminus B_{\alpha^{n - 1} + 4\alpha^{n - 1 + \zeta}}^\bullet(y; D_\Phi)$. 
    \end{itemize}
    First, we consider the case where $B_{\alpha^{n - 1 + \zeta}}(x_j; D_\Phi)$ does not intersect $B_{\alpha^{n - 1} + 2\alpha^{n - 1 + \zeta}}^\bullet(y; D_\Phi)$. Since $B_{\alpha^{n - 1} + 2\alpha^{n - 1 + \zeta}}^\bullet(y; D_\Phi)$ is connected and $y \notin B_{\alpha^{n - 1 + \zeta}}^\bullet(x_j; D_\Phi)$, it follows that $B_{\alpha^{n - 1} + 2\alpha^{n - 1 + \zeta}}^\bullet(y; D_\Phi) \cap B_{\alpha^{n - 1 + \zeta}}^\bullet(x_j; D_\Phi) = \emptyset$. Since (by the definition of $i_1$) $B_{\alpha^{n - 1} + 2\alpha^{n - 1 + \zeta}}^\bullet(y; D_\Phi) \cap B_{4\alpha^n}(x_{i_1}; D_\Phi) \neq \emptyset$, it follows that $B_{4\alpha^n}(x_{i_1}; D_\Phi) \not\subset B_{\alpha^{n - 1 + \zeta}}^\bullet(x_j; D_\Phi)$. Next, we consider the case where $B_{\alpha^{n - 1 + \zeta}}(x_j; D_\Phi)$ does not intersect $(\BC \setminus B_{\alpha^{n - 1} + 4\alpha^{n - 1 + \zeta}}^\bullet(y; D_\Phi))$. Since $\BC \setminus B_{\alpha^{n - 1} + 4\alpha^{n - 1 + \zeta}}^\bullet(y; D_\Phi)$ is connected and unbounded, it follows that $(\BC \setminus B_{\alpha^{n - 1} + 4\alpha^{n - 1 + \zeta}}^\bullet(y; D_\Phi)) \cap B_{\alpha^{n - 1 + \zeta}}^\bullet(x_j; D_\Phi) = \emptyset$. Since (by the definition of $j_1$) $(\BC \setminus B_{\alpha^{n - 1} + 4\alpha^{n - 1 + \zeta}}^\bullet(y; D_\Phi)) \cap B_{4\alpha^n}(x_{j_1}; D_\Phi) \neq \emptyset$, it follows that $B_{4\alpha^n}(x_{j_1}; D_\Phi) \not\subset B_{\alpha^{n - 1 + \zeta}}^\bullet(x_j; D_\Phi)$. This completes the proof of~\eqref{eq:admissibility-proof-2}, hence also the proof of~\eqref{eq:admissibility-proof-1}. 

    Next, we consider the case where there exists $k_1 \in [i_1, j_1]$ such that $y \in B_{\alpha^{n - 1 + \zeta}}^\bullet(x_{k_1}; D_\Phi)$ (cf.~\Cref{fig:admissibility2}). We \emph{claim} that $A_{\alpha^{n - 1}/16,5\alpha^{n - 1}/16}^\bullet(x_{k_1}; D_\Phi)$ is contained in $A_{\alpha^{n - 1},3\alpha^{n - 1}/2}^\bullet(y; D_\Phi)$ and disconnects the inner and outer boundaries of $A_{\alpha^{n - 1},3\alpha^{n - 1}/2}^\bullet(y; D_\Phi)$, or, equivalently, that
    \begin{equation*}
        B_{\alpha^{n - 1}}^\bullet(y; D_\Phi) \subset B_{\alpha^{n - 1}/16}^\bullet(x_{k_1}; D_\Phi) \subset B_{5\alpha^{n - 1}/16}^\bullet(x_{k_1}; D_\Phi) \subset B_{3\alpha^{n - 1}/2}^\bullet(y; D_\Phi). 
    \end{equation*}
    Since $A_{\alpha^{n - 1} + 2\alpha^{n - 1 + \zeta},\alpha^{n - 1} + 4\alpha^{n - 1 + \zeta}}^\bullet(y; D_\Phi) \cap B_{4\alpha^n}(x_{k_1}; D_\Phi) \neq \emptyset$, it follows that $B_{\alpha^{n - 1}}^\bullet(y; D_\Phi) \cap B_{\alpha^{n - 1 + \zeta}}(x_{k_1}; D_\Phi) = \emptyset$ and $B_{5\alpha^{n - 1}/16}(x_{k_1}; D_\Phi) \subset B_{3\alpha^{n - 1}/2}^\bullet(y; D_\Phi)$. Since $B_{\alpha^{n - 1}}^\bullet(y; D_\Phi)$ is connected and $y \in B_{\alpha^{n - 1 + \zeta}}^\bullet(x_{k_1}; D_\Phi)$, it follows that $B_{\alpha^{n - 1}}^\bullet(y; D_\Phi) \subset B_{\alpha^{n - 1 + \zeta}}^\bullet(x_{k_1}; D_\Phi) \subset B_{\alpha^{n - 1}/16}^\bullet(x_{k_1}; D_\Phi)$. Since $B_{3\alpha^{n - 1}/2}^\bullet(y; D_\Phi)$ is simply connected, it follows that $B_{5\alpha^{n - 1}/16}^\bullet(x_{k_1}; D_\Phi) \subset B_{3\alpha^{n - 1}/2}^\bullet(y; D_\Phi)$. This completes the proof of the \emph{claim}. 

    Since the event $\widetilde F(x_{k_1}, n)$ occurs, by definition, there exists a doubly connected domain $A$ that is contained in $A_{\alpha^{n - 1}/16,5\alpha^{n - 1}/16}^\bullet(x_{k_1}; D_\Phi)$ and disconnects its inner and outer boundaries such that~\eqref{eq:G} is satisfied with $w = 8\alpha^{n - 1 + \zeta}$, and the event $G$ of the $\gamma$-LQG surface parameterized by $A$ occurs. Write $B$ for the connected component of $\BC \setminus A$ that contains $y$. The \emph{claim} of the preceding paragraph implies that
    \begin{equation*}
        B_{\alpha^{n - 1}}^\bullet(y; D_\Phi) \subset B_{\alpha^{n - 1}/16}^\bullet(x_{k_1}; D_\Phi) \subset B \subset B \cup A \subset B_{5\alpha^{n - 1}/16}^\bullet(x_{k_1}; D_\Phi) \subset B_{3\alpha^{n - 1}/2}^\bullet(y; D_\Phi). 
    \end{equation*}
    Thus, the path $B_{4\alpha^n}(x_0; D_\Phi), B_{4\alpha^n}(x_1; D_\Phi), \cdots, B_{4\alpha^n}(x_N; D_\Phi)$ crosses between $I$ and $O$. (By convention, $I$ (resp.~$O$) is the boundary of $B$ (resp.~the unbounded connected component of $\BC \setminus A$).) Thus, we may write 
    \begin{itemize}
        \item $j_{k_1}$ for the smallest $j \in [0, N]_\BZ$ such that $B_{4\alpha^n}(x_j; D_\Phi)$ intersects $\partial B_{5\alpha^{n - 1 + \zeta}}^O(I; D_\Phi)$; 
        \item $i_{k_1}$ for the largest $i \in [0, N]_\BZ$ smaller than $j_{k_1}$ such that $B_{4\alpha^n}(x_i; D_\Phi)$ intersects $\partial B_{3\alpha^{n - 1 + \zeta}}^O(I; D_\Phi)$. 
    \end{itemize}
    We \emph{claim} that 
    \begin{equation}\label{eq:admissibility-proof-5}
        \sum_{j \in [i_{k_1}, j_{k_1}]_\BZ} \sigma(x_j, n) = \sum_{j \in [i_{k_1}, j_{k_1}]_\BZ} \frac{\diam(B_{4\alpha^n}(x_j; D_\Phi))}{\inradius(B_{\alpha^{n - 1 + \zeta}}^\bullet(x_j; D_\Phi))} \ge 1.
    \end{equation}
    By a similar argument to the argument applied in the proof of~\eqref{eq:admissibility-proof-1}, it suffices to show that
    \begin{equation}\label{eq:admissibility-proof-4}
        \bigcup_{i \in [i_{k_1}, j_{k_1}]_\BZ} B_{4\alpha^n}(x_i; D_\Phi) \not\subset B_{\alpha^{n - 1 + \zeta}}^\bullet(x_j; D_\Phi), \quad \forall j \in [i_{k_1}, j_{k_1}]_\BZ. 
    \end{equation}
    For each $t \in [0, 8\alpha^{n - 1 + \zeta}]$, write $B_t \defeq B \cup A_{0,t}^O(I; D_\Phi)$. Since the inner and outer boundaries of $A_{3\alpha^{n - 1 + \zeta},5\alpha^{n - 1 + \zeta}}^O(I; D_\Phi)$ have $D_\Phi$-distance $2\alpha^{n - 1 + \zeta}$, and the path
    \[ B_{4\alpha^n}(x_{i_{k_1}}; D_\Phi), B_{4\alpha^n}(x_{i_{k_1} + 1}; D_\Phi), \cdots, B_{4\alpha^n}(x_{j_{k_1}}; D_\Phi)\]
    crosses between the inner and outer boundaries of $A_{3\alpha^{n - 1 + \zeta},5\alpha^{n - 1 + \zeta}}^O(I; D_\Phi)$, it follows that either
    \begin{itemize}
        \item $B_{3\alpha^{n - 1 + \zeta}} \cap B_{\alpha^{n - 1 + \zeta}}(x_j; D_\Phi) = \emptyset$, or
        \item $(\BC \setminus B_{5\alpha^{n - 1 + \zeta}}) \cap B_{\alpha^{n - 1 + \zeta}}(x_j; D_\Phi) = \emptyset$. 
    \end{itemize}
    For simplicity, we consider the first case; the second case is entirely similar. Choose $w \in \partial B_{3\alpha^{n - 1 + \zeta}} \cap B_{4\alpha^n}(x_{i_{k_1}}; D_\Phi)$. (By definition, $\partial B_{3\alpha^{n - 1 + \zeta}} \cap B_{4\alpha^n}(x_{i_{k_1}}; D_\Phi) = \partial B_{3\alpha^{n - 1 + \zeta}}^O(I; D_\Phi) \cap B_{4\alpha^n}(x_{i_{k_1}}; D_\Phi) \neq \emptyset$.) Note that there is a $D_\Phi$-geodesic from $w$ to $I$ that is contained in $B_{3\alpha^{n - 1 + \zeta}}$. Since the event $G$ of the $\gamma$-LQG surface parameterized by $A$ occurs, by definition, $B_{\alpha^{n - 1 + \zeta}}(x_j; D_\Phi)$ does not disconnect $I$ and $O$. (One verifies immediately that $x_j \in A_{2\alpha^{n - 1 + \zeta},6\alpha^{n - 1 + \zeta}}^O(I; D_\Phi)$ for all $j \in [i_{k_1}, j_{k_1}]_\BZ$.) Thus, we conclude that $w$ and $I$ and $O$ are contained in the same connected component of $\BC \setminus B_{\alpha^{n - 1 + \zeta}}(x_j; D_\Phi)$. Since $A \subset A_{\alpha^{n - 1}/16,5\alpha^{n - 1}/16}^\bullet(x_{k_1}; D_\Phi)$ is bounded and $B_{\alpha^{n - 1 + \zeta}}(x_j; D_\Phi) \subset A$, it follows that this connected component contains $\BC \setminus A$, hence is unbounded. Thus, we conclude that $w \notin B_{\alpha^{n - 1 + \zeta}}^\bullet(x_j; D_\Phi)$. This completes the proof of~\eqref{eq:admissibility-proof-4}, hence also the proof of~\eqref{eq:admissibility-proof-5}. This completes the proof of the case where $B_{3\alpha^{n - 1}/2}^z(y; D_\Phi)$ is bounded. 

    Next, we consider the case where $B_{3\alpha^{n - 1}/2}^z(y; D_\Phi)$ is unbounded. Write
    \begin{itemize}
        \item $j_2$ for the smallest $j \in [0, N]_\BZ$ such that $B_{4\alpha^n}(x_j; D_\Phi)$ intersects $\BC \setminus B_{2\alpha^{n - 1} - 2\alpha^{n - 1 + \zeta}}^z(y; D_\Phi)$; 
        \item $i_2$ for the largest $i \in [0, N]_\BZ$ smaller than $j_2$ such that $B_{4\alpha^n}(x_i; D_\Phi)$ intersects $B_{2\alpha^{n - 1} - 4\alpha^{n - 1 + \zeta}}^z(y; D_\Phi)$. 
    \end{itemize}
    First, we consider the case where $z \notin B_{\alpha^{n - 1 + \zeta}}^\bullet(x_j; D_\Phi)$ for all $j \in [i_2, j_2]_\BZ$. We \emph{claim} that 
    \begin{equation}\label{eq:admissibility-proof-0}
        \sum_{j \in [i_2, j_2]_\BZ} \sigma(x_j, n) = \sum_{j \in [i_2, j_2]_\BZ} \frac{\diam(B_{4\alpha^n}(x_j; D_\Phi))}{\inradius(B_{\alpha^{n - 1 + \zeta}}^\bullet(x_j; D_\Phi))} \ge 1.
    \end{equation}
    By a similar argument to the argument applied in the proof of~\eqref{eq:admissibility-proof-1}, it suffices to show that
    \begin{equation}\label{eq:admissibility-proof-3}
        \bigcup_{i \in [i_2, j_2]_\BZ} B_{4\alpha^n}(x_i; D_\Phi) \not\subset B_{\alpha^{n - 1 + \zeta}}^\bullet(x_j; D_\Phi), \quad \forall j \in [i_2, j_2]_\BZ. 
    \end{equation}
    Fix $j \in [i_2, j_2]_\BZ$. Since the inner and outer boundaries of $A_{2\alpha^{n - 1} - 4\alpha^{n - 1 + \zeta},2\alpha^{n - 1} - 2\alpha^{n - 1 + \zeta}}^z(y; D_\Phi)$ have $D_\Phi$-distance $2\alpha^{n - 1 + \zeta}$, and the path $B_{4\alpha^n}(x_{i_2}; D_\Phi), B_{4\alpha^n}(x_{i_2 + 1}; D_\Phi), \cdots, B_{4\alpha^n}(x_{j_2}; D_\Phi)$ crosses between the inner and outer boundaries of $A_{2\alpha^{n - 1} - 4\alpha^{n - 1 + \zeta},2\alpha^{n - 1} - 2\alpha^{n - 1 + \zeta}}^z(y; D_\Phi)$, it follows that either
    \begin{itemize}
        \item $B_{\alpha^{n - 1 + \zeta}}(x_j; D_\Phi)$ does not intersect $B_{2\alpha^{n - 1} - 4\alpha^{n - 1 + \zeta}}^z(y; D_\Phi)$, or
        \item $B_{\alpha^{n - 1 + \zeta}}(x_j; D_\Phi)$ does not intersect $\BC \setminus B_{2\alpha^{n - 1} - 2\alpha^{n - 1 + \zeta}}^z(y; D_\Phi)$. 
    \end{itemize}
    First, we consider the case where $B_{\alpha^{n - 1 + \zeta}}(x_j; D_\Phi)$ does not intersect $B_{2\alpha^{n - 1} - 4\alpha^{n - 1 + \zeta}}^z(y; D_\Phi)$. Since $B_{2\alpha^{n - 1} - 4\alpha^{n - 1 + \zeta}}^z(y; D_\Phi)$ is connected and contains $B_{3\alpha^{n - 1}/2}^z(y; D_\Phi)$ (hence is unbounded), it follows that $B_{2\alpha^{n - 1} - 4\alpha^{n - 1 + \zeta}}^z(y; D_\Phi) \cap B_{\alpha^{n - 1 + \zeta}}^\bullet(x_j; D_\Phi) = \emptyset$. However, since (by the definition of $i_2$) $B_{2\alpha^{n - 1} - 4\alpha^{n - 1 + \zeta}}^z(y; D_\Phi) \cap B_{4\alpha^n}(x_{i_2}; D_\Phi) \neq \emptyset$, it follows that $B_{4\alpha^n}(x_{i_2}; D_\Phi) \not\subset B_{\alpha^{n - 1 + \zeta}}^\bullet(x_j; D_\Phi)$. Next, we consider the case where $B_{\alpha^{n - 1 + \zeta}}(x_j; D_\Phi)$ does not intersect $\BC \setminus B_{2\alpha^{n - 1} - 2\alpha^{n - 1 + \zeta}}^z(y; D_\Phi)$. Since $\BC \setminus B_{2\alpha^{n - 1} - 2\alpha^{n - 1 + \zeta}}^z(y; D_\Phi)$ is connected and (by assumption) $z \notin B_{\alpha^{n - 1 + \zeta}}^\bullet(x_j; D_\Phi)$, it follows that $(\BC \setminus B_{2\alpha^{n - 1} - 2\alpha^{n - 1 + \zeta}}^z(y; D_\Phi)) \cap B_{\alpha^{n - 1 + \zeta}}^\bullet(x_j; D_\Phi) = \emptyset$. However, since (by the definition of $j_2$) $(\BC \setminus B_{2\alpha^{n - 1} - 2\alpha^{n - 1 + \zeta}}^z(y; D_\Phi)) \cap B_{4\alpha^n}(x_{j_2}; D_\Phi) \neq \emptyset$, it follows that $B_{4\alpha^n}(x_{j_2}; D_\Phi) \not\subset B_{\alpha^{n - 1 + \zeta}}^\bullet(x_j; D_\Phi)$. This completes the proof of~\eqref{eq:admissibility-proof-3}, hence also the proof of~\eqref{eq:admissibility-proof-0}. 

    Next, we consider the case where there exists $k_2 \in [i_2, j_2]$ such that $z \in B_{\alpha^{n - 1 + \zeta}}^\bullet(x_{k_2}; D_\Phi)$. We \emph{claim} that $A_{\alpha^{n - 1}/16,5\alpha^{n - 1}/16}^\bullet(x_{k_2}; D_\Phi)$ is contained in $A_{3\alpha^{n - 1}/2,2\alpha^{n - 1}}^z(y; D_\Phi)$ and disconnects the inner and outer boundaries of $A_{3\alpha^{n - 1}/2,2\alpha^{n - 1}}^z(y; D_\Phi)$, or, equivalently, that
    \begin{equation*}
        \BC \setminus B_{2\alpha^{n - 1}}^z(y; D_\Phi) \subset B_{\alpha^{n - 1}/16}^\bullet(x_{k_2}; D_\Phi) \subset B_{5\alpha^{n - 1}/16}^\bullet(x_{k_2}; D_\Phi) \subset \BC \setminus B_{3\alpha^{n - 1}/2}^z(y; D_\Phi). 
    \end{equation*}
    Since $A_{2\alpha^{n - 1} - 4\alpha^{n - 1 + \zeta},2\alpha^{n - 1} - 2\alpha^{n - 1 + \zeta}}^z(y; D_\Phi) \cap B_{4\alpha^n}(x_{k_2}; D_\Phi) \neq \emptyset$, we have $(\BC \setminus B_{2\alpha^{n - 1}}^z(y; D_\Phi)) \cap B_{\alpha^{n - 1 + \zeta}}(x_{k_1}; D_\Phi) = \emptyset$ and $B_{5\alpha^{n - 1}/16}(x_{k_1}; D_\Phi) \subset \BC \setminus B_{3\alpha^{n - 1}/2}^z(y; D_\Phi)$. Since $\BC \setminus B_{2\alpha^{n - 1}}^z(y; D_\Phi)$ is connected and $z \in B_{\alpha^{n - 1 + \zeta}}^\bullet(x_{k_2}; D_\Phi)$, it follows that $\BC \setminus B_{2\alpha^{n - 1}}^z(y; D_\Phi) \subset B_{\alpha^{n - 1 + \zeta}}^\bullet(x_{k_2}; D_\Phi) \subset B_{\alpha^{n - 1}/16}^\bullet(x_{k_2}; D_\Phi)$. Since $\BC \setminus B_{3\alpha^{n - 1}/2}^z(y; D_\Phi)$ is simply connected, we have $B_{5\alpha^{n - 1}/16}^\bullet(x_{k_2}; D_\Phi) \subset \BC \setminus B_{3\alpha^{n - 1}/2}^z(y; D_\Phi)$, This completes the proof of the \emph{claim}. 

    Since the event $\widetilde F(x_{k_2}, n)$ occurs, by definition, there exists a doubly connected domain $A$ that is contained in $A_{\alpha^{n - 1}/16,5\alpha^{n - 1}/16}^\bullet(x_{k_2}; D_\Phi)$ and disconnects its inner and outer boundaries such that~\eqref{eq:G} is satisfied with $w = 8\alpha^{n - 1 + \zeta}$, and the event $G$ of the $\gamma$-LQG surface parameterized by $A$ occurs. Write $B$ for the connected component of $\BC \setminus A$ that contains $z$. The \emph{claim} of the preceding paragraph implies that
    \begin{equation*}
        \BC \setminus B_{2\alpha^{n - 1}}^z(y; D_\Phi) \subset B_{\alpha^{n - 1}/16}^\bullet(x_{k_2}; D_\Phi) \subset B \subset B \cup A \subset B_{5\alpha^{n - 1}/16}^\bullet(x_{k_2}; D_\Phi) \subset \BC \setminus B_{3\alpha^{n - 1}/2}^z(y; D_\Phi). 
    \end{equation*}
    Thus, the path $B_{4\alpha^n}(x_0; D_\Phi), B_{4\alpha^n}(x_1; D_\Phi), \cdots, B_{4\alpha^n}(x_N; D_\Phi)$ crosses between $I$ and $O$. (By convention, $I$ (resp.~$O$) is the boundary of $B$ (resp.~the unbounded connected component of $\BC \setminus A$).) Thus, we may write 
    \begin{itemize}
        \item $j_{k_2}$ for the smallest $j \in [0, N]_\BZ$ such that $B_{4\alpha^n}(x_j; D_\Phi)$ intersects $\partial B_{3\alpha^{n - 1 + \zeta}}^O(I; D_\Phi)$; 
        \item $i_{k_2}$ for the largest $i \in [0, N]_\BZ$ smaller than $j_{k_2}$ such that $B_{4\alpha^n}(x_i; D_\Phi)$ intersects $\partial B_{5\alpha^{n - 1 + \zeta}}^O(I; D_\Phi)$. 
    \end{itemize}
    We \emph{claim} that 
    \begin{equation}\label{eq:admissibility-proof-6}
        \sum_{j \in [i_{k_2}, j_{k_2}]_\BZ} \sigma(x_j, n) = \sum_{j \in [i_{k_2}, j_{k_2}]_\BZ} \frac{\diam(B_{4\alpha^n}(x_j; D_\Phi))}{\inradius(B_{\alpha^{n - 1 + \zeta}}^\bullet(x_j; D_\Phi))} \ge 1.
    \end{equation}
    By a similar argument to the argument applied in the proof of~\eqref{eq:admissibility-proof-1}, it suffices to show that
    \begin{equation}\label{eq:admissibility-proof-7}
        \bigcup_{i \in [i_{k_2}, j_{k_2}]_\BZ} B_{4\alpha^n}(x_i; D_\Phi) \not\subset B_{\alpha^{n - 1 + \zeta}}^\bullet(x_j; D_\Phi), \quad \forall j \in [i_{k_2}, j_{k_2}]_\BZ. 
    \end{equation}
    For each $t \in [0, 8\alpha^{n - 1 + \zeta}]$, write $B_t \defeq B \cup A_{0,t}^O(I; D_\Phi)$. Since the inner and outer boundaries of $A_{3\alpha^{n - 1 + \zeta},5\alpha^{n - 1 + \zeta}}^O(I; D_\Phi)$ have $D_\Phi$-distance $2\alpha^{n - 1 + \zeta}$, and the path 
    \begin{equation*}
        B_{4\alpha^n}(x_{i_{k_2}}; D_\Phi), B_{4\alpha^n}(x_{i_{k_2} + 1}; D_\Phi), \cdots, B_{4\alpha^n}(x_{j_{k_2}}; D_\Phi). 
    \end{equation*}
    crosses between the inner and outer boundaries of $A_{3\alpha^{n - 1 + \zeta},5\alpha^{n - 1 + \zeta}}^O(I; D_\Phi)$, it follows that either
    \begin{itemize}
        \item $B_{3\alpha^{n - 1 + \zeta}} \cap B_{\alpha^{n - 1 + \zeta}}(x_j; D_\Phi) = \emptyset$, or
        \item $(\BC \setminus B_{5\alpha^{n - 1 + \zeta}}) \cap B_{\alpha^{n - 1 + \zeta}}(x_j; D_\Phi) = \emptyset$. 
    \end{itemize}
    For simplicity, we consider the first case; the second case is entirely similar. Choose $w \in \partial B_{3\alpha^{n - 1 + \zeta}} \cap B_{4\alpha^n}(x_{j_{k_2}}; D_\Phi)$. (By definition, $\partial B_{3\alpha^{n - 1 + \zeta}} \cap B_{4\alpha^n}(x_{j_{k_2}}; D_\Phi) = \partial B_{3\alpha^{n - 1 + \zeta}}^O(I; D_\Phi) \cap B_{4\alpha^n}(x_{j_{k_2}}; D_\Phi) \neq \emptyset$.) Note that there is a $D_\Phi$-geodesic from $w$ to $I$ that is contained in $B_{3\alpha^{n - 1 + \zeta}}$. Since the event $G$ of the $\gamma$-LQG surface parameterized by $A$ occurs, by definition, $B_{\alpha^{n - 1 + \zeta}}(x_j; D_\Phi)$ does not disconnect $I$ and $O$. (One verifies immediately that $x_j \in A_{2\alpha^{n - 1 + \zeta},6\alpha^{n - 1 + \zeta}}^O(I; D_\Phi)$ for all $j \in [i_{k_2}, j_{k_2}]_\BZ$.) Thus, we conclude that $w$ and $I$ and $O$ are contained in the same connected component of $\BC \setminus B_{\alpha^{n - 1 + \zeta}}(x_j; D_\Phi)$. Since $A \subset A_{3\alpha^{n - 1}/2,2\alpha^{n - 1}}^z(y; D_\Phi)$ is bounded (since, by assumption, $B_{3\alpha^{n - 1}/2}^z(y; D_\Phi)$ is unbounded) and $B_{\alpha^{n - 1 + \zeta}}(x_j; D_\Phi) \subset A$, it follows that this connected component contains $\BC \setminus A$, hence is unbounded. Thus, we conclude that $w \notin B_{\alpha^{n - 1 + \zeta}}^\bullet(x_j; D_\Phi)$. This completes the proof of~\eqref{eq:admissibility-proof-7}, hence also the proof of~\eqref{eq:admissibility-proof-6}. This completes the proof of \Cref{lem:admissibility}. 
\end{proof}

\begin{definition}\label{def:tilde-rho}
    Let $x \in \BC$ and $n \in \BN$. Then we shall write
    \begin{equation*}
        \varsigma(x, n) \defeq (\eta + 128\re^{-2\pi m_n} + 2 \cdot \mathbf 1_{F(x, n)^c}) \wedge 1, 
    \end{equation*}
    where $m_n$ denotes the conformal modulus of $A_{32\alpha^n,\alpha^{n - 1 + \zeta}/2}^\bullet(x; D_\Phi)$; we shall write
    \begin{equation*}
        \varpi(x, n) \defeq \prod_{j = 1}^n \varsigma(x, j). 
    \end{equation*}
\end{definition}

\begin{lemma}\label{lem:pi-upper-bound}
    In the notation of \Cref{section:hyperbolic-filling}, $\pi(u) \le \varpi(x, n)$ for all $n \in \BN$ and $u = (x, n) \in V_n$. 
\end{lemma}

\begin{proof}
    Recall from~\eqref{eq:rho} that $\pi(u) = \prod_{j = 1}^n \varrho(g(u)_j)$. Thus, it suffices to show that 
    \begin{equation}\label{eq:pi-upper-bound-proof-2}
        \varrho(g(u)_j) \le \varsigma(x, j), \quad \forall j \in [1, n]_\BZ. 
    \end{equation}
    For each $j \in [1, n]_\BZ$, write $g(u)_j = (x_j, j)$. We have
    \begin{align*}
        \varrho(g(u)_j) &\le \sup\{\mu(v) : v \in V_j \text{ with } g(u)_j \sim v\} \quad \text{(by \Cref{lem:H1-H2}, \eqref{it:H1})} \\
        &\le \eta + \sup\{\nu(v) : v \in V_j \text{ with } g(u)_j \sim v\} \quad \text{(by~\eqref{eq:definition-nu-mu})} \\
        &\le \eta + 2\sup\{\sigma(v^\pprime) : v, v^\prime, v^\pprime \in V_j \text{ with } g(u)_j \sim v \sim v^\prime \sim v^\pprime\} \quad \text{(by~\eqref{eq:definition-nu-mu})} \\
        &\le \eta + 2\sup\{\sigma(y, j) : y \in \BC \text{ with } D_\Phi(x_j, y) \le 24\alpha^j\}, 
    \end{align*}
    where the last inequality follows from the fact that if $(x_j, j) \sim (y, j)$ is a horizontal edge, then $B_{4\alpha^j}(x_j; D_\Phi)$ and $B_{4\alpha^j}(y; D_\Phi)$ intersect. Note that
    \begin{equation*}
        D_\Phi(x, x_j) = D_\Phi(x_n, x_j) \le \sum_{k = j}^{n - 1} D_\Phi(x_k, x_{k + 1}) \le \sum_{k = j}^{n - 1} \alpha^k < 2\alpha^j, 
    \end{equation*}
    where the penultimate inequality follows from~\eqref{eq:parent}. Thus, we obtain that
    \begin{equation*}
        \varrho(g(u)_j) \le \eta + 2\sup\{\sigma(y, j) : y \in \BC \text{ with } D_\Phi(x, y) \le 26\alpha^j\}. 
    \end{equation*}
    We have
    \begin{align*}
        \sigma(y, j) &= \left(\frac{\diam(B_{2\alpha^j}(y; D_\Phi))}{\inradius(B_{\alpha^{j - 1 + \zeta}}^\bullet(y; D_\Phi))} \wedge 1\right)\mathbf 1_{\widetilde F(y, j)} + \mathbf 1_{\widetilde F(y, j)^c} \quad \text{(by definition)} \\
        &\le \left(\frac{\diam(B_{2\alpha^j}(y; D_\Phi))}{\inradius(B_{\alpha^{j - 1 + \zeta}}^\bullet(y; D_\Phi))} \wedge 1\right) + \mathbf 1_{\widetilde F(y, j)^c} \\
        &\le \left(\frac{\diam(B_{28\alpha^j}(x; D_\Phi))}{\inradius(B_{\alpha^{j - 1 + \zeta}/2}^\bullet(x; D_\Phi))} \wedge 1\right) + \mathbf 1_{\widetilde F(y, j)^c} \quad \text{(since $26\alpha^j \le \alpha^{j - 1 + \zeta}/2$)} \\
        &\le \left(\frac{\diam(B_{28\alpha^j}(x; D_\Phi))}{\inradius(B_{\alpha^{j - 1 + \zeta}/2}^\bullet(x; D_\Phi))} \wedge 1\right) + \mathbf 1_{F(x, j)^c} \quad \text{(by \Cref{lem:F-2-tilde-F})} \\
        &\le \left(\frac{\diam(B_{32\alpha^j}(x; D_\Phi))}{\inradius(B_{\alpha^{j - 1 + \zeta}/2}^\bullet(x; D_\Phi))} \wedge 1\right) + \mathbf 1_{F(x, j)^c}
    \end{align*}
    for all $y \in \BC$ with $D_\Phi(x, y) \le 26\alpha^j$. Thus, we obtain that
    \begin{equation*}
        \varrho(g(u)_j) \le \eta + 2\left(\frac{\diam(B_{32\alpha^j}(x; D_\Phi))}{\inradius(B_{\alpha^{j - 1 + \zeta}/2}^\bullet(x; D_\Phi))} \wedge 1\right) + 2 \cdot \mathbf 1_{F(x, j)^c}. 
    \end{equation*}
    We have
    \begin{align*}
        \frac{\diam(B_{32\alpha^j}(x; D_\Phi))}{\inradius(B_{\alpha^{j - 1 + \zeta}/2}^\bullet(x; D_\Phi))} \wedge 1 &\le \frac{2\outradius(B_{32\alpha^j}^\bullet(x; D_\Phi))}{\inradius(B_{\alpha^{j - 1 + \zeta}/2}^\bullet(x; D_\Phi))} \wedge 1 \\
        &\le 2\left((\re^{2\pi m_j}/16 - 1) \vee 0\right)^{-1} \wedge 1 \quad \text{(by \Cref{lem:Teichmueller})} \\
        &\le 64 \re^{-2\pi m_j}, 
    \end{align*}
    where $\outradius(B_{32\alpha^j}^\bullet(x; D_\Phi)) \defeq \inf\{R > 0 : B_{32\alpha^j}^\bullet(x; D_\Phi) \subset B_R(0)\}$ and the last inequality follows from the fact that $2((x - 1) \vee 0)^{-1} \wedge 1 \le 4x^{-1}$ for all $x > 0$. Thus, we obtain that
    \begin{equation*}
        \varrho(g(u)_j) \le \eta + 128\re^{-2\pi m_j} + 2 \cdot \mathbf 1_{F(x, j)^c}. 
    \end{equation*}
    Combining this with the fact that $\varrho(g(u)_j) \le 1$ (cf.~\Cref{lem:H1-H2}, \eqref{it:H1}), we complete the proof of~\eqref{eq:pi-upper-bound-proof-2}, hence the proof of \Cref{lem:pi-upper-bound}. 
\end{proof}

\begin{lemma}\label{lem:pi-expectation}
    Set $\eta \defeq \alpha^{100}$. Suppose that $(\BC, \Phi; 0, \infty)$ is a $\gamma$-LQG cone. Then for each $p > 2$, there exists $q = q(p) > d_\gamma$ and $\zeta_\ast = \zeta_\ast(p) > 0$ such that for each $\zeta \in (0, \zeta_\ast]$, there exists $\alpha_\ast = \alpha_\ast(\zeta) \in (0, 1)$ such that 
    \begin{equation*}
        \BE\lbrack\varpi(0, n)^p\rbrack \le \alpha^{qn}, \quad \forall \alpha \in (0, \alpha_\ast], \ \forall n \in \BN. 
    \end{equation*}
\end{lemma}

\begin{proof}
    Fix $p > 2$. For each $t \ge 0$, write $Y_{-t} \defeq \SN_\Phi(\partial B_t^\bullet(0; D_\Phi))$. Write $E$ for the event that $\#\{j \in [1, n]_\BZ : Y_{-32\alpha^j} \le \alpha^{2j - 2\zeta}\} \ge (1 - \zeta)n$. By \Cref{lem:CSBP}, $\BP\lbrack E^c\rbrack \le \exp(-\alpha^{-\zeta}n)$ for all sufficiently small $\alpha \in (0, 1)$ and all $n \in \BN$. Recall that
    \begin{equation*}
        \varpi(0, n) = \prod_{j = 1}^n \left((\eta + 128\re^{-2\pi m_j} + 2 \cdot \mathbf 1_{F(0, j)^c}) \wedge 1\right),
    \end{equation*}
    where $m_j$ denotes the conformal modulus of $A_{32\alpha^j,\alpha^{j - 1 + \zeta}/2}^\bullet(0; D_\Phi)$.
    Note that 
    \begin{multline*}
        \BE\!\left\lbrack(\eta + 128\re^{-2\pi m_j} + 2 \cdot \mathbf 1_{F(0, j)^c})^p \ \middle\vert \ Y_{-32\alpha^j}\right\rbrack \\
        \le 3^{p - 1}\left(\eta^p + 2^{7p} \BE\lbrack\re^{-2\pi m_jp} \mid Y_{-32\alpha^j}\rbrack + 2^p \BP\lbrack F(0, j)^c \mid Y_{-32\alpha^j}\rbrack\right). 
    \end{multline*}
    By \Cref{lem:F-probability}, $\sup_{\ell \ge 0} \BP\lbrack F(0, j)^c \mid Y_{-32\alpha^j} = \ell\rbrack = O(\alpha^\infty)$ as $\alpha \to 0$. By the scaling property, the conditional law of $m_j$ given $Y_{-32\alpha^j}$ is identical to the conditional law of the conformal modulus of a metric band of cone type with inner boundary length one and width $(\alpha^{j - 1 + \zeta}/2 - 32\alpha^j)Y_{-32\alpha^j}^{-1/2}$. Write $m$ for the conformal modulus of a metric band of cone type with inner boundary length one and width $(\alpha^{j - 1 + \zeta}/2 - 32\alpha^j)\alpha^{-(j - \zeta)} \ge \alpha^{2\zeta - 1}/4$. Thus, almost surely on the event that $Y_{-32\alpha^j} \le \alpha^{2j - 2\zeta}$, we have $\BE\lbrack\re^{-2\pi m_jp} \mid Y_{-32\alpha^j}\rbrack \le \BE\lbrack\re^{-2\pi mp}\rbrack$. By \Cref{lem:conformal-moduli-band}, there exists $q = q(p) > d_\gamma$ and $\zeta_\ast = \zeta_\ast(p) \in (0, 1 - d_\gamma/q)$ such that $\BE\lbrack\re^{-2\pi mp}\rbrack = o(\alpha^q)$ as $\alpha \to 0$ whenever $\zeta \in (0, \zeta_\ast]$. Thus, we conclude that there exists $\alpha_\ast = \alpha_\ast(\zeta) \in (0, 1)$ such that for each $\alpha \in (0, \alpha_\ast]$, almost surely on the event that $Y_{-32\alpha^j} \le \alpha^{2j - 2\zeta}$, 
    \begin{equation*}
        \BE\!\left\lbrack(\eta + 128\re^{-2\pi m_j} + 2 \cdot \mathbf 1_{F(0, j)^c})^p \ \middle\vert \ Y_{-32\alpha^j}\right\rbrack \le \alpha^q. 
    \end{equation*}
    For each $k \in \BN$, write $\alpha^{j_k}$ for the $k$-th smallest number in $\{\alpha^j\}_{j \in (-\infty, n]_\BZ}$ for which $Y_{-32\alpha^j} \le \alpha^{2j - 2\zeta}$. By definition, $E = \{j_{\lceil(1 - \zeta)n\rceil} \ge 1\}$. Note that $\eta + 128\re^{-2\pi m_j} + 2 \cdot \mathbf 1_{F(0, j)^c}$ is almost surely determined by the $\gamma$-LQG surfaces parameterized by $A_{32\alpha^j,32\alpha^{j - 1}}^\bullet(0; D_\Phi)$ (cf.~\Cref{lem:F-measurability}). Since the $\gamma$-LQG surfaces parameterized by $B_{32\alpha^j}^\bullet(0; D_\Phi)$ and $\BC \setminus B_{32\alpha^j}^\bullet(0; D_\Phi)$ are conditionally independent given $Y_{-32\alpha^j}$, it follows that
    \begin{equation*}
        \BE\!\left\lbrack\prod_{k = 1}^n (\eta + 128\re^{-2\pi m_{j_k}} + 2 \cdot \mathbf 1_{F(0, j_k)^c})^p\right\rbrack \le \alpha^{qn}
    \end{equation*}
    for all $\alpha \in (0, \alpha_\ast]$ and all $n \in \BN$. Thus, we conclude that
    \begin{align*}
        \BE\lbrack\varpi(0, n)^p\rbrack &\le \BE\lbrack\varpi(0, n)^p\mathbf 1_E\rbrack + \BP\lbrack E^c\rbrack \\
        &\le \BE\!\left\lbrack\prod_{k = 1}^{\lceil(1 - \zeta)n\rceil} (\eta + 128\re^{-2\pi m_{j_k}} + 2 \cdot \mathbf 1_{F(0, j_k)^c})^p\right\rbrack + \exp(-\alpha^{-\zeta}n) \\
        &\le \alpha^{q(1 - \zeta)n} + \exp(-\alpha^{-\zeta}n)
    \end{align*}
    for all $\alpha \in (0, \alpha_\ast]$ and all $n \in \BN$. Since $q(1 - \zeta) \ge q(1 - \zeta_\ast) > d_\gamma$ (by assumption), this completes the proof of \Cref{lem:pi-expectation}. 
\end{proof}

\section{Proof of \texorpdfstring{\Cref{thm:main}}{Theorem~\ref{thm:main}}}\label{section:proof}

We continue to use the notation of \Cref{section:hyperbolic-filling,section:weight}. Set $\eta \defeq \alpha^{100}$. By the construction of \Cref{section:hyperbolic-filling} (especially \Cref{lem:corollary}), in order to show that the conformal dimension of the metric space $(X, D)$ is equal to two, it suffices to show that for each $p > 2$, we may choose the parameter $\alpha$, the finite subsets $A_0 \subset A_1 \subset A_2 \subset \cdots \subset X$, and the assignment $\sigma \colon V \setminus \{o\} \to \BR_{\ge 0}$ in such a way that~\eqref{eq:admissibility} is satisfied and $\sum_{u \in V_n} \pi(u)^p \to 0$ as $n \to \infty$. (We shall always implicitly assume that $A_n$ is a maximal $\alpha^n$-separated subset.)

Henceforth fix $p > 2$. Let $(\BC, \Phi; 0, \infty)$ be a $\gamma$-LQG cone. First, we consider the case where $(X, D) = (B_t(0; D_\Phi), D_\Phi)$ for some $t \in (0, 1/2)$. We \emph{claim} that 
\begin{equation}\label{eq:main-proof-0}
    \parbox{.85\linewidth}{there exists $\zeta > 0$ and $\alpha \in (0, 1)$ such that for each $t \in (0, 1/2)$, there almost surely exist finite subsets $A_0 \subset A_1 \subset A_2 \subset \cdots \subset B_t(0; D_\Phi)$ such that 
    \begin{equation*}\hspace{-.15\linewidth}
        \sum_{(x, n) \in V_n} \varpi(x, n; \Phi)^p \to 0 \quad \text{as } n \to \infty, 
    \end{equation*}
    where we use the notation $\varpi(x, n; \Phi) = \varpi(x, n)$ to emphasize the dependence of $\varpi(x, n)$ on $\Phi$.}
\end{equation}

Before proceeding, we complete the proof of \Cref{thm:main} assuming~\eqref{eq:main-proof-0}. Let $(\widehat\BC, \Phi^{\mathrm{sph}}; 0, \infty)$ be a $\gamma$-LQG sphere. By the scaling property, we may assume without loss of generality that $(\widehat\BC, \Phi^{\mathrm{sph}}; 0, \infty)$ is conditioned so that $\diam(\widehat\BC; D_{\Phi^{\mathrm{sph}}}) = 1$. Let $\{x_k\}_{k \in \BN}$ and $\{y_k\}_{k \in \BN}$ be conditionally independent samples from $\SM_{\Phi^{\mathrm{sph}}}$ (renormalized to be a probability measure). Let $\{t_k\}_{k \in \BN}$ be sampled independently and independently of everything else from the uniform probability measure on $(0, 1/2)$. Then for each $k \in \BN$, on the event that $D_{\Phi^{\mathrm{sph}}}(x_k, y_k) > 1/2$, the laws of the $\gamma$-LQG surfaces parameterized by $B_{1/2}^{y_k}(x_k; D_{\Phi^{\mathrm{sph}}})$ and $B_{1/2}^\bullet(0; D_{\Phi^{\mathrm{sph}}})$ are mutually absolutely continuous. (Indeed, the laws of the boundary lengths of the $\gamma$-LQG surfaces parameterized by $B_{1/2}^{y_k}(x_k; D_{\Phi^{\mathrm{sph}}})$ and $B_{1/2}^\bullet(0; D_{\Phi^{\mathrm{sph}}})$ are both mutually absolutely continuous with respect to the Lebesgue measure on $\BR_{\ge 0}$, and they have the same conditional law given their boundary lengths.) This, together with~\eqref{eq:main-proof-0}, implies that for each $k \in \BN$, almost surely on the event that 
\begin{itemize}
    \item $D_{\Phi^{\mathrm{sph}}}(x_k, y_k) > 1/2$, and 
    \item the restrictions of $D_{\Phi^{\mathrm{sph}}}$ and $D_{\Phi^{\mathrm{sph}}}(\bullet, \bullet; B_{1/2}^{y_k}(x_k; D_{\Phi^{\mathrm{sph}}}))$ on $B_{t_k}(x_k; D_{\Phi^{\mathrm{sph}}})$ are equal (in which case the restriction of $D_{\Phi^{\mathrm{sph}}}$ on $B_{t_k}(x_k; D_{\Phi^{\mathrm{sph}}})$ is almost surely determined by the $\gamma$-LQG surfaces parameterized by $B_{1/2}^{y_k}(x_k; D_{\Phi^{\mathrm{sph}}})$), 
\end{itemize} 
there exist finite subsets $A_0^k \subset A_1^k \subset A_2^k \subset \cdots \subset B_{t_k}(x_k; D_{\Phi^{\mathrm{sph}}})$ such that 
\begin{equation*}
    \sum_{(x, n) \in V_n^k} \varpi(x, n; \Phi^{\mathrm{sph}})^p \to 0 \quad \text{as } n \to \infty.
\end{equation*}
On the other hand, there almost surely exists a finite subset $\SK \subset \BN$ such that the above event occurs for all $k \in \SK$, and that $\widehat\BC = \bigcup_{k \in \SK} B_{t_k}(x_k; D_{\Phi^{\mathrm{sph}}})$. We may choose $A_0 \subset A_1 \subset A_2 \subset \cdots \subset \widehat\BC$ in such a way that $A_n \subset \bigcup_{k \in \SK} A_n^k$. Thus, 
\begin{align*}
    \sum_{(x, n) \in V_n} \pi(x, n; \Phi^{\mathrm{sph}})^p &\le \sum_{(x, n) \in V_n} \varpi(x, n; \Phi^{\mathrm{sph}})^p \quad \text{(by \Cref{lem:pi-upper-bound})}\\
    &\le \sum_{k \in \SK} \sum_{(x, n) \in V_n^k} \varpi(x, n; \Phi^{\mathrm{sph}})^p \to 0 \quad \text{as } n \to \infty. 
\end{align*}
Since condition~\eqref{eq:admissibility} is satisfied (cf.~\Cref{lem:admissibility}), by the discussion of the first paragraph of the present section, this completes the proof of \Cref{thm:main}. 

It remains to verify~\eqref{eq:main-proof-0}. Fix $t \in (0, 1/2)$. Let $\eta^\prime \colon (-\infty, \infty) \to \BC$ be an independent whole-plane space-filling SLE$_{\kappa^\prime}$ curve from $\infty$ to $\infty$ parameterized by $\SM_\Phi$. Fix $T > 0$. Let $\{x_k\}_{k \in \BN}$ be conditionally independent samples from $\SM_\Phi|_{\eta^\prime([-T, T])}$ (renormalized to be a probability measure). On the event that $B_t(0; D_\Phi) \subset \eta^\prime([-T, T])$, we may arrange that $A_n \subset \{x_k : k \in [1, k_n]_\BZ, \ x_k \in B_t(0; D_\Phi)\}$ for all $n \in \BN$, where 
\begin{equation*}
    k_n \defeq \inf\!\left\{K \in \BN : B_t(0; D_\Phi) \subset \bigcup\{B_{\alpha^n}(x_k; D_\Phi) : k \in [1, K]_\BZ, \ x_k \in B_t(0; D_\Phi)\}\right\}.
\end{equation*}
By \Cref{lem:net}, almost surely on the event that $B_t(0; D_\Phi) \subset \eta^\prime([-T, T])$, we have $k_n \le \alpha^{-(d_\gamma + \zeta)n}$ for all sufficiently large $n \in \BN$. Fix $n_\ast \in \BN$. Write $E(T, n_\ast)$ for the event that $B_t(0; D_\Phi) \subset \eta^\prime([-T, T])$ and $k_n \le \alpha^{-(d_\gamma + \zeta)n} \text{ for all } n \ge n_\ast$. Note that each $(\BC, \Phi; x_k, \infty)$ has the same law as $(\BC, \Phi; 0, \infty)$ (as $\gamma$-LQG surfaces). Thus, 
\begin{align*}
    \BE\!\left\lbrack\left(\sum_{(x, n) \in V_n} \varpi(x, n; \Phi)^p\right)\mathbf 1_{E(T, n_\ast)}\right\rbrack &\le \BE\!\left\lbrack\left(\sum_{k \in [1, k_n]_\BZ} \varpi(x_k, n; \Phi)^p\right)\mathbf 1_{E(T, n_\ast)}\right\rbrack \\
    &\le \alpha^{-(d_\gamma + \zeta)n} \BE\lbrack\varpi(0, n; \Phi)^p\rbrack 
\end{align*}
for all $n \ge n_\ast$. Thus, by \Cref{lem:pi-expectation}, we may choose $\zeta$ to be sufficiently small (specifically, $\zeta \le \zeta_\ast \wedge ((q - d_\gamma)/2)$ in the notation of that lemma) to ensure the existence of some $\alpha \in (0, 1)$ such that
\begin{equation*}
    \BE\!\left\lbrack\left(\sum_{(x, n) \in V_n} \varpi(x, n; \Phi)^p\right)\mathbf 1_{E(T, n_\ast)}\right\rbrack \le \alpha^{\zeta n}, \quad \forall n \ge n_\ast. 
\end{equation*}
(Here, we note that $\zeta$ and $\alpha$ do not depend on the choice of $t$, $T$, and $n_\ast$.) By the Borel--Cantelli lemma, this implies that, almost surely on the event $E(T, n_\ast)$, 
\begin{equation*}
    \sum_{(x, n) \in V_n} \varpi(x, n; \Phi)^p \to 0 \quad \text{as } n \to \infty.
\end{equation*}
Since there almost surely exists $T > 0$ and $n_\ast \in \BN$ such that $E(T, n_\ast)$ occurs, this completes the proof of~\eqref{eq:main-proof-0}, hence also the proof of \Cref{thm:main}. \qed

\bibliographystyle{alpha}
\bibliography{references}

\end{document}